\documentclass[12pt]{amsart}
\usepackage[mathcal]{euscript}
\usepackage{mathrsfs, amssymb}
\usepackage{hyperref}
\usepackage{array}
\usepackage{tabularx,multirow,makecell,longtable}
\usepackage{mathcomp}
\usepackage{upgreek}
\usepackage[matrix,arrow,curve,all, frame]{xy}
\usepackage{amsthm}
\usepackage[dvipsnames]{xcolor}
\usepackage{graphicx}

\newcommand{\n}{\mathrm{n}}

\newcommand{\CC}{\mathbb{C}}
\newcommand{\ZZ}{\mathbb{Z}}
\newcommand{\PP}{\mathbb{P}}
\newcommand{\QQ}{\mathbb{Q}}

\newcommand{\Cf}{\mathfrak{C}}
\newcommand{\Xf}{\mathfrak{X}}
\newcommand{\Zf}{\mathfrak{Z}}

\newcommand{\fb}{\mathfrak{f}}

\newcommand{\I}{{\mathcal{I}}}

\newcommand{\EEE}{\mathscr{E}}
\newcommand{\JJJ}{{\mathscr{J}}}
\newcommand{\III}{{\mathscr{I}}}

\newcommand{\OOO}{{\mathscr{O}}}

\newcommand{\NNN}{{\mathscr{N}}}

\newcommand{\CCC}{\mathscr{C}}
\newcommand{\RRR}{\mathscr{R}}
\newcommand{\ord}{\operatorname{ord}}
\newcommand{\wt}{\operatorname{wt}}
\newcommand{\rk}{\operatorname{rk}}
\newcommand{\ow}{\operatorname{ow}}
\newcommand{\aw}{\operatorname{aw}}

\newcommand{\Sing}{\operatorname{Sing}}
\newcommand{\Def}{\operatorname{Def}}
\newcommand{\Cl}{\operatorname{Cl}}
\newcommand{\id}{\operatorname{id}}
\newcommand{\Exc}{\operatorname{Exc}}

\newcommand{\Proj}{\operatorname{Proj}}
\newcommand{\tors}{\mathrm{tors}}
\newcommand{\Pic}{\operatorname{Pic}}

\newcommand{\gr}{\operatorname{gr}}

\newcommand{\Hom}{\operatorname{Hom}}
\newcommand{\HHom}{\mathscr{H}{om}}
\newcommand{\red}{\operatorname{red}}
\newcommand{\Clsc}{\operatorname{Cl^{\operatorname{sc}}}}
\newcommand{\Spec}{\operatorname{Spec}}

\newcommand{\mult}{{\operatorname{mult}}}

\newcommand{\Supp}{{\operatorname{Supp}}}

\newcommand{\Res}{{\operatorname{Res}}}

\newcommand{\Coker}{\operatorname{Coker}}

\newcommand{\p}{{{p}_{\operatorname{a}}}}

\newcommand{\typec}[1]{$\mathrm{(#1)}$}
\newcommand{\type}[1]{$\mathrm{#1}$}
\newcommand{\mumu}{{\boldsymbol{\mu}}}
\newcommand{\ii}{\operatorname{i}}

\renewcommand{\emptyset}{\varnothing}
\newcommand{\comp}\circ
\newcommand{\len}{\operatorname{len}}

\newcommand{\lin}{\text{---}}
\newcommand{\xref}[1]{\textup{\ref{#1}}}

\makeatletter
\@addtoreset{equation}{subsection}
\makeatother

\renewcommand{\theequation}{\arabic{section}.\arabic{subsection}.\arabic{equation}}
\renewcommand{\thesubsection}{\arabic{section}.\arabic{subsection}}

\newcommand{\nom}{\refstepcounter{equation}{\rm{\theequation}}}

\swapnumbers
\theoremstyle{plain}
\newtheorem{theorem}[subsection]{Theorem}
\newtheorem{lemma}[subsection]{Lemma}

\newtheorem{proposition}[subsection]{Proposition}

\newtheorem{stheorem}[equation]{Theorem}

\newtheorem{corollary}[subsection]{Corollary}
\newtheorem{scorollary}[equation]{Corollary}
\newtheorem*{claim*}{Claim}
\newtheorem{sclaim}[equation]{Claim}
\newtheorem{slemma}[equation]{Lemma}
\newtheorem{sproposition}[equation]{Proposition}

\newtheorem{emptytheorem}[equation]{}

\theoremstyle{definition}
\newtheorem{definition}[subsection]{Definition}

\newtheorem*{definition*}{Definition}
\newtheorem{sdefinition}[equation]{Definition}
\newtheorem{example-remark}[subsection]{Remark-Example}
\newtheorem{subexample-remark}[equation]{Remark-Example}

\newtheorem{scase}[equation]{}

\newtheorem{setup}[subsection]{Setup}
\newtheorem*{notation*}{Notation}
\newtheorem{example}[subsection]{Example}

\newtheorem{sexample}[equation]{Example}
\newtheorem{sexamples}[equation]{Examples}

\newtheorem{sremark}[equation]{Remark}
\newtheorem{construction}[subsection]{Construction}
\newtheorem{sconstruction}[equation]{Construction}

\title{Threefold extremal curve germs 
\\
with one non-Gorenstein point}
\author{Shigefumi Mori}

\address{
Shigefumi~Mori:
Kyoto University Institute for Advanced Study,
Kyoto University, Kyoto, Japan
\newline\indent
Research Institute for Mathematical Sciences,
Kyoto University, Kyoto, Japan
}
\email{mori@kurims.kyoto-u.ac.jp}

\author{Yuri Prokhorov}

\address{Yuri~Prokhorov:
Steklov Mathematical Institute of Russian Academy of Sciences, Moscow, Russia
\newline\indent
Department of Algebra, 
Moscow State Lomonosov University
\newline\indent
National Research University Higher School of Economics
}
\email{prokhoro@mi.ras.ru}

\thanks{
This work was supported by the Research Institute for Mathematical
Sciences, a Joint Usage/Research Center located in Kyoto University.
The first author was supported by the JSPS Grant Numbers JP25287005 and JP17H06127. The second author was 
supported by the Program of the Presidium of the Russian Academy of Sciences No.~01 ``Fundamental Mathematics and
its Applications'' under grant \mbox{PRAS-18-01} and the Russian Academic Excellence Project ``5-100''. 
}
\date{}

\begin{document}
\begin{abstract}
 An extremal curve germ is the analytic germ of a threefold with terminal singularities
along a reduced complete curve admitting a contraction 
whose fibers have dimension at most one.  
The aim of the present paper 
is to review the results concerning those contractions whose  central fiber is irreducible 
and contains  only one non-Gorenstein point. 
\end{abstract}
\maketitle
\tableofcontents

\section{Introduction}
One of the most important problems in the three-dimensional birational 
geometry is to describe explicitly all the steps of the Minimal Model Program (MMP).
These steps consist of certain maps, called divisorial, flipping, and fiber-type contractions (Mori contractions).
The structure of these maps is still unknown in its complete generality, though 
much progress has been made in this direction. 
We refer to \cite{CKM} for an introduction to the subject.
The aim of the present paper 
is to review the results concerning those contractions whose fibers have dimensioв специальном случаеn 
at most one. The project was started in the initial paper \cite{Mori:flip} where 
the minimal model problem was solved in the three-dimensional case.
To study Mori contractions in this situation one needs to work in the analytic 
category and analytic counterparts of the corresponding notions are needed.
The central objects of this paper are so-called \emph{extremal curve germs}.

An extremal curve germ is the analytic germ of a threefold with terminal singularities
along a reduced complete curve admitting a contraction 
whose fibers have dimension at most one. 
The present paper is a survey of known results on the classification of objects of this type.
Basically we concentrate on the case of irreducible central fiber with only
one non-Gorenstein point. In this case the results are complete, however 
they are scattered in the literature. This is the main reason to write 
this survey.

The classification of extremal curve germs is done in terms of a general element 
$H$ of the linear system $|\OOO_X|$ of 
trivial Cartier divisors containing $C$. In many cases this element $H$ is a 
normal surface
and then the threefold can be viewed as a one-parameter deformation of $H$.

A birational extremal curve germ $f: (X, C)\to (Z,o)$ is said to be 
\emph{semistable}, if for a general member $D\in |-K_Z|$, the germ 
$D_Z:=\Spec_{Z}f_*\OOO_D$ is Du Val of type~\type{A} \cite{KM92}. The 
semistable case is subdivided into two cases \typec{k1A} and \typec{k2A} according to the 
number of non-Gorenstein points of $X$ on $C$. Other cases are called 
\emph{exceptional}. It turns out that treating semistable and exceptional germs 
uses different approaches. For example, in the exceptional flipping case, 
\cite{KM92} provides relatively simple computations of flipped variety. For 
semistable germs these computations become more explicit; in the \typec{k2A} 
case, from the general member $H\in |\OOO_X|$ one can decide whether $(X,C)$ is 
flipping or divisorial \cite[Corollary 4.1]{Mori:ss} and furthermore describe 
the flipped variety \cite[Theorem 4.7]{Mori:ss} and $Z$ \cite[Theorem 
4.5]{Mori:ss}, respectively; the \typec{k1A} case is similarly treated by 
\cite{HTU} under additional assumption ``$b_2(X_s)=1$'' (see \ref{HTU}). According to 
local classification (see Propositions \xref{prop:local-primitive} and 
\xref{prop-imp-types}), a semistable extremal curve germ of type \typec{k1A} can be 
of type \typec{IA^\vee} or \typec{IA}. They are treated in Sect. 
\xref{sect:index2}, \xref{sect:imprimitive}, and \xref{sect:IA}. 

Here are summary of some of the results.

\begin{theorem}
Let $f:(X,C)\to (Z,o)$ be a flipping extremal curve germ
with irreducible central fiber $C$
and let $H_Z\in |\OOO_Z|$ be a general hyperplane section containing $C$.
Then $H_Z$ and $f^{-1}(H_Z)$ are normal and have rational singularities. The 
singularity $(H_Z,o)$ 
is log terminal except for the cases described in \xref{cD/3:flip:.3.2)}
and \xref{IIA:flip:iP=2}. Moreover, $(H_Z,o)$ is a cyclic quotient singularity 
if and only if $(X,C)$ is semistable.
\end{theorem}

The case where $(X,C)$ is semistable follows from Lemma \xref{lemma:lc}. 
If $(X,C)$ has only one non-Gorenstein point (i.e. of type \typec{k1A}), we can use 
explicit classification \xref{index2flipping}, \xref{imprimitiveIA},
\xref{IC-main}, \xref{theorem-main-birational}, \xref{cD/3:thm)}, \xref{IIA:thm}.
For the remaining \typec{kAD} case we refer to \cite[\S 9]{KM92}.

\begin{theorem}
Let $f:(X,C)\to (Z,o)$ be a divisorial extremal curve germ
with irreducible central fiber $C$
and let $H_Z\in |\OOO_Z|$ be a general hyperplane section containing $o$.
Then $(H_Z,o)$ is either a Du Val point,
a rational log canonical point of type \type{\tilde D} \textup(in the case \xref{imprimitiveII}\textup),
or a cyclic quotient singularity of class \type{T}. 
Moreover, the last two possibilities occurs only if $(X,C)$ has a locally imprimitive point
or $(X,C)$ is semistable and has two non-Gorenstein points whose indices are not coprime.
\end{theorem}
Moreover, by Theorem \xref{thm:div:Q-Cartier} the singularity
$(Z,o)$ is terminal. If $(X,C)$ has no covers \'etale in codimension one, then 
$(Z,o)$ is of index one and $H_Z$ is Du Val. In the semistable case 
the assertion as above follows from Lemma \xref{lemma:lc}. 
It remains to consider locally imprimitive cases \typec{IA^\vee} and \typec{II^\vee}
(see \xref{imprimitiveII} and \xref{imprimitiveIA}). Note that for a divisorial curve germ 
the surface $H:=f^{-1}(H_Z)$ can be non-normal (see e.g. Example~\ref{ex:IIA-n-normal}).

For the $\QQ$-conic bundle $f:(X,C)\to (Z,o)$ 
we can show that the base is Du Val of type \type{A}
(Corollary \xref{base}).
The proof uses the existence of a Du Val member $D\in |-K_X|$, see 
\cite{MP:cb1}, \cite{MP:cb3}, and Theorem \xref{thm:ge}.

The paper was written during the second author's stay at RIMS, Kyoto University. 
The author is very grateful to the institute for 
the invitation, hospitality and good working environment. 

\section{Preliminaries} 
\subsection{Threefold terminal singularities}
\label{3terminal}
Recall that a three-dimensional terminal singularity of index $m$ is a 
quotient of an isolated hypersurface singularity by a cyclic group $\mumu_m$ of 
order $m$.
More precisely, let $(X, P)$ be an analytic germ of a three-dimensional
terminal singularity of index $m$. Then there exists a terminal singularity 
$(X^\sharp, P^\sharp)$ of index $1$ and a cyclic $\mumu_m$-cover 
\begin{equation*}
(X^\sharp, P^\sharp) \longrightarrow (X, P)
\end{equation*}
which is \'etale outside $P$ \cite{Reid:Pagoda}. 
Moreover, the singularity $(X^\sharp, P^\sharp)$ can be embedded to $(\CC^4, 
0)$ 
so that its general hyperplane section is a surface Du Val singularity
(thus $(X^\sharp, P^\sharp)$ is so-called \type{cDV} singularity).
A detailed classification of all possibilities for equations of 
$X^\sharp\subset \CC^4$ and the action of $\mumu_m$ was obtained in
\cite{Mori:sing} (see also \cite{Reid:YPG}, \cite{KSh88}). 

Assume that $m > 1$. Then the $\mumu_m$-action on 
$(X^\sharp,P^\sharp)$ will be analyzed. We 
fix a character $\chi$ generating $\Hom(\mumu_m, \CC^*) =\ZZ/m\ZZ$.
For a $\mumu_m$-semi-invariant $z$, we 
write  
\[
\wt(z)\equiv a \mod m
\]
if $g(z) = \chi(g)^a z$ for all $g\in 
\mumu_m$. 

\begin{stheorem}[\cite{Mori:sing}]
\label{clasifiction-terminal}
In the above notation the singularity $(X^\sharp, P^\sharp)$ is
$\mumu_m$-isomorphic to a hypersurface $\phi = 0$ in $(\CC^4_{x_1,\dots,x_4}, 
0)$ such that
for some $a, b\in \ZZ$ prime to $m$ one of the following holds:
\begin{enumerate}
\item \label{classification-singularities-m}
$\wt(x, \phi) \equiv (a, b, - a, 0, 0) \mod m$;
\item \label{classification-singularities-cAx/4}
$m = 4$, and $\wt(x,\phi) \equiv (a, b, - a, 2, 2) \mod m$.
\end{enumerate}
In the case \xref{classification-singularities-cAx/4} 
we say that $(X,P)$ is a point of type \type{cAx/4}.
\end{stheorem}
Thus the locus $\Upsilon \subset\CC^4$ of the points at which $\mumu_m$-action is not free
is a coordinate axis which is not contained in $X^\sharp$.
The following number 
\begin{equation*}
\aw(X,P):=\mult_0(\phi|_{\Upsilon})
\end{equation*} 
is well defined 
and called the \emph{axial multiplicity} of $(X,P)$.

\subsection{}
Recall that a \emph{contraction} is a proper surjective morphism $f:X\to Z$ of 
normal varieties such that $f_*\OOO_X=\OOO_Z$. 

\begin{sdefinition}
Let $(X,C)$ be the analytic germ 
of a threefold with terminal singularities along a reduced complete curve. We 
say that $(X,C)$ is an \emph{extremal curve germ} if there is a contraction 
\[
f: (X,C)\to (Z,o) 
\]
such that $C=f^{-1}(o)_{\red}$ and $-K_X$ is $f$-ample. 
Furthermore, $f$ is called \emph{flipping} if its exceptional locus coincides 
with $C$ and \emph{divisorial} if its exceptional locus is two-dimensional. If 
$f$ is not birational, then $Z$ is a surface and $(X,C)$ is said to be a 
\emph{$\QQ$-conic bundle germ} \cite{MP:cb1}.
\end{sdefinition}

In general, we do not assume that $X$ is $\QQ$-factorial.
This is because the $\QQ$-factoriality is not a local condition 
in the analytic category (see \cite[\S 1]{Kaw:Crep}).

For future references we need the following easy example.
\begin{example}
\label{ex-toric}
Consider the following action of $\mumu_m$ on $\PP^1_x\times
\CC^2_{u,v}$:
\begin{equation*}
(x;u,v) \longmapsto(\varepsilon^a x; \varepsilon u,
\varepsilon^{-1} v),
\end{equation*} 
where $\varepsilon$ is a primitive $m$-th root of unity and $\gcd
(m,a)=1$. Let $X:=\PP^1\times\CC^2/\mumu_m$,
$Z:=\CC^2/\mumu_m$ and let $f: X\to Z$ be the natural
projection. Since $\mumu_m$ acts freely in codimension one, $-K_X$ is
$f$-ample. The images of two fixed points on $\PP^1\times \CC^2$ are terminal cyclic
quotient singularities of types $\frac1m(\pm a,1,-1)$ on $X$. Hence, $f$ is a $\QQ$-conic bundle. 
A $\QQ$-conic bundle germ 
biholomorphic to $f$ as above is called
\emph{toroidal}.
\end{example}

The following key fact is an immediate consequence of the Kawamata-Viehweg
vanishing theorem.
\begin{theorem}
\label{th-vanish}
Let $f: (X,C)\to (Z,o)$ be an extremal curve germ. Then
$R^if_*\OOO_X=0$ for $i>0$.
\end{theorem}

\begin{scorollary}[cf. {\cite[Remark 1.2.1, Cor. 1.3]{Mori:flip}}]
\label{cor-C-pa=0}
\begin{enumerate}
\item\label{cor-C-pa=0a}
If $\JJJ$ is an ideal such that $\Supp(\OOO_X/\JJJ)\subset C$, then
$H^1(\OOO_X/\JJJ)=0$.
\item\label{cor-C-pa=0b}
$\p(C)=0$ and $C$ is a union of smooth rational curves.
\item\label{cor-C-pa=0c}
$\Pic X\simeq H^2(C,\ZZ)\simeq \ZZ^\uprho$, where $\uprho$ is the number
of irreducible components of $C$.
\end{enumerate}
\end{scorollary}

\begin{sremark}
\label{rem-prel-extr-nbd}
If $C$ is reducible, then for every proper curve
$C'\subsetneq C$, the germ $(X,C')$ is also an extremal curve germ.
\end{sremark}

\begin{lemma}\label{lemma:base}
Let $f: (X,C)\to (Z,o)$ be an extremal curve germ.
\begin{enumerate}
\item \label{lemma:base-1}
If $f$ is birational, then on $Z$ there exists an effective $\QQ$-divisor 
$B$ such that the pair $(Z,B)$ has only canonical singularity at $o$.
If moreover $f$ is flipping, then the singularity of $(Z,B)$ at $o$ is terminal. 
\item \label{lemma:base-2}
If $f$ is a $\QQ$-conic bundle, then $Z$ has a log terminal singularity at $o$. 
\end{enumerate} 
\end{lemma}
\begin{proof}
Take $n\gg 0$ so that the divisor $nK_X$ is Cartier and the linear 
system $|-nK_X|$ is base point free. Let $H\in |-nK_X|$ be a general member.
Then $H$ is a smooth surface meeting the components of $C$ transversally.
For \xref{lemma:base-1}, put $D:=\frac 1n H$ and $B:=f_*D$.
Then the singularities of the pair $(X,D)$ are terminal.
Since $f$ is crepant with respect to $K_X+D$ and does not 
contract components of $D$, we see that
the singularities of $(Z,B)$ are canonical \cite[Lemma 3.38]{KM:book}. 
To show \xref{lemma:base-2} we note that 
the restriction $f_H: H\to Z$ is a finite morphism.
Thus $(Z, o)$ is a log terminal singularity \cite[Prop.~5.20]{KM:book}. 
\end{proof}
Note however that in \xref{lemma:base-1} we do not assert that the 
point $(Z,o)$ is $\QQ$-Gorenstein, even in the divisorial case, 
see Theorem~\xref{thm:div:Q-Cartier}.
The result of \xref{lemma:base-2} is significantly improved in \xref{corollary:cyclic}
and \cite[1.2.7]{MP:cb1}.

\subsection{General member of $|\OOO_X|$}
\label{H}
Let $f: (X,C)\to (Z,o)$ be an extremal curve germ. If $f$ is a $\QQ$-conic bundle, then we assume that $(Z,o)$ is smooth. We denote by $|\OOO_Z|$ the linear system of Cartier divisors (hyperplane sections) passing through $o$ and $|\OOO_X|:=f^*|\OOO_Z|$. Let $H$ be a general member of $|\OOO_X|$ and let $H_Z=f(H)$. Let $H^{\n}\to H$ be the normalization (we put $H^{\n}= H$ if $H$ is normal). By \cite[5.25]{KM:book} and Lemma \xref{lemma:base} both $H_Z$ and $H$ are Cohen-Macaulay. Hence by Bertini's theorem $H_Z$ is normal. Then the composition map $H^{\n}\to H_Z$ has connected fibers. Moreover, it is a rational curve fibration if $\dim Z=2$; a birational contraction to a point $(H_Z, o)$ if $f$ is birational. Thus in the $\QQ$-conic bundle case $H^{\n}$ has only rational singularities. The same is true in the birational case if the singularity $(H_Z,o)$ is rational.

\subsection{Notation on dual graphs}
Let $S$ be a normal surface and let $C\subset S$ be a curve. Suppose that on the minimal resolution of $S$ the exceptional divisors and the proper transform of $C$ form a normal crossing divisor, say $R$. We use the usual notation of dual graphs $\Delta (S,C)$ of $R$: each $\diamond$ corresponds to an irreducible component of $C$ and each $\circ$ corresponds to an exceptional divisor, and we may use $\bullet$ instead of $\diamond$ if we want to emphasize that it is a complete $(-1)$-curve. A number attached to a vertex denotes the minus self-intersection number. For short, we may omit $2$ if the self-intersection equals $-2$.

\begin{sproposition}[{\cite{Cutkosky-1988}}]
\label{prop:Gor}
\begin{enumerate}
\item \label{prop:Gor-cb}
Let $f: X\to Z$ be a $\QQ$-conic bundle. If $X$ is Gorenstein
\textup(and terminal\textup), then $Z$ is smooth and there is a
vector bundle $\EEE$ of rank $3$ on $Z$ and an embedding
$X\hookrightarrow \PP(\EEE)$ such that every scheme fiber $X_z$,
$z\in Z$ is a conic in $\PP(\EEE)_z$.

\item \textup(see also \cite[4.7.2]{KM92}\textup)
\label{prop:Gor-bir}
Let $f: (X,C)\to (Z,o)$ be a birational curve germ such that $X$ is Gorenstein.
Then $f$ is divisorial, $(Z,o)$ is smooth, $C$ is irreducible, and 
$f$ is the blowup of a curve $B\subset Z$ 
having only planar singularities. Moreover, $X$ has exactly one singular point 
which is of type~\type{cA} and for a general member $H\in |\OOO_X|$ 
the graph $\Delta(H,C)$ has the form
\begin{equation*} 
\bullet\lin\underbrace{\circ\lin\cdots\lin\circ}_m
\end{equation*}
\end{enumerate}
\end{sproposition}

The following fact is a particular case of \cite[Theorem~4.9]{KM92}.
\begin{theorem}
Let $f:(X,C)\to (Z,o)$ be a divisorial extremal curve germ.
Let $E$ be its exceptional locus \textup(with reduced structure\textup) and let $B:=f(E)_{\red}$.
Assume that $K_Z$ is $\QQ$-Cartier \textup(this automatically 
holds if $C$ is irreducible, see \xref{thm:div:Q-Cartier}\textup).
Then the following holds.
\begin{enumerate}
\item 
The set $E$ of $f$ is purely two-dimensional and 
is a $\QQ$-Cartier divisor, and the singularity $(Z,o)$ is terminal.
\item 
The variety $X$ is the symbolic blowup of $B$, that is, 
\[
X=\Proj_Z \bigoplus_{m=0}^{\infty} \I_B^{(m)}, 
\]
where $\I_B$ is the ideal sheaf of $B$, 
and $\I_B^{(m)}$ denotes its symbolic power. In particular, $X$ is uniquely determined by
$B\subset Z$.
\end{enumerate}
\end{theorem}
It is possible to study divisorial curve germs algebraically,
by scrupulous analysis of the curve $B$ and its embedding $B\subset X$
(see \cite{Tzi:03}, \cite{Tzi:05D}, \cite{Kaw:div}, 
\cite{Tzi:10}, \cite[\S~6.1]{Prokhorov-Reid}, \cite{Ducat:16}). This method 
is completely different from our approach.

\section{Basic techniques} 
\subsection{}
Let $\I_C\subset \OOO_X$ be the
ideal sheaf of $C$ and let $\I_C^{(n)}$ be its symbolic $n$th power, that is, the saturation of $\I_C^n$ in $\OOO_X$. Put
\begin{equation*}
\gr_C^n\OOO:=\I_C^{(n)}/\I_C^{(n+1)}. 
\end{equation*} 
Further, let $F^n\upomega_X$ be
the saturation of $\I_C^n\upomega_X$ in $\upomega_X$ and let
\begin{equation*}
\gr_C^n\upomega:=F^n\upomega_X/F^{n+1}\upomega_X. 
\end{equation*} 
Let $m$ be the index
of $K_X$. We have natural homomorphisms
\begin{equation*}
\begin{array}{lllll}
\alpha_1 &:& \bigwedge^2 \gr_C^1\OOO &\longrightarrow&
\HHom_{\OOO_C}(\Omega_C^1,\gr_C^0\upomega),
\\[7pt]
\beta_0&:& (\gr_C^0\upomega)^{\otimes m} &\longrightarrow&
(\upomega_X^{\otimes m})^{\vee\vee}\otimes \OOO_C.
\end{array}
\end{equation*} 
Denote
\begin{equation}
\label{equation:iP-wP}
i_P(1):=\len_P\Coker (\alpha_1),\qquad w_P(0):=\len_P\Coker (\beta_0)/m.
\end{equation}
To study extremal germs more carefully, 
Mori \cite{Mori:flip} introduced also series of local invariants $i_P(n)$, 
$w_P(n)$, $w^*_P(n)$ similar to $i_P(1)$ and $w_P(0)$. 
We do not define them here.

Assume that $C\simeq \PP^1$. Then we have by \cite[2.3.1]{Mori:flip}
\begin{eqnarray}
\label{eq-grw-w}
&&
-\deg \gr_C^0\upomega=-K_X\cdot C+\sum_P w_P(0),
\\
\label{eq-grO-iP1}
&&
2+\deg \gr_C^0\upomega-\deg \gr_C^1\OOO=\sum_P i_P(1).
\end{eqnarray}
Since $\rk \gr_C^1\OOO=2$, taking \xref{th-vanish} into account we obtain 
\begin{equation}
\label{eq-grO-iP1-1}
\deg \gr_C^1\OOO\ge -2,
\end{equation}
\begin{equation}
\label{eq-grO-iP1-2}
4\ge -\deg \gr_C^0\upomega+\sum_P i_P(1)= -K_X\cdot C+\sum_P
w_P(0)+\sum_P i_P(1).
\end{equation}

\begin{sremark}
\label{remark:grw}
In the case where $f$ is birational, by the Grauert-Riemenshneider
vanishing, one has $\gr_C^0\upomega=\OOO_{C}(-1)$ (see
\cite[2.3]{Mori:flip}). This is no longer true for $\QQ$-conic
bundles: in the toroidal example \xref{ex-toric} easy computations 
show 
$\deg \gr_C^0\upomega= -2$ (see \eqref {eq-grw-w}). Similarly, in the case
\xref{item-main-th-impr-barm=1} we also have $\deg \gr_C^0\upomega=
-2$.
We will show below that these two examples are the only
exceptions (see Corollaries \xref{cor-prop-grw=2-2-points-prim} and
\xref{corollary:gr-w}).
\end{sremark}

\subsection{}
Let $(X,P)$ be a germ of threefold terminal singularity.
Throughout this paper $(X^\sharp, P^\sharp)\to (X,P)$
denotes the index-one cover. For any object $V$ on $X$
we denote by $V^\sharp$ the pull-back of $V$ on $X^\sharp$.

\begin{lemma}[{\cite[2.16]{Mori:flip}}]
\label{equation-iP}

In the above notation, assume that $C^\sharp$ is smooth. Denote
\begin{equation*}
\ell(P):=\len_P \I_C^{\sharp (2)}/\I_C^{\sharp 2},
\end{equation*}
where $\I_C^\sharp$ is the ideal of $C^\sharp$ in $X^\sharp$.
Then 
\begin{equation}\label{equation-iP-lP}
i_P(1)=
\begin{cases}
\ell(P)&\text{if $m=1$},
\\
\lfloor(\ell(P)+6)/4\rfloor&\text{if $(X,P)$ is of type~\type{cAx/4}},
\\
\lfloor\ell(P)/m\rfloor+1&\text{if $(X,P)$ is not as above}.
\end{cases}
\end{equation}
\end{lemma}

\begin{lemma}[{\cite[2.10, 2.15]{Mori:flip}}]
\label{lemma:iP-wP}
If $(X,P)$ is singular, then $i_P(1)\ge 1$.
If $(X,P)$ is not Gorenstein, then $w_P(0)>0$.
\end{lemma}
Then from \eqref{eq-grO-iP1-2} we obtain
\begin{scorollary}
An extremal curve germ $(X,C\simeq\PP^1)$ has at most three
singular points.
\end{scorollary}

\subsection{}
Let $(X,C)$ be an extremal curve germ.
By Lemma \xref{cor-C-pa=0}\xref{cor-C-pa=0a} we have
$H^1 (\gr^1_ C\OOO) = 0$.
From the 
standard exact sequence 
\begin{equation*}
0\xrightarrow{\hspace*{20pt}} \I_C^{(n+1)} \xrightarrow{\hspace*{20pt}} \I_C^{(n)} \xrightarrow{\hspace*{20pt}} \gr_C^n\OOO\xrightarrow{\hspace*{20pt}} 0.
\end{equation*}
we obtain the following easy but useful fact.
\begin{slemma}\label{lemma-grC}
The following assertions hold.
\begin{enumerate}
\item \label{lemma-grC-1}
If $H^1\bigl(\gr_C^n\OOO\bigr)=0$ and the 
map $H^0\bigl(\I_C^{(n)}\bigr)\to H^0\bigl(\gr_C^n\OOO\bigr)$ is surjective, then 
$H^1\bigl(\I_C^{(n+1)}\bigr)\simeq H^1\bigl(\I_C^{(n)}\bigr)$. In particular, $H^1(I)=0$ from the case $n=0$.

\item \label{lemma-grC-2}
If for all $i<n$ one has $H^1(\gr_C^i\OOO)=0$ and the 
map $H^0(\I_C^{(i)})\to H^0(\gr_C^i\OOO)$ is surjective, then 
$H^1(\I_C^{(n)})\simeq H^1(\gr_C^n\OOO)=0$.
\item \label{lemma-grC-3}
If 
$H^0(\gr_C^1\OOO)=0$, then
$H^1(\I_C^{(2)})= H^1(\gr_C^2\OOO)=0$.
\end{enumerate}
In particular, if a general member $H\in |\OOO_X|$ is normal, then 
$H^0(\gr_C^1\OOO)\neq0$. 
\end{slemma}

Note however that this is necessary but not sufficient condition 
for normality of $H$ \cite{MP:IA}.

\subsection{Sheaves $\gr_C^n\upomega$}
\begin{slemma}
\label{lemma-omega-main}
Let $f: (X,C) \to (Z,o)$ be an extremal curve germ. 
\begin{enumerate}
\item \textup(\cite[1.2]{Mori:flip}\textup)\label{lemma-omega-main-1}
If $f$ is birational, then $R^if_*\upomega_X=0$ for $i>0$. 
\item \textup(\cite[Lemma 4.1]{MP:cb1}\textup)\label{lemma-omega-main-2}
If $f$ is a $\QQ$-conic bundle 
and $Z$ is smooth, then there is a canonical 
isomorphism
$R^1f_*\upomega_X\simeq \upomega_Z$. 
\end{enumerate}
\end{slemma}

\begin{proof}
\xref{lemma-omega-main-1} follows from the Grauert-Riemenshneider vanishing.
Let us prove \xref{lemma-omega-main-2}.
Let $g: W\to X$ be a resolution. By {\cite[Prop.
7.6]{Kollar-1986-I}} we have $R^1(f \comp g)_* \upomega_W=\upomega_Z$.
Since $X$ has only terminal singularities, $g_*\upomega_W=\upomega_X$
and by the Grauert-Riemenshneider vanishing, $R^ig_* \upomega_W=0$ for
$i>0$. Then the Leray spectral sequence gives us $R^1f_*
\upomega_X=R^1(f \comp g)_* \upomega_W= \upomega_Z$.
\end{proof}

We also have the following useful fact 

\begin{corollary}
\label{cor-gr-omega-=0}
Let $f: (X,C\simeq \PP^1) \to (Z,o)$ be an extremal curve germ. 
\begin{enumerate}
\item \label{cor-gr-omega-=0a}
If $f$ is birational, then $\deg \gr_C^0\upomega=-1$.
\item\label{cor-gr-omega-=0b} 
Assume that $f$ is a $\QQ$-conic bundle with smooth base. If 
$\deg \gr_C^0\upomega\neq-1$, then $f^{-1}(o)=C$ \textup(as a scheme\textup).
\end{enumerate}
\end{corollary}
\begin{proof}[Sketch of the proof]
For \xref{cor-gr-omega-=0a} we note that by \xref{lemma-omega-main}\xref{lemma-omega-main-1}
for an arbitrary ideal $\JJJ$ such that $\Supp (\OOO_X/\JJJ)\subset C$
we have $H^1 (\upomega_X/\JJJ\upomega_X)) = 0$. Hence,
$H^1(\gr_C^0\upomega)=0$ in this case. On the other hand, $\deg \gr_C^0\upomega<0$ by 
\eqref{eq-grw-w}.
For \xref{cor-gr-omega-=0b} we 
apply \cite[Theorem~4.4]{MP:cb1} with $J=\I_C$.
\end{proof}

\begin{slemma}[{\cite[Cor.~1.15]{Mori:flip}}, {\cite[Prop. 4.2]{Kollar-1999-R}}, {\cite[Lemma 4.4.2]{MP:cb1}}]
\label{lemma-int-non-Gor}
Let $f: (X,C) \to (Z,o)$ be an extremal curve germ.
Suppose that $C$ is reducible and let $P$ be a singular 
point of $C$. 
If $X$ is Gorenstein at $P$, then 
$f$ is a $\QQ$-conic bundle and  $C$ has two components meeting at $P$.
If moreover $(Z,o)$ is smooth, then $X$ is Gorenstein
\textup(see \xref{prop:Gor}\xref{prop:Gor-cb}\textup). 
\end{slemma} 
We will show below in \ref{cor:int-non-Gor} that in the above assumptions $(Z,o)$ is smooth
automatically.

\begin{proof}
By Corollary \ref{cor-C-pa=0}
there are at least two components, say $C_1,\, C_2\subset C$ passing through $P$.
Replacing $(X,C)$ with $(X,C_1\cup C_2)$ we may assume that $C=C_1\cup C_2$
(see Remark \xref{rem-prel-extr-nbd}). Since the point $P\in X$ is Gorenstein,
the sheaf $\gr_C^0\upomega=\upomega_X\otimes
\OOO_C$ is invertible at $P$. Consider the injection 
\[
\varphi:\gr_C^0\upomega\hookrightarrow \gr_{C_{1}}^0\upomega \oplus
\gr_{C_{2}}^0\upomega. 
\]
Recall that $(X,C_i)$ is a (birational) extremal curve 
germ 
by Remark \xref{rem-prel-extr-nbd}. Then by \xref{cor-gr-omega-=0}\xref{cor-gr-omega-=0a}
we have $\gr_{C_{i}}^0\upomega=\OOO_{C_i}(-1)$, so
$H^0(\Coker (\varphi))= H^1(\gr_C^0\upomega)$. On the other hand,
$\Coker (\varphi)$ is a sheaf of finite length supported at $P$. Since
$\gr_C^0\upomega$ is invertible, $\Coker (\varphi)$ is non-trivial. So,
$H^1(\gr_C^0\upomega)\neq 0$ and by Corollary \xref{cor-gr-omega-=0}
the contraction $f$ is a $\QQ$-conic bundle. Moreover, if the base $(Z,o)$ is smooth,
then again by Corollary \xref{cor-gr-omega-=0} we have 
$C=f^{-1}(o)$ (scheme-theoretically). Hence $P$ is the only 
singular point of $X$ and 
are done. 
\end{proof}

\section{Topological observations}
Let $\Clsc(X)$ be the subgroup of the divisor class group $\Cl(X)$
consisting of Weil divisor classes which are $\QQ$-Cartier.
We will use the following easy consequence of the classification of terminal singularities
without additional reference.

\begin{proposition}[{\cite[Lemma~5.1]{Kaw:Crep}}]
\label{Clsc}
Let $(X,P)$ be an \textup(analytic\textup) germ of three-dimensional terminal singularity of index $m$.
Then
\begin{equation}\label{eq:Clsc}
\Clsc(X,P)\simeq \uppi_1(X\setminus \{P\})\simeq \ZZ/m\ZZ.
\end{equation}
\end{proposition}

\begin{definition}[{\cite[(0.4.16), (1.7)]{Mori:flip}}]
\label{splitting}
Let $(X,P)$ be a terminal three-dimensional singularity of index $m$
and let $C\subset X$ be a smooth curve passing through $P$. We say
that $C$ is (locally) \emph{primitive} at $P$ if the natural map
\begin{equation*}
\varrho : \ZZ\simeq \uppi_1(C\setminus \{P\})\longrightarrow
\uppi_1(X\setminus \{P\})\simeq \ZZ/m\ZZ
\end{equation*}
is surjective and \emph{imprimitive} 
at $P$
otherwise. The order $s$ of
$\Coker (\varrho)$ is called the \emph{splitting degree} and the
number $\bar m=m/s$ is called the \emph{subindex} of $P\in C$.
\end{definition}
It is easy to see that the splitting degree coincides 
with the number of
irreducible components of the preimage $C^\sharp$ of $C$ under the
index-one cover $X^\sharp\to X$ near $P$. If $P$ is primitive, we
put $s=1$ and $\bar m=m$.

\subsection{}
In the above notation it is easy to show that for any Weil divisor class $\xi\in \Clsc(X,P)$
there exists an effective Weil $\QQ$-Cartier divisor $D$ whose class in 
$\Clsc(X,P)$
equals $\xi$ and such that $D\cap C=\{P\}$. Then one can define the intersection number $\xi\cdot C:= (D\cdot C)_P\mod \ZZ$. 
Hence there exists a well-defined homomorphism 
\begin{equation*}
\mathrm{cl}: \Clsc(X,P) \longrightarrow \textstyle{\frac 1m} \ZZ/\ZZ\subset \QQ/\ZZ,\qquad \xi 
\longmapsto \xi\cdot C.
\end{equation*}
The curve $C$ is locally primitive at $P$ if and only if the map $\mathrm{cl}$ is an 
isomorphism.
In general, the splitting degree equals the order of 
the kernel of $\mathrm{cl}$ \cite[1.7]{Mori:flip}. 

\subsection{}
Let $(X,C)$ be an extremal curve germ with irreducible central fiber $C$.
Let $P_1,\dots, P_n$ be all the non-Gorenstein points of $X$
and let $m_1,\dots, m_n$ be their indices. 
We have the following exact sequence
\begin{equation}
\label{exact-Clcs}
\vcenter{
\xymatrix@R=-4pt{
0\ar[r] & \Pic(X)\ar[r] & \Clsc (X)\ar[r]& \oplus \Clsc(X,P_i)\ar[r] & 0
\\
& \rotatebox{90}{$\simeq$}& &\rotatebox{90}{$\simeq$}&
\\
& \ZZ^{\uprho(X)}&& \oplus \ZZ/m_i\ZZ
}}
\end{equation}

\begin{scorollary}\label{corollary:cover}
In the above notation assume that $C$ is irreducible.
Let $D_i$, $n=1,\dots,n$ be an effective Weil $\QQ$-Cartier divisor whose class generates $\Clsc(X, P_i)$ and
let $H$ be an effective Cartier divisor such that $H\cdot C=1$. Then the following holds.
\begin{enumerate}
 \item \label{corollary:cover0}
The group $\Clsc(X)$ is generated by the classes of $H$, $D_1,\dots, D_n$. 
  \item \label{corollary:cover-i}
If the point $P_i$ is imprimitive of splitting degree $s_i$ and subindex $\bar m_i$, then 
the class of $H-\bar m_i D_i$ is an $s_i$-torsion element in $\Clsc(X)$.
  \item \label{corollary:cover-p}
If $(X,C)$ is locally primitive at distinct points
$P_i, P_j\in C$  and $\gcd(m_i,m_j)=d\neq 1$, then  
the class of $\frac {m_i}d D_i-\frac {m_j}d D_j$ is a $d$-torsion element in $\Clsc(X)$.
\end{enumerate}
\end{scorollary}

\begin{construction}
\label{base-change}
Let $f: (X,C)\to (Z,o)$ be an extremal curve germ
and let $\theta :(X^{\flat},C^{\flat})\to (X,C)$ be a finite cover which is \'etale in codimension one.
Clearly, $\theta$ must be \'etale over the Gorenstein locus of $X$.
The Stein factorization gives us the following diagram.
\begin{equation}
\label{eq:base-change}
\vcenter{
\xymatrix{
(X^{\flat},C^{\flat})\ar[r]^{\theta}\ar[d]^{f^{\flat}}& (X,C)\ar[d]^f
\\
(Z^{\flat},o^{\flat})\ar[r]^{}&(Z,o)
}}
\end{equation}
where $(Z^{\flat},o^{\flat})\to (Z,o)$ is a finite cover
which is \'etale over $Z\setminus \{o\}$. 
We have $K_{X^{\flat}}=\theta^*K_X$ and
singularities of $X^{\flat}$ are terminal. In particular, $(X^{\flat},C^{\flat})$ is an extremal curve germ.
Note that in our situation
$X^{\flat}$ is the normalization of $X\times_Z Z^{\flat}$ and
$C^{\flat}:=f^{\flat  -1}(C)_{\red}$.  

Conversely, if  $f: (X,C)\to (Z,o)$ is an extremal curve germ
and $(Z^{\flat},o^{\flat})\to (Z,o)$ is a finite cover
which is \'etale over $Z\setminus \{o\}$. Then the base change produces 
the diagram \eqref{eq:base-change},  where $X^{\flat}$ is the normalization of $X\times_Z Z^{\flat}$,
$(X^{\flat},C^{\flat})$ is an extremal curve germ, and $\theta$ is \'etale in codimension one.
\end{construction}

\begin{definition}\label{torsion-free-cover}
Let $(X,C)$ be an extremal curve germ. By the above construction \xref{base-change}
the torsion part  $\Cl(X)_{\tors}\subset \Cl(X)$ defines an abelian Galois  cover 
\begin{equation}
\label{eq:torsion-free-cover}
\tau :(X',C')\longrightarrow (X,C)
\end{equation} 
which is \'etale over the Gorenstein locus of $X$. We call this map 
the \emph{torsion free cover} of $(X,C)$
and the degree of this cover we call the \emph{topological index} of $(X,C)$.
Similar to \eqref{eq:base-change} we have the diagram
\begin{equation}
\label{eq:base-cover}
\vcenter{
\xymatrix{
(X',C')\ar[r]^{\tau}\ar[d]^{f'}& (X,C)\ar[d]^f
\\
(Z',o')\ar[r]^{}&(Z,o)
}}
\end{equation}
Hence $(X',C')$ is also an extremal curve germ.
Clearly, $\Cl(X')$ is torsion free.
\end{definition}

\begin{slemma}\label{lemma:cyclic}
Let  $(X,C)$ be an extremal curve germ and let $\theta:(X^\flat,C^\flat)\to (X,C)$
be a finite cover which is \'etale in codimension two. Then $\theta$ is a cyclic cover.
\end{slemma}

\begin{proof}
We may assume that the cover $\theta$ is Galois with group $G$
and it is sufficient to show that the group $G$ is cyclic.
By taking composition with  the torsion free cover, we may assume also that 
$\Clsc(X^\flat)$ is torsion free.
By the construction $G$ effectively acts on $C^\flat=\cup C_i^\flat$ (because $X$ has only isolated singularities).
Since $C^\flat$ is a tree of smooth 
rational curves, it is easy to prove by induction 
on the number of components of $C^\flat$ that $G$ has either an
invariant component $C_i^\flat\subset C^\flat$ or a fixed point $P^\flat\in \Sing(C^\flat)$.

In the latter case, let $P=\theta(P^\flat)$. There is a surjection $\uppi_1(X\setminus
\{P\})\twoheadrightarrow G$. Since $\uppi_1(U\setminus
\{P\})$ is cyclic (see \xref{Clsc}), we are done.

In the former case, let $C_i:=\theta (C^\flat_i)$. By Remark \ref{rem-prel-extr-nbd}
we may replace $(X^\flat, C^\flat)$ with $(X^\flat, C^\flat_i)$ 
and $(X, C)$ with $(X, C_i)$. Thus
$C=C_i$, $C^\flat=C^\flat_i$, and $C^\flat/G=C\simeq \PP^1$. Assume that the group $G$ is not cyclic.
Then there is no fixed points on $C^\flat$. 
If $X^\flat$ has a point of index $m>1$, then its orbit contains 
at least two points of the same index. By \eqref{exact-Clcs}
the torsion part of the group  $\Clsc(X^\flat)$ is non-trivial. This contradicts the assumption above. 
Thus $X^\flat$ is Gorenstein. 

Let $P_1,\dots, P_n\in C$ be all branch points of $C^\flat\to C$
and let  $m_1,\dots, m_n$ be their ramification indices.  
By the Hurwitz formula we can write 
\[
 \frac 1{|G|} \bigl (2 g(C^\flat_i)-2\bigr)= 2g(C_i)-2 +\sum_{i=1}^n \left (1-\frac 1{ m_i} \right)
\]
Hence, $\sum 1/m_i >n- 2$.
Since  the group $G$ is not cyclic, we have  $n>2$. 
The index of the point $P_i\in X$ is equal to $m_i$.
By \eqref{equation:iP-wP} and Lemma \ref{lemma:iP-wP} we
have  $w_{P_i}(0)\ge 1/m_i$ and $i_{P_i}(1)\ge 1$. 
Therefore, $\sum w_{P_i}(0) >1$ and $\deg \gr_C^0\upomega=-1$ by \eqref{eq-grO-iP1-2}.
Then  we get a contradiction by \eqref{eq-grw-w}.
\end{proof}

\begin{scorollary}
\label{prop-cyclic-quo}
Let $(X,C)$ be an extremal curve germ.
Then the torsion part $\Cl(X)_{\tors}\subset \Cl(X)$ is a cyclic group.
Hence the torsion free cover \eqref{eq:torsion-free-cover}
is cyclic.
Moreover, 
$X',C')$ has no finite cover which is \'etale in codimension one.
\end{scorollary}

\begin{scorollary}[{\cite[Lemma 1.10]{P97}}]
\label{corollary:cyclic}
If $f: (X,C)\to (Z,o)$ is a $\QQ$-conic bundle germ, then $(Z,o)$ is a 
cyclic quotient singularity.
\end{scorollary}

\begin{proof}
Follows from Lemma~\xref{lemma:cyclic} and \xref{base-change}. 
\end{proof}

\subsection{}
From now on we assume that $f: (X,C)\to (Z,o)$ is an extremal curve germ with $C\simeq\PP^1$.
Assume that the torsion part $\Cl(X)_{\tors}=\ZZ/d\ZZ$ is non-trivial
and consider the torsion free cover \eqref{eq:torsion-free-cover}. 
Thus $(X,C)=(X', C')/G$
and $(Z,o)=(Z',o')/G$, where $G=\mumu_d$ acts on $Z'\setminus \{o'\}$
and $X'\setminus \tau^{-1}\left(\Sing(X)\right)$ freely.
We distinguish two cases (cf.~\cite[(1.12)]{Mori:flip}):

\begin{scase}\label{case-prim-1}
{\bf Case: $C'$ is irreducible.}
Then $G=\mumu_d$ has exactly two fixed points $P_1'$ and $P_2'$ on $C'\simeq \PP^1$.
They give us two points $P_i:=\tau(P_i')$ on $C$ whose indices are divisible by $d$.
The germ $(X,C)$ is locally primitive along $C$. 
\end{scase}

\begin{scase}\label{top:imprim}
{\bf Case: $C'=\cup_{i=1}^s C_i'$, where $s>1$ and $C'_i\simeq \PP^1$.}
In this case, $G$ acts on $\{C'_1,\dots, C_s'\}$ transitively.
Since $\p(C')=0$, each
component $C_i'$ meets the closure of $C'\setminus C_i'$ at one point.
Therefore, in this case,
all the irreducible components
$C_i'$ pass through one point $P'$ and do not meet each other
elsewhere. In this case $(X,C)$ is imprimitive at $\tau(P')$ of splitting degree $s$
and has no other locally imprimitive points.
\end{scase}

It is worthwhile to mention in the case \xref{top:imprim} that
$\tau(P')$ is the only non-Gorenstein point of $X$ 
and $d=s$ (see \cite[Th.~6.7, 9.4]{Mori:flip} and \cite[\S~7]{MP:cb1}).

\begin{scorollary}[{\cite[(1.10)]{Mori:flip}}] 
Let $(X,C\simeq \PP^1)$ be an extremal curve germ.
Let $P_1,\dots, P_n$ be all the non-Gorenstein points of $X$. 
The following
are equivalent:
\begin{enumerate}
\item 
$D\cdot C= 1/m_1\cdots m_n$ for some $D\in\Clsc(X)$,
\item 
$\Clsc(X)\simeq \ZZ$,
\item 
$\Clsc(X)$ is torsion-free,
\item 
$(X, C)$ is locally primitive and $\gcd(m_i,m_j)=1$, $i\neq j$.
\end{enumerate}
\end{scorollary}

\begin{proof}
Follows from Lemma \xref{lemma:cyclic} and \eqref{exact-Clcs}.
\end{proof}

\begin{scorollary}[cf. {\cite[Lemma~2.8]{MP:cb1}}]
\label{lemma:KC}
Let $(X,C\simeq \PP^1)$ be an extremal curve germ. Let
$d$ be the topological index of $(X,C)$ and let $m_1,\dots,m_r$ be
indices of all the non-Gorenstein points.
Assume that $(X,C)$ is either  divisorial or a $\QQ$-conic bundle which is not 
toroidal \xref{ex-toric}. Then
\begin{equation}
\label{eq:KC}
-K_X\cdot C=d/m_1\cdots m_r.
\end{equation}
\end{scorollary}

\begin{proof}
It follows from \eqref{exact-Clcs} that for the ample generator $D$ of 
the group $\Clsc(X)/{\equiv}$
one has $D\cdot C=d/m_1\cdots m_r$. Write $-K_X\equiv a D$ for some $a\in \ZZ$.
Intersecting $D$ and $K_X$ with a general one-dimensional fiber $L$,
we obtain $-K_X\cdot L=D\cdot L$ and $a=1$.
\end{proof}

Now we can strengthen the assertion of Lemma  \xref{lemma-int-non-Gor}.

\begin{scorollary}
\label{cor:int-non-Gor}
 Let $f: (X,C) \to (Z,o)$ be an extremal curve germ.
Suppose that $C$ is reducible and let $P$ be a singular 
point of $C$. 
If $X$ is Gorenstein at $P$, then $(Z,o)$ is smooth and $f$ is a  Gorenstein conic bundle.
\end{scorollary}

\begin{proof}
By Lemma  \xref{lemma-int-non-Gor} \ $f$ is a $\QQ$-conic bundle and 
$(Z,o)$ is singular. Recall  (see Lemma \ref{lemma:base}) that $(Z,o)$ is a quotient singularity.
Thus there is a finite Galois \'etale over $Z\setminus \{o\}$ cover $(Z^{\flat},o^{\flat})\to(Z,o)$
where $(Z^{\flat},o^{\flat})$ is smooth. Then we can consider the base change (see
\eqref{eq:base-change}). Thus $X=X^{\flat}/G$, where $G$ is a finite group acting on $X^{\flat}$ 
freely outside finite number of points.
Since $X$ is Gorenstein  at
$P$, so is $X^{\flat}$ at all the points $P_i^{\flat}\in \theta^{-1}(P)$. Moreover, $\theta$
is \'etale over $P$ by \eqref{eq:Clsc}. Hence, the central curve $C^{\flat}$ is singular at
$P_i^{\flat}$. By Lemma \ref{lemma-int-non-Gor} the variety $X^{\flat}$ is Gorenstein and by Corollary
\xref{prop:Gor} the contraction $f^{\flat}: X^{\flat}\to Z^{\flat}$ is a standard Gorenstein conic
bundle. In particular, $C^{\flat}$ is a plane conic. 
Thus $C^{\flat}$ has two components meeting at one point $\theta^{-1}(P)$
which must be fixed by $G$. Again by  \eqref{eq:Clsc}
the group $G$ is trivial, a contradiction.
\end{proof}

\begin{scorollary}
\label{zam-imp-Gor-fac}
Let $f: (X,C\simeq \PP^1) \to (Z,o)$ be an extremal curve germ.
Assume that $(X,C)$ is locally imprimitive at $P$.
If the subindex of $P$ equals $1$, then $f$ is a $\QQ$-conic bundle and
in the diagram \eqref{eq:base-cover} the contraction $f'$ is a Gorenstein conic bundle.
\end{scorollary}

$\QQ$-conic bundles which are quotients of Gorenstein conic bundles by a finite group were
described in \cite[\S~2]{P97}. It turns out that such a $\QQ$-conic bundle
is locally imprimitive if and only if it is of type~\xref{item-main-th-impr-barm=1}.

\begin{scorollary}[cf. {\cite[Prop. 1.14]{Mori:flip}}]
\label{cor-prop-grw=2-2-points-prim}
Let $f: (X,C)\to (Z,o)$ be an extremal curve germ.
Assume that $C$ is irreducible. If $\gr_C^0\upomega \not\simeq
\OOO_C(-1)$, then $f$ is a $\QQ$-conic bundle and 
in notation of \eqref{eq:base-cover} we
have $f'^{-1}(o')=C'$. If furthermore $(X,C)$ is locally primitive,
then it is toroidal \textup(see \xref{ex-toric}\textup).
\end{scorollary}

\begin{proof}
By Remark \ref{remark:grw} the contraction $f$ is a $\QQ$-conic bundle.
Apply the construction \eqref{eq:base-cover}.
Since 
\[
H^1(\gr_{C}^0\upomega)=H^1(\gr_{C'}^0\upomega)^{\mumu_d}, 
\]
we have $H^1(\gr_{C'}^0\upomega)\neq 0$. 
By Corollary~\xref{cor-gr-omega-=0} $C'=f'^{-1}(o')$. If $f$ is locally
primitive, $C'$ is irreducible (see \xref{case-prim-1}). So $C'\simeq
\PP^1$ and $X'$ is smooth. Up to analytic isomorphism we may assume
that there exists a $\mumu_d$-equivariant 
decomposition $X'\simeq Z'\times \PP^1$. So, $f$ is toroidal \xref{ex-toric},
\cite[\S~2]{P97}.
\end{proof}

\begin{proposition}[{\cite[Lemma 9.2.3]{MP:cb1}}, {\cite[0.4.13.3]{Mori:flip}}]
\label{lem-prop-3points}
An extremal curve germ $(X,C\simeq\PP^1)$ has at most two
non-Gorenstein points.
\end{proposition}

\begin{proof}
Assume that $P_1$, $P_2$, $P_3\in X$ are singular points of indices $m_1$,
$m_2$, $m_3 >1$. 
If $(X,C)$ is locally imprimitive at some point, then 
the torsion free cover $\tau: (X',C')\to(X,C)$ has the form \ref{top:imprim},
i.e. $C'$ is a union of $s$ components $C_1',\dots, C_s'$ passing through one point, say $P'$,
and $\tau$ is \'etale over $X'\setminus \{P'\}$. By Corollary \xref{zam-imp-Gor-fac}
the point $P'\in X'$ is not Gorenstein.
Thus for any component $C_i'$ the germ $(X,C_i')$ has at least three non-Gorenstein 
points. 

Replacing $(X,C)$ with $(X',C_i')$ 
we may assume that  $(X,C)$ is
locally primitive, i.e. the maps $\uppi_1(C\setminus \{P_i\})\to \uppi_1(U_i\setminus \{P_i\})$
are surjective, where $U_i\subset X$ is a small neighborhood of $P_i$. 
Then using Van Kampen's theorem and \eqref{eq:Clsc} it is easy to compute
the fundamental group of $X\setminus\{P_1,\, P_2,\, P_3\}$:
\begin{equation*}
\uppi_1(X\setminus\{P_1,\, P_2,\, P_3\})= \langle
\upsigma_1,\upsigma_2,\upsigma_3\rangle/ \{\upsigma_1^{m_1}=\upsigma_2^{m_2}=
\upsigma_3^{m_3}=\upsigma_1\upsigma_2\upsigma_3=1\}.
\end{equation*}
This group has a finite quotient group $G$ in which the images of
$\upsigma_1$, $\upsigma_2$, $\upsigma_3$ are exactly of order $m_1$, $m_2$ and
$m_3$, respectively (see, e.g., \cite{Feuer-1971}). The above quotient
defines a finite Galois cover $\tau: (X',C')\to (X,C)$ with non-abelian Galois group $G$.
This contradicts Lemma~\xref{lemma:cyclic}.
\end{proof}

\section{Local description}
\subsection{Notation}\label{notation:terminal} 
Let $(X,P)$ be a threefold terminal singularity of index $m$ and
let $C \subset (X, P)$ be a smooth
curve such that $P$ has subindex $\bar m$ and splitting degree $s$
(see \xref{splitting}). 
We use the notation of \xref{3terminal}. Put $C^\sharp:=\pi^{-1}(C)$. Then $C^\sharp$ has $s$ 
irreducible components.

Let $(C^\dag, P^\dag)$ be the normalization 
of an irreducible component $C^\sharp(i)\subset C^\sharp$, $1\le i\le s$
and let $\tau: (X',C')\to (X,C)$ be the torsion free cover  (see \eqref{eq:base-cover}). 
Then $\mumu_m$ naturally acts on $(X^\sharp, P^\sharp)$ and $(C^\sharp, 
P^\sharp)$, and 
so does $\mumu_{\bar m}$ on $(C',P')$. Let 
\begin{equation*}
\eta: \OOO_{X^\sharp, P^\sharp} \longrightarrow \OOO_{X^\dag, P^\dag}
\end{equation*}
be the natural map. Since $(X, P)$ and $(C, P)$ are normal, one has 
\begin{equation*}
\OOO_{X,P}= \left(\OOO_{X^\sharp, P^\sharp}\right)^{\mumu_m}\quad\text{and}
\quad
\OOO_{C,P}= \left(\OOO_{C^\dag, P^\dag}\right)^{\mumu_{\bar m}}.
\end{equation*}
Since $\mumu_m$ acts freely on $X^\sharp\setminus \{P^\sharp\}$, so it does 
on $C^\sharp\setminus \{P^\sharp\}$ and hence $\mumu_{\bar m}$ on 
$C^\dag\setminus \{P^\dag\}$.
Hence $\OOO_{C^\dag, P^\dag}$ has a uniformizing parameter, say $t$,
such that $t$ is a $\mumu_{\bar m}$-semi-invariant. Let $\chi$ be a generator 
of 
$\Hom(\mumu_m, \CC^*) = \ZZ/m\ZZ$ whose restriction $\bar \chi$ to 
$\mumu_{\bar m}$ is the character associated to $t$. Then 
$\OOO_{C,P}=\CC\{t\}^{\mumu_{\bar m}}$.

For a semi-invariant $z\neq 0$, let $C^\sharp\text{-}\wt(z)$(or simply $\wt(z)$ 
if there is no confusion) 
be $n\in \ZZ/m\ZZ$ such that $n\chi$ is the
character associated to $z$. 
For a $\mumu_{m}$-semi-invariant $z\in \OOO_{X^\sharp,P^\sharp}$, let 
\begin{equation*}
C^\sharp\text{-}\ord(z):= \sup\left\{n\in \ZZ_{\ge 0} \mid \eta(z)\in 
t^n\CC\{t\}\right\}. 
\end{equation*}
We also write $\ord(z)$, if it does not 
cause confusion. Let 
\begin{equation*}
\ow(z):=(\ord(z), \wt(z)).
\end{equation*} 
We define semigroups
\begin{eqnarray*}
\ord(C^\sharp) &:= &\left\{\ord(z) \mid z\in \OOO_{C^\sharp, P^\sharp},\ z\neq 
0\right\}\subset \ZZ_{>0},
\\
\ow(C^\sharp) &:= &\left\{\left(\ord(z), \wt(z)\right) \mid z\in 
\OOO_{C^\sharp, P^\sharp},\ z\neq 0\right\}
\subset \ZZ_{\ge 0}\times \ZZ/m\ZZ.
\end{eqnarray*}
One can show that in some coordinates $C^\sharp$ can be given 
by a monomial parametrization (see \cite[Lemma~2.7]{Mori:flip} for the precise 
statement).

\subsection{Notation}
\label{not-nazalo-loc}
Let $f: (X,C\simeq\PP^1)\to (Z,o)$ be an
extremal curve germ and let $P\in C$ be a point of index $m\ge 1$. 
Let $s$ and $\bar
m$ be the splitting degree and subindex, respectively. 
Consider the index-one $\mumu_m$-cover $\pi:
(X^\sharp,P^\sharp)\to (X,P)$ and let $C^\sharp:=\pi^{-1}(C)$. Take
normalized $\ell$-coordinates $(x_1,\dots,x_4)$ and let
$\phi$ be an $\ell$-equation of $X\supset C\ni P$ (see
\cite[2.6]{Mori:flip}). Put $a_i:=\ord(x_i)$.
Note that $a_i<\infty$ and $\wt(x_i)\equiv a_i\mod \bar m$. 

The following is the key fact in the local classification of 
possible singularities of extremal curve germs.

\begin{lemma}[{\cite[3.8, 4.2]{Mori:flip}}, {\cite[\S~5]{MP:cb1}}]
\label{lemma:planar}
In the above notation, assume that $P$ is not of type~\typec{IE^\vee} below.
Then $\ow(C^\sharp)$ is generated by $\ow(x_1)$ and $\ow(x_2)$.
In particular, $C^\sharp$ is a planar curve. 
\end{lemma}

This lemma allows to obtain a local classification of 
possible singularities. We reproduce this classification below.
We start with the primitive case.

\begin{proposition}[{\cite[Prop. 4.2]{Mori:flip}}, {\cite[Prop. 5.2.1]{MP:cb1}}]
\label{prop:local-primitive}
Let $f: (X,C\simeq\PP^1)\to (Z,o)$ be an
extremal curve germ and let $P\in C$ be a primitive point of index $m\ge 1$. 
Then modulo permutations of $x_i$'s, the
semigroup $\ord(C^\sharp)$ is generated by $a_1$ and $a_2$. Moreover,
exactly one of the following holds:
\begin{enumerate}
\item [\typec{IA}]
$a_1+a_3\equiv 0\mod m$, $a_4=m$, $m\in \ZZ_{>0} a_1+\ZZ_{>0} a_2$,
where we may still permute $x_1$ and $x_3$ if $a_2=1$,
\item [\typec{IB}]
$a_1+a_3\equiv 0\mod m$,\ $a_2=m$,\ $a_1\ge 2$,
\item [\typec{IC}]
$a_1+a_2=a_3=m$,\ $a_4\not\equiv a_1,\, a_2 \mod m$,\ $2\le a_1<a_2$, $m\ge 5$,
\item [\typec{IIA}]
$m=4$, $P$ is of type~\type{cAx/4}, and $\ord(x)=(1,1,3,2)$,
\item [\typec{IIB}]
$m=4$, $P$ is of type~\type{cAx/4}, and $\ord(x)=(3,2,5,5)$,
\item [\typec{III}]
$m=1$, $X=X^\sharp$, $C=C^\sharp$, and $P\in X$ is a \type{cDV} point.
\end{enumerate}
\end{proposition}

Now consider the locally imprimitive case.

\begin{proposition}[{\cite[Prop. 4.2]{Mori:flip}}, {\cite[Prop. 5.3.1]{MP:cb1}}]
\label{prop-imp-types}
Let $f: (X,C\simeq\PP^1)\to (Z,o)$ be an
extremal curve germ and let $P\in C$ be an imprimitive point of index $m$. 

Modulo permutations of
$x_i$'s and changes of $\ell$-characters, the semigroup $\ow(C^\sharp)$ is generated by 
$\ow(x_1)$ and $\ow(x_2)$ except for the
case \typec{IE^{\vee}} below. Moreover, exactly one of the following
holds:
\begin{enumerate}
\item[\typec{IA^{\vee}}]
$\bar m>1$, $\wt(x_1)+\wt(x_3)\equiv 0\mod m$, $\ow x_4=(\bar m,0)$,
$\ow(C^\sharp)$ is generated by $\ow(x_1)$ and $\ow(x_2)$, and
$w_P(0)\ge 1/2$.

\item[\typec{IC^{\vee}}]
$s=2$, $\bar m$ is an even integer $\ge 4$, and
\begin{equation*}
\begin{array}{cccccc}
&x_1&x_2&x_3&x_4&
\\
\wt&1&-1&0&\bar m+1&\mod m
\\
\ord&1&\bar m-1&\bar m&\bar m+1&
\end{array}
\end{equation*}

\item[\typec{II^{\vee}}]
$\bar m=s=2$, $P$ is of type~\type{cAx/4}, and
\begin{equation*}
\begin{array}{cccccc}
&x_1&x_2&x_3&x_4&
\\
\wt&1&3&3&2&\mod 4
\\
\ord&1&1&1&2&
\end{array}
\end{equation*}

\item[\typec{ID^{\vee}}]
$\bar m=1$, $s=2$, $P$ is of type~\type{cA/2} or \type{cAx/2}, and
\begin{equation*}
\begin{array}{cccccc}
&x_1&x_2&x_3&x_4&
\\
\wt&1&1&1&0&\mod 2
\\
\ord&1&1&1&1&
\end{array}
\end{equation*}

\item[\typec{IE^{\vee}}]
$\bar m=2$, $s=4$, $P$ is of type~\type{cA/8}, and
\begin{equation*}
\begin{array}{cccccc}
&x_1&x_2&x_3&x_4&
\\
\wt&5&1&3&0&\mod 8
\\
\ord&1&1&1&2&
\end{array}
\end{equation*}
\end{enumerate}
Moreover, cases \typec{ID^{\vee}} and \typec{IE^{\vee}} occurs if 
and only if $f$ is 
a $\QQ$-conic bundle and $C'=f'^{-1}(o')$. In these cases, $P$ is the only 
non-Gorenstein point.
\end{proposition}

Proofs are based on very careful local computations. We do not present them here. See 
Example~\xref{ex:ICdual}
for sample of computations.

\begin{scorollary}
\label{corollary:gr-w}
Let $(X,C\simeq\PP^1)$ be an extremal curve germ. Assume that $\gr_C^0\upomega\not\simeq \OOO(-1)$.
Then $(X,C)$ is a $\QQ$-conic bundle germ which is either toroidal or the only 
non-Gorenstein point of $(X,C)$ is of type~\typec{ID^\vee}.
\end{scorollary}

\begin{proof}
By Corollary \xref{cor-prop-grw=2-2-points-prim}\ 
$(X,C)$ is $\QQ$-conic bundle and $C'=f'^{-1}(o')$. 
Assume that it is not a toroidal. 
Then again by \xref{cor-prop-grw=2-2-points-prim} it is locally imprimitive 
and by the proposition above $(X,C)$ has a unique non-Gorenstein point, say $P$,
which is of type~\typec{ID^\vee} or \typec{IE^\vee}. In the case \typec{IE^\vee}
we have $-K_X\cdot C=1/2$ (see \xref{lemma:KC}). Easy computations show that $w_P(0)=1/2$
and so $\deg \gr_C^0\upomega=-1$ by \eqref{eq-grw-w}.
\end{proof}

\subsection{}
By Lemma \xref{lemma:planar} there exist monomials
$\lambda_3$ and $\lambda_4$ in $x_1$, $x_2$ such that 
$x_3= \lambda_3(x_1, x_2)$ and $x_4= \lambda_4(x_1, x_2)$ on $C^\sharp$.
Then
\begin{equation*}
x_1^{sa_2}-x_2^{sa_1},\qquad x_3-\lambda_3,\qquad x_4-\lambda_4
\end{equation*} 
generate the defining ideal $I^\sharp$ of $C^\sharp\subset \CC^4$.
Then the equation of $X^\sharp$ can be written as follows
\begin{equation*}
\phi= (x_1^{sa_2}-x_2^{sa_1})\phi_2+(x_3-\lambda_3)\phi_3+(x_4-\lambda_4)\phi_4
\end{equation*} 
for some semi-invariant $\phi_i\in \CC\{x_1,\dots,x_4\}$ with suitable weights.

\begin{lemma}
\label{lemma:local-eq}
Under the notation of \xref{prop:local-primitive} and \xref{prop-imp-types}, one has
\begin{enumerate}
\item \label{lemma:local-eq1}
if $P$ is of type~\typec{IC} or \typec{IC^\vee}, then 
$(X^\sharp, P^\sharp)$ is smooth and $I^\sharp = (x_1^{sa_2}-x_2^{sa_1},\, x_4-\lambda_4,\phi)$;
\item\label{lemma:local-eq2}
if $P$ is of type~\typec{IIB} or \typec{II^\vee}, then 
$I^\sharp = (x_3-\lambda_3, x_4-\lambda_4, \phi)$.
\end{enumerate}
\end{lemma}

\begin{proof}
Let us consider for example the case \typec{IC}. Then $\lambda_3$ must be $x_1x_2$, 
and so
\begin{equation*}
\phi= (x_1^{sa_2}-x_2^{sa_1})\phi_2+(x_3-x_1x_2)\phi_3+(x_4-\lambda_4)\phi_4.
\end{equation*} 
Since $P$ is 
of type~\typec{IC}, one sees that $m = a_1 + a_2 > 4$, $\phi_2\in (x)$, and that 
$\phi_4,\, \lambda_4\in (x)^2$ because 
$\wt(x_4)\not\equiv 0, \pm \wt(x_1), \pm \wt(x_2) \mod m$.
Since $m\ge 5$, by the classification of terminal singularities
either $x_1 x_2$ or $x_3$ must appear in 
the power series expansion. Since $a_1,\, a_2 \ge 2$, this is only possible if $\phi_3$ 
is a unit. 
\end{proof}

\begin{example}
\label{ex:ICdual}
According to Lemma \xref{lemma:local-eq}\xref{lemma:local-eq1}
a point $P\in (X,C)$ of type \typec{IC^\vee} can be written as follows:
\begin{equation*}
(X,C,P)=\bigl(\CC^3_{x_1,x_2,x_4}, \{x_4=x_2^2-x_1^{2\bar{m}-2}=0\}, 0\bigr)/\mumu_{2\bar m}(1,-1,\bar{m}+1)
\end{equation*} 
We have $C\simeq \{x_4=x_2-x_1^{\bar{m}-1}=0\}/\mumu_{\bar{m}}$ and a local uniformizing
parameter on $C$ is $x_1^{\bar{m}}$. Hence, $\OOO_{C,P}=\CC\{x_1^{\bar{m}}\}$.
Furthermore, 
\begin{eqnarray*}
\OOO_C(mK_X)&=&\OOO_C(d x_1\wedge d x_2\wedge d x_4)^m,
\\
\gr_C^0\upomega &=& \OOO_C(x_1^{\bar{m}-1}d x_1\wedge d x_2\wedge d x_4),
\\
\gr_C^1\OOO&=&\OOO_C(x_1^{\bar{m}-1}x_4)\oplus\OOO_C(x_1^2(x_2^2-x_1^{2\bar{m}-2})),
\\
w_P(0)&=&(\bar{m}-1)/\bar{m},\quad i_p(1)=2.
\end{eqnarray*}
\end{example}

\section{Deformations}
In this section we discuss deformations of extremal curve germs.
It is known that a small deformation of a terminal singularity is again terminal
(see e.g. \cite[Theorem~9.1.14]{Ishii:book}). Moreover, any three-dimensional
terminal singularity admits a $\QQ$-smoothing, i.e. a deformation to
a collection of cyclic quotient singularities \cite[6.4A]{Reid:YPG}.

To study extremal curve germs it is very convenient to deform 
an original germ to more general one. For example sometimes
this procedure increases the number of singular points and 
it can be used to derive a contradiction.

\begin{definition}
Let $X$ be a threefold
with at worst terminal singularities.
We say $X$ is \emph{ordinary} at $P$ (or $P$ is an ordinary point) if $(X, P)$
is either a cyclic quotient singularity or an ordinary double point.
\end{definition}

First, we note that for extremal curve germs deformations are unobstructed:

\begin{proposition}[{\cite[1b.8.2]{Mori:flip}}, {\cite[11.4.2]{KM92}}, 
{\cite[6.1]{MP:cb1}}]
\label{prop-cor-def-f}
Let $f : (X, C) \to (Z,o)$ be an extremal curve germ
and let $P\in C$. Then
every deformation of germs $(X,P)\supset (C,P)$ can be extended to a
deformation of $(X, C)$ so that the deformation is trivial outside
some small neighborhood of $P$.
\end{proposition}

\begin{proof}
Let $P_i\in X$ be singular points. Consider the natural morphism
\begin{equation*}
\Psi: \Def (X) \longrightarrow \prod\Def (X,P_i).
\end{equation*} 
It is sufficient to show that $\Psi$ is smooth (in particular,
surjective). The obstruction to globalizing a deformation in
$\prod\Def (X,P_i)$ lies in $R^2 f_*T_X$. Since $f$ has only
one-dimensional fibers, $R^2 f_*T_X=0$.
Alternate, more explicit proof can be found in {\cite[1b.8.2]{Mori:flip}}.
\end{proof}

The following theorem was proved in {\cite[Th. 3.2]{MP:IA}} for $f$ divisorial;
in {\cite[(11.4)]{KM92}} for $f$ flipping; or
in {\cite[(6.2)]{MP:cb1}} for $f$ a $\QQ$-conic bundle.

\begin{theorem}[{\cite[(6.2)]{MP:cb1}}]
\label{theorem-main-def}
Let $f: (X,C) \to (Z,o)$ be a divisorial 
\textup{(}resp. flipping, $\QQ$-conic bundle\textup{)} curve germ,
where $C$ is not necessarily irreducible.
Let $ {\pi}: \Xf \to (\CC_\lambda^1,0)$ be a flat deformation of 
$X= \Xf_0:= {\pi}^{-1}(0)$
over a germ $(\CC_\lambda^1,0)$ with a flat closed subspace 
$\Cf\subset\Xf$ such that $C= \Cf_0$. Then there exists a
flat deformation $\Zf \to (\CC_\lambda^1,0)$ and
a proper $\CC_\lambda^1$-morphism $\fb: \Xf \to \Zf$
such that $f= \fb_0$ and 
\[
\fb_\lambda: (\Xf_\lambda, 
\fb_\lambda^{-1}(o_\lambda)_{\red})
\to (\Zf_\lambda,o_\lambda)
\]
is a divisorial
\textup{(}resp. flipping, $\QQ$-conic bundle\textup{)} extremal curve germ
for every small $\lambda$, where 
$o_\lambda:=\fb_\lambda(\Cf_\lambda)$.
\end{theorem}

Note however that the deformations do not preserve 
irreducibility of the central fiber: one can easily construct 
an example of
an extremal curve germ $(X,C\simeq \PP^1)$ whose deformation 
$(\Xf_\lambda, \fb_\lambda^{-1}(o_\lambda)_{\red})$ 
has reducible central fiber.
In practice, we often pick up a suitable irreducible component
of $\fb_\lambda^{-1}(o_\lambda)_{\red}$ and obtain 
an extremal curve germ whose central fiber is irreducible (see Remark~\xref{rem-prel-extr-nbd}).

\begin{scase}
Let $f: (X,C) \to (Z,o)$ be an extremal curve germ with a singular point $P\in 
C$ of of index 
$m$, and let $P_1,\dots, P_r$ be all the other singular points of $X$ on $C$. 
Let $(X^\sharp, P^\sharp)\to (X,P)$ be the index one cover
and let $(X^\sharp, P^\sharp)\subset (\CC^4_{x_1,\dots,x_4},0)$ be an equivariant embedding as in \xref{clasifiction-terminal}.
Let $\phi=0$ be an equation of $X^\sharp$.
We will choose semi-invariant $\psi\in \CC\{x_1,\dots,x_4\}$ 
with $\wt(\psi) \equiv \wt(\phi) \mod m$
such that 
\begin{equation*}
X_{\lambda,\epsilon}:=\{(x_1,\dots,x_4)\mid \phi+\lambda\psi=0,\ |x_i|<\epsilon\}/\mumu_m
\subset \CC^4/\mumu_m
\end{equation*} 
has only terminal singularities
for $|\lambda|\ll \epsilon\ll 1$.
\end{scase}

\begin{sproposition}[{\cite[4.7]{Mori:flip}}]
\label{prop:localDeformations}
For suitable choice of $\psi$, each nearby extremal
curve germ $X_\lambda^o\supset C_\lambda\simeq \PP^1$ contains $P,P_1, \dots 
P_r$ 
so that $(X_{\lambda}^o, P_i) \supset (C_{\lambda}, P_i)$ is naturally 
isomorphic to $(X, P_i) \supset(C, P_i)$ for all $i$ and 
$X_{\lambda,\epsilon} \supset C_{\lambda,\epsilon}$ 
contains all 
the singularities ($\in C_\lambda$) of $X_\lambda\supset C_\lambda$ other than 
$P_1, \dots P_r$. All the singularities of $X_{\lambda,\epsilon} \supset C_{\lambda,\epsilon}$
are ordinary.

If $P$ is a primitive \textup(resp. an imprimitive\textup) point, then 
$X_{\lambda,\epsilon} \supset C_{\lambda,\epsilon}$ is locally primitive 
\textup(resp. $P$ is an imprimitive point of $X_{\lambda,\epsilon} \supset 
C_{\lambda,\epsilon}$ 
with the same subindex and splitting degree as $X \supset C \ni P$\textup).
Depending on the type 
of $X \supset C \ni P$, one has
\par\medskip\noindent
{\rm
\begin{tabularx}{1\textwidth}{l|l|l|l|X}
\multirow{2}{*}{type}&\multicolumn{4}{c}{$X_{\lambda,\epsilon} \supset 
C_{\lambda,\epsilon}$}
\\\cline{2-5}
& & type & index & $w_{P}(0)$, $w_{P_i}(0)$ on $X_\lambda$
\\\hline
\typec{IA}&$P$ & \typec{IA} & $m$ & the same as for $X$
\\
\typec{IA^\vee}& $P$ & \typec{IA^\vee} & $m$ & the same as for $X$
\\
\typec{IIA} &
$P$ & \typec{IA} & $m$ &the same as for $X$ 
\\
\typec{II^\vee} & $P$ &\typec{IA^\vee}& $m$ &the same as for $X$ 
\\
\typec{IB} &
$a_1$ points & \typec{IA}& $m$ & the same as for $X$ 
\\
\typec{IIB}&$P$ and $Q$ & \typec{IA}&$4$ and $2$& 
\\
\typec{III}&$i_P(1)$ points& \typec{III} & $1$ &$0$
\end{tabularx}}
\par\medskip\noindent
In the case \typec{III}, one can also make $X_{\lambda,\epsilon}$ smooth by 
choosing 
some other suitable $\psi$. 
\end{sproposition}

\begin{scorollary}
Arbitrary extremal curve germ $(X,C)$ can be deformed to an extremal curve germ
$(X^{o},C^{o})$ with only ordinary points.
\end{scorollary}

\begin{scorollary}\label{corollary:flipGorenstein}
A flipping extremal curve germ $(X,C)$ has at least one non-Gorenstein point.
\end{scorollary}

\begin{proof}
Assume that $(X,C)$ has only type~\typec{III}
singular points. Applying smoothings 
as in \xref{prop:localDeformations} repeatedly at type~\typec{III} points, 
one obtains a flipping extremal curve germ $(X^{o},C^{o})$ such that
$X^{o}$ is smooth.
Thus by \eqref{eq-grO-iP1} and \xref{cor-prop-grw=2-2-points-prim} 
for the normal bundle of $C^{o}$ one has 
\begin{equation*}
\deg \NNN_{C^{o}/X^{o}} = -\deg\gr^1_{C^{o}} \OOO = -1.
\end{equation*} 
Hence the space of deformation of $C^{o}$ in $X^{o}$ has dimension $\ge 1$. This means
that $C^{o}$ moves inside $X^{o}$. 
This contradicts our assumption that $(X,C)$ is flipping.
\end{proof}

If $f: X\to Z$ is a $K$-negative extremal divisorial contraction from a 
variety $X$ with terminal $\QQ$-factorial singularities, then 
the target variety $Z$ is also terminal.
This is no longer true for divisorial extremal curve germs.
The problem is that the exceptional locus of $f$ is not necessarily 
a divisor in this case (because the $\QQ$-factoriality is not assumed).
Nevertheless we have the following.

\begin{theorem}[{\cite[Th.~3.1]{MP:IA}}]
\label{thm:div:Q-Cartier}
Let $f:(X,C)\to (Z,o)$ be a three-dimensional divisorial
extremal curve germ, 
where $C$ is not necessarily irreducible, 
and let $E$ be its exceptional locus. 
Then the divisorial part of $E$ is a $\QQ$-Cartier divisor.
If furthermore $C$ is irreducible, then
$E$ is $\QQ$-Cartier and
$(Z,o)$ is a terminal singularity. 
\end{theorem}

\begin{corollary}
\label{theorem-main-Q-Cartier-i}
Let $f:(X,C\simeq \PP^1)\to (Z,o)$ be a three-dimensional birational
extremal curve germ. 
Then $f$ is divisorial if and only if 
$(Z,o)$ is a terminal singularity. 
\end{corollary}

The proof uses deformation techniques.
\begin{scorollary}[{\cite[Th. 6.3]{Mori:flip}}, {\cite[Prop.~8.3]{MP:cb1}}]
\label{corollary:IB}
An extremal curve germ $(X,C)$ cannot have a point of type~\typec{IB}.
\end{scorollary}
\begin{proof}
Assume that $(X,C)$ has a type~\typec{IB} point $P$.
We may apply deformations to $(X,C)$ and obtain
an extremal curve germ $(X^{o},C^{o}\simeq \PP^1)$ with only ordinary singular points which has at
least two points $P^{o}$ and $Q^{o}$ of type~\typec{IA} with the same index $m$ ($> 1$).
By Proposition \xref{lem-prop-3points} \ $P^{o}$ and $Q^{o}$ are the only 
non-Gorenstein points of $X^{o}$, and $P$ is the only 
non-Gorenstein point of $X$. Thus by \eqref{exact-Clcs} we have $\Clsc(X^o)\simeq \ZZ\oplus \ZZ/m\ZZ$
and so there exists 
an \'etale outside $\{P^{o},\, Q^{o}\}$ cyclic cover $(X',C')\to (X^{o},C^{o})$ of degree $m$
such that $(X',C')$ is again an extremal curve germ with $C'\simeq \PP^1$.
By Corollary \xref{corollary:flipGorenstein} the germ $(X',C')$ and $(X^{o},C^{o})$ cannot be flipping.

Assume that $(X^{o},C^{o})$ is divisorial and let $f^o: (X^{o},C^{o}) \to (Z^{o}, o^{o})$
be the corresponding contraction. By Theorem \xref{thm:div:Q-Cartier}
the point $(Z^{o}, o^{o})$ is terminal and the construction 
\eqref{base-change} shows that $(Z^{o}, o^{o})$ is of index $m$.
According to \cite{Kawamata:discr} the exists an exceptional divisor, say $E$,
with center $o^{o}$ whose discrepancy equals $a(E,Z^{o})=1/m$. On the other hand,
since $E$ is not $f^o$-exceptional and the contraction $f^o$ is $K$-negative, we have 
$a(E,Z^{o})>a(E,Z^{o})=1/m$, a contradiction. 

Finally, assume that $(X^{o},C^{o})$ is a $\QQ$-conic bundle. Since 
$(X,C)$ has exactly one non-Gorenstein point which is locally primitive, the base 
$(S,o)$ is smooth. Since $(X^{o},C^{o})$ has two points of the same index $>1$,
the base $(Z^{o}, o^{o})$ is singular. By \xref{theorem-main-def} there exists a 
deformation family whose general fiber is $(Z^{o}, o^{o})$ and the special fiber is $(S,o)$
(in this case a general fiber $\fb_\lambda^{-1}(o_\lambda)_{\red}$
must be irreducible). 
This is impossible.
\end{proof}

\subsection{}
\label{def:existence}
Deformation arguments are also used to show the existence of extremal curve germs.
Suppose we are given a normal surface germ $(H, C)$ along a curve $C\simeq \PP^1$
and a contraction $f_H: H\to H_Z$ such that $C$ is a fiber. 
Let $P_1, \dots, P_r\in H$ be singular points. Assume also that near each point $P_i$
there exists a small one-parameter deformation $H_t^{i}$ 
of $H \cap U_{P_i}$, where $U_{P_i}$ is a neighborhood of $P_i$, such that 
the total space $V^{i} = \cup H_t^{i}$ has terminal singularity at $P_i$. 
Further, by the arguments similar to that in Proposition \xref{prop-cor-def-f} we see that the 
natural morphism $\Def H \to\prod \Def(H, P_i)$ is 
smooth. 
Hence there exists a global one-parameter deformation $H_t$ of $H$
which induces a local deformation of $H_t^{i}$ near each $P_i$.
Then we construct a threefold $X$ as a total one-parameter deformation 
space $X = \cup H_t$. 
This shows the existence of $X\supset C$ with $H\in |\OOO_X |$ and such that 
$P_i\in C\cap U_{P_i}\subset U_{P_i}$ has the desired structure. 
Note however 
that $H$ may not be \emph{general} in $|\OOO_X |$.)
The contraction $f : X \to Z$ exists by arguments similar to 
\cite[11.4.1]{KM92} and Theorem~\xref{theorem-main-def}. 
The contraction is birational (resp. $\QQ$-conic bundle)
if $H_Z$ is a surface (resp. a curve).

\section{General member of $|-K_X|$}
\label{sect:elephant}
\subsection{}
Let $X$ be a threefold having terminal singularities only and let $D$
be an effective integral $\QQ$-Cartier divisor on $X$. Then $D$ is 
Cohen-Macaulay \cite[Cor. 5.25]{KM:book}. Therefore, we have 
\begin{itemize}
\item 
if $D$ has only isolated 
singularities, then $D$ is normal;
\item
if $D\sim -K_X$, then $D$ is Gorenstein.
\end{itemize}
In certain situations we can say more:

\begin{theorem}[{\cite[(6.4B)]{Reid:YPG}}]
\label{thm:elephant}
Let $(X, P)$ be a three-dimensional terminal singularity. Then a general 
member of $|- K_X|$ has at most a Du Val singularity at $P$.
\end{theorem}

\begin{scase}\label{local:elephant}
Depending on the types of terminal singularities, a general member 
$D\in |-K_X|$ and its preimage $D^\sharp$ under the index-one cover are described 
below (see {\cite[(6.4B)]{Reid:YPG}}).

\par\medskip\noindent
\begin{center}
\begin{tabularx}{1\textwidth}{c|c|c|c|c}
name & equation of $D^\sharp$&$\mumu_m$-action &cover $D^\sharp\to D$& $\aw(X,P)$
\\\hline
\type{cA/m} & $xy+z^{k}$ &$(1,-1,0)$&\type{A_{k-1}}$\overset{m:1}\longrightarrow$ \type{A_{km-1}}&$k$
\\
\type{cAx/4} & $x^2+y^2+z^{2k-1}$ &$(1,3,2)$&\type{A_{2k-2}}$\overset{4:1}\longrightarrow$ \type{D_{2k+1}} &$2k-1$
\\
\type{cD/3} & $x^2+y^3+z^3$ &$(0,1,2)$&\type{D_4}$\overset{3:1}\longrightarrow$ \type{E_6} &$2$
\\
\type{cAx/2} &$x^2+y^2+z^{2k}$&$(0,1,1)$ & \type{A_{2k-1}}$\overset{2:1}\longrightarrow$ \type{D_{k+2}} &$2$ 
\\
\type{cD/2} & $x^2+y^2z+z^k$ &$(1,1,0)$ &\type{D_{k+1}}$\overset{2:1}\longrightarrow$ \type{D_{2k}} &$k$
\\
\type{cE/2} & $x^2+y^3+z^4$&$(1,0,1)$ &\type{E_6}$\overset{2:1}\longrightarrow$ \type{E_7} &$3$
\end{tabularx}
\end{center}
\end{scase}

M. Reid conjectured that an analog of \xref{thm:elephant} holds for any $K$-negative contraction of terminal threefolds (general elephant). The conjecture is very important in birational geometry. 
The following  theorem shows that this conjecture is true for extremal curve germs. Different parts of this theorem were proved in \cite{Mori:flip}, \cite{KM92} \cite{MP:cb1}, \cite{MP:cb3}.

\begin{theorem}
\label{thm:ge}
Let $(X,C\simeq\PP^1)$ be an extremal curve germ. Then a general member of the linear system $|-K_X|$
is normal and has only Du Val singularities.
\end{theorem}

Note that by the inversion of adjunction \cite[\S~3]{Sh:flips}, \cite[Ch.~17]{Utah} the Du Val property of a general member $D\in |-K_X|$
is equivalent to the plt property of the pair $(X,D)$. 

All the possibilities for general members of $|-K_X|$ have been classified,
see \cite{KM92} and \cite{MP:cb3}. 
Below we reproduce this classification in the case where $(X,C\simeq\PP^1)$ 
has only one non-Gorenstein point. 

\begin{theorem}
\label{thm:ge1}
Let $f:(X,C\simeq\PP^1)\to (Z,o)$ be an extremal curve germ. 
Assume that $(X,C)$ 
has only one non-Gorenstein point $P$.
Let $D\in |-K_{X}|$ be a general member.
If $f$ is birational, we let $D_Z:= f(D)$ which is a general member of $|-K_Z|$.
If $f$ is a $\QQ$-conic bundle, we let $D_{Z} :=\Spec_Z f_*\OOO_{D}$.
Then $D$ and $D_Z$ have only 
Du Val singularities and the morphism $f_D: D\to D_Z$ is birational and crepant. 
Moreover, only one of the following possibilities holds.

\begin{emptytheorem}[{\cite[(2.2.1), (2.2.1${}'$)]{KM92}, \cite[(1.2.3)--(1.2.6)]{MP:cb1}}]
\label{ge:simple}
We have $D\cap C=\{P\}$ and $f_D: D\to D_Z$ is an isomorphism.
In this case,
$D$ induces a general member of $|-K_{(X,P)}|$, and 
$\Delta(D)$ is described by \xref{local:elephant}.
\end{emptytheorem}

\begin{emptytheorem}[{\cite[(2.2.2)]{KM92}, {\cite[1.3.1]{MP:cb3}}}]
\label{ge:IC}
$P\in (X,C)$ is of type \typec{IC},
\begin{equation*}
\Delta (D,C):\qquad
\vcenter{
\xymatrix@R=-1pt{
&\circ \ar@{-}[d]
\\ 
{\underbrace{\circ -\cdots - \circ}_{m-3}} \ar@{-}[r]&\circ \ar@{-}[r]&\diamond, 
}}
\end{equation*}
where $m$, the index of $(X,P)$, is odd and $m\ge 5$.
\end{emptytheorem}

\begin{emptytheorem}[{\cite[(2.2.2${}'$)]{KM92}, {\cite[1.3.2]{MP:cb3}}}]
\label{ge:IIB}
$P\in (X,C)$ is of type \typec{IIB},
\begin{equation*}
\Delta (D,C):\qquad
\vcenter{
\xymatrix@R=8pt{
&&\circ \ar@{-}[d]
\\ 
\circ\ar@{-}[r]&\circ \ar@{-}[r]&\circ \ar@{-}[r]& \circ \ar@{-}[r]&\diamond.
}}
\end{equation*}
\end{emptytheorem}
\end{theorem}

In some cases of \xref{ge:simple} there are additional restrictions on the general member $D\in |-K_X|$.
For example, in the case where $f$ is birational and $(X,P)$ is of type \type{cAx/2},
the general $D\in |-K_X|$ is of type \type{D_4} \cite[4.8.5.7]{KM92}.
A lot of restrictions are imposed on imprimitive $\QQ$-conic bundles 
(see \xref{th:imprimitive}).

\subsection{}
Let us outline the main ideas of the proof in the case where $(X,C)$ has 
only one non-Gorenstein point. Thus, let $(X,C)$ be an extremal curve germ 
with a unique non-Gorenstein point $P$. 

\begin{slemma}[see {\cite[Theorem~7.3]{Mori:flip}}, 
{\cite[\S~7, 8.6.1]{MP:cb1}}]
\label{lemma-ge}
In the notation of \xref{thm:ge1} 
and with the symbols in 
Propositions \xref{prop:local-primitive} for primitive points and 
\xref{prop-imp-types} for imprimitive points,
we have.
\begin{enumerate}
\item
\label{lemma-ge1}
If $P\in (X,C)$ is of type \typec{IA}, \typec{IIA}, \typec{IA^\vee}, \typec{IIA^\vee},
\typec{ID^\vee}, or \typec{IE^\vee}, then for a general member $D\in |-K_X|$ 
we have $D\cap C=\{P\}$.
\item
\label{lemma-ge2}
If $P\in (X,C)$ is of type \typec{IC} or \typec{IIB}, then for a general member $S\in |-2K_X|$ 
we have $S\cap C=\{P\}$. Moreover, the pair $(X,\frac12 S)$ is klt.
\end{enumerate}
\end{slemma}

\begin{proof}[Sketch of the proof]
Consider the case where $P$ is of type \typec{IA}. 
Take $\psi:=x_2+\psi_\bullet$,
where $\psi_\bullet\in \CC\{x_1,\dots,x_4\}$ is a sufficiently general semi-invariant 
with $\wt(\psi_\bullet)\equiv \wt(x_2)$ and let $D:=\{\psi=0\}/\mumu_m$.
Then $D\cap C=\{P\}$ and $\wt(\psi)\equiv \wt(\varOmega)$, where 
$\varOmega=\Res(\phi^{-1}d x_1\wedge\dots \wedge d x_4)$.
Therefore, $\psi\varOmega^{-1} \in \OOO_X(-K_X)$ in a neighborhood of $P$.
Since $P$ is the only non-Gorenstein point,
this implies that $K_X+D$ is 
Cartier globally by \eqref{exact-Clcs}.
On the other hand, $D\cdot C=\frac 1m \ord{\psi}=\frac {a_2}m <1$
and $-K_X\cdot C<1$ (see \eqref{eq-grO-iP1} and \xref{corollary:gr-w}). Therefore, $K_X+D\equiv 0$.
Then Corollary \xref{cor-C-pa=0}\xref{cor-C-pa=0c} implies $K_X+D\sim 0$.
Other cases are treated similarly.
\end{proof}

Thus in the cases of \xref{lemma-ge}\xref{lemma-ge1} we are done.
Cases \typec{IC} and \typec{IIB} are much more delicate.
Rough idea of proof in these cases is to use surjectivity 
of the restriction map
\begin{equation*}
H^0(X,\OOO_X(-K_X)) \longrightarrow H^0(S,\OOO_S(-K_X))
\end{equation*} 
and extend a ``good'' member of $|-K_X|\bigr|_S$ to $X$. 

\subsection{}
Kawamata \cite{Kaw:Crep} had shown that 
Theorem \xref{thm:ge} in the flipping case is a sufficient condition for the existence of flips.
Indeed, 
applying a Bertini type arguments (see \cite[Corollary~2.33]{KM:book})
one can show that, for a general member $S\in |-2K_X|$, the pair $(X,\frac 12 S)$ is klt. 
Consider the double cover $(X^\flat,C^\flat)\to (X,C)$ branched over $S$.
Then $X^\flat$ has only canonical singularities (see \cite[5.20]{KM:book}, \cite[7.2]{Mori:flip})
and admits a flopping contraction of $C^\flat$.
Then the existence of flip for $(X,C)$ follows from the existence of flop for
$(X^\flat,C^\flat)$.

\subsection{}
As a corollary of Theorem \xref{thm:ge} we have the following fact which was conjectured by
V. Iskovskikh \cite{I96}.

\begin{scorollary}[{\cite{P97}}, {\cite{MP:cb1}}]
\label{base}
Let $f:(X,C)\to (Z,o)$ be a $\QQ$-conic bundle germ.
Then $(Z,o)$ is either smooth or a Du Val singularity of type \type{A}.
\end{scorollary}

The corollary has important applications in birational geometry of conic bundles
(see \cite{I96}, \cite{P17}). 

\section{Index two germs}
\label{sect:index2}
In this section we discuss extremal curve germs 
having index two points only. The methods are different from those used in other 
sections. Throughout this section we do not assume that the central 
curve of an extremal curve germ is irreducible.

\begin{proposition}[{\cite[4.6]{KM92}}]
\label{index:2}
Let $(X,C)$ be an extremal curve germ of index two. 
If $(X,C)$ is a $\QQ$-conic bundle germ, then we assume that
the base surface is smooth. Then we have the following.
\begin{enumerate}
\item \label{index:2a}
If $P$ is a point of index two, then $P$ is the only non-Gorenstein point, all the components of $C$ pass through $P$ and they do not
meet each other elsewhere.
\item \label{index:2b}
Each germ $(X,C_i)$ is of type \typec{IA} at $P$. 
\item \label{index:2c}
A general member $F\in |-K_X|$ satisfies $F\cap
C=\{P\}$ and has only Du Val singularity at $P$.
\end{enumerate}
\end{proposition}

\begin{proof}
By Lemma \xref{lemma-int-non-Gor} it is sufficient to show that every irreducible component of $C$ has at most one non-Gorenstein point. Assume the converse: a component $C_i\subset C$ contains two points $P$ and $Q$ of index two. If $\gr_{C_i}^0\not \simeq \OOO_C(-1)$, then $(X,C_i)$ is a $\QQ$-conic bundle germ by Corollary~\xref{cor-gr-omega-=0}\xref{cor-gr-omega-=0a}. In this case, $C=C_i$ and $(X,C_i)$ is primitive (because the base is smooth). This contradicts Corollary~\xref{cor-prop-grw=2-2-points-prim}. Thus $\gr_{C_i}^0\simeq \OOO_C(-1)$. Since the numbers $K_X\cdot C_i$, $w_P(0)$, and $w_Q(0)$ are strictly positive and contained in $\frac12 \ZZ$, we get a contradiction by \eqref{eq-grw-w}.

\xref{index:2b} follows from 
\xref{prop:local-primitive} (the case \typec{IB} is excluded
by Corollary~\xref{corollary:IB}).
\xref{index:2c} is proved as in Sect. \xref{sect:elephant}.
\end{proof}

First, we consider the birational case following \cite[\S 4]{KM92}.

\begin{theorem}[{\cite[4.7]{KM92}}]
\label{thm:index2}
Let $(X,C)$ be a birational extremal curve germ of index two. Let $P\in X$ be a non-Gorenstein point. Then a general member $H\in |\OOO_X|$ is normal and has only rational singularities. The following are the only possibilities  for the dual graph $\Delta(H,C)$, where $\overset{3}\circ \lin\underbrace{\circ\lin\cdots\lin\circ}_{n-2}\lin\overset{3}\circ$ should be replaced with $\overset{4}\circ$ if $n=1$.

The following is the only flipping case.

\begin{emptytheorem}
\label{index2flipping}
Then $C\simeq\PP^1$, the singularity $(X,P)$ is of type \type{cA/2}, $(H_Z,o)$ is of type $\frac1{2n+1}(1, 2n-1)$, 
and 

\begin{equation*}
\Delta(H,C):\qquad 
\bullet\lin\overset{3}\circ \lin\underbrace{\circ\lin\cdots\lin\circ}_{n-2} \lin\overset{3}\circ
\end{equation*}
\end{emptytheorem}

In the remaining cases $(X,C)$ is divisorial.
Then $(Z,o)$ is a \type{cDV} point and $(H_Z,o)$ is a Du Val singularity.
If we say that $(H_Z,o)$ is of type \type{A_0}, this means that it is smooth. 
\begin{longtable}{l|p{43pt}|l|c}
No. & $(X,P)$ & $(H_Z,o)$ &\multicolumn{1}{c}{$\Delta(H,C)$}
\\\hline
\endfirsthead
No. & $(X,P)$ & $(H_Z,o)$ &\multicolumn{1}{c}{$\Delta(H,C)$}
\\\pagebreak\hline\pagebreak
\endhead
\multicolumn{4}{c}{$C$ has one component}
\\\hline
\nom 
\label{KM:4.7.3.1.1} & \type{cA/2} & \type{A_1} 
&$
\circ\lin\bullet\lin\overset{3}\circ \lin\underbrace{\circ\lin\cdots\lin\circ}_{n-2}\lin\overset{3}\circ
$
\\ 
\nom 
\label{KM:4.7.3.1.2}& \type{cA/2}& \type{A_0} 
&$
\circ\lin\circ\lin\bullet\lin\overset{4}\circ
$
\\ \nom 
\label{KM:4.7.3.1.3}& \type{cA/2} & \type{A_2} 
&$
\vcenter{
\xymatrix@R=3pt@C=16pt{
\overset{3}\circ\ar@{-}[r]&\circ\ar@{-}[r]\ar@{-}[d]&\overset{3}\circ
\\
&\bullet&
}}
$
\\ 
\nom 
\label{KM:4.7.3.1.4} & \type{cA/2}& \type{A_0} 
&$
\vcenter{
\xymatrix@R=3pt@C=16pt{
\overset{3}\circ\ar@{-}[r]&\circ\ar@{-}[r]\ar@{-}[d]&\circ\ar@{-}[r]&\overset{3}\circ
\\
&\bullet&
}}
$
\\\hline
\multicolumn{4}{c}{$C$ has two components}
\\ \hline
\nom \label{KM:4.7.3.2.1}& \type{cA/2}& \type{A_m} 
&$
\bullet\lin\overset{3}\circ \lin\underbrace{\circ\lin\cdots\lin\circ}_{n-2}\lin\overset{3}\circ\lin\bullet
$
\\ \nom 
\label{KM:4.7.3.2.2}& \type{cA/2}& \type{A_0} 
&$
\circ\lin\bullet\lin\overset{3}\circ \lin\underbrace{\circ\lin\cdots\lin\circ}_{n-2}\lin\overset{3}\circ\lin\bullet
$
\\ \nom 
\label{KM:4.7.3.2.3}& \type{cA/2}& \type{A_1} 
&$
\vcenter{
\xymatrix
@R=-7pt@C=13pt
{
\bullet\ar@{-}[dr]
\\
&\overset{3}\circ\ar@{-}[r] & 
{\underbrace{\circ\lin\cdots\lin\circ}_{n-2}} &\overset{3}\circ\ar@{-}[l]
\\
\bullet\ar@{-}[ur]
}}
$
\\ \nom 
\label{KM:4.7.3.2.4}& \type{cA/2}& \type{A_0} 
&$
\vcenter{
\xymatrix@R=3pt@C=16pt{
\overset{3}\circ\ar@{-}[r]&\circ\ar@{-}[r]\ar@{-}[d]&\overset{3}\circ\ar@{-}[r]&\bullet
\\
&\bullet&
}}
$
\\\hline
\multicolumn{4}{c}{$C$ has three components}
\\ \hline
\nom \label{KM:4.7.3.3.1}& \type{cA/2}& \type{A_0} 
&$
\vcenter{
\xymatrix
@R=-7pt@C=13pt
{
\bullet\ar@{-}[dr]
\\
&\overset{3}\circ\ar@{-}[r] &{\underbrace{\circ\lin\cdots\lin\circ}_{n-2}}
&\overset{3}\circ\ar@{-}[l]&\bullet\ar@{-}[l]
\\
\bullet\ar@{-}[ur]
}}
$
\\\hline
\multicolumn{4}{c}{$C$ has one component}
\\\hline 
\nom \label{KM:4.7.4} & \type{cAx/2}, \type{cD/2} or \type{cE/2}
& \type{D_4} 
&$
\vcenter{
\xymatrix@R=10pt@C=16pt{
\circ\ar@{-}[r]&\overset{3}\circ\ar@{-}[r]&\circ\ar@{-}[r]&\bullet
\\
\circ\ar@{-}[ru]&&\circ\ar@{-}[lu]
}}
$
\\ \nom 
\label{KM:4.7.5} & \type{cD/2} or \type{cE/2}& \type{D_{n+4}} 
&$
\vcenter{
\xymatrix
@R=-7pt@C=13pt
{
\circ\ar@{-}[dr]&&&&\circ\ar@{-}[dl]
\\
&\circ\ar@{-}[r]
&{\underbrace{\circ\lin\cdots\lin\circ}_{n-1}}
&\overset 3 \circ\ar@{-}[l]
\\
\circ\ar@{-}[ur]&&&&\circ\ar@{-}[ul]&\bullet\ar@{-}[l]
}}
$
\\
&\multicolumn{3}{c}{where $n\ge 1$ and $n=1$ if $(X,P)$ is of type \type{cE/2}}

\\ \nom 
\label{KM:4.7.6} & \type{cE/2}& \type{E_6} 
&$
\vcenter{
\xymatrix@R=6pt@C=16pt{
\circ\ar@{-}[r]&\circ\ar@{-}[r]&\circ\ar@{-}[r]\ar@{-}[d]&\circ\ar@{-}[r]&\circ
\\
\bullet\ar@{-}[r]&\circ\ar@{-}[r]&\underset{3}\circ&
}}
$
\end{longtable}
\end{theorem}

Note that the singularities of $H$ are log terminal 
in all the cases except for \xref{KM:4.7.4}, \xref{KM:4.7.5}, \xref{KM:4.7.6}.
In the cases \xref{KM:4.7.4} and \xref{KM:4.7.5} the singularities of $H$ are log canonical.

\begin{stheorem}[{\cite[4.2]{KM92}}]
\label{flip-index=2}
In the notation of Theorem \xref{thm:index2} assume that $f$ is flipping
\textup(see \xref{index2flipping}\textup)
and let $(X,C)\dashrightarrow (X^+,C^+)$ be the corresponding flip.
Then the following hold.
\begin{enumerate}
\item 
\label{KM(4.2.1)} In appropriate
coordinates the point $(X\ni P)$
is given by 
\begin{equation*}
\left\{x_1x_2 +p(x_3^2, x_4) = 0\right\}/\mumu_2(1, 1,1, 0)
\end{equation*}
and $C$ is the $x_1$-axis.

\item 
\label{KM(4.2.2)} 
$X^+$ has at most one singular point which is isolated \type{cDV} with equation
$x_1x_2 +p(x_3, x_4) = 0$
and $C^+$ is the $x_1$-axis. 
\item 
\label{KM(4.2.3)} 
$(Z,o)$ is a rational triple point given by the $2\times 2$-minors
of the matrix
\begin{equation*}
\begin{pmatrix}
z_1& z_2& z_3
\\
z_2& z_5 &p(z_1,z_4)
\end{pmatrix}
\end{equation*}
\end{enumerate}
\end{stheorem}

The proofs use the following standard construction.

\begin{sconstruction}
\label{index2:constr1}
Let $C_i\subset C$ be the
irreducible
components of $C$. Since $X$ has only points of index one and two,
$m_i = -2K_X\cdot C_i$ is a positive integer. Let $E_i\subset X$ be the union of $m_i$ disjoint
discs transversal
to $C_i$ and let $E = \sum E_i$. Then $E \in | - 2K_X|$. Hence we can
take the corresponding
double cover $X'\to X$ branched over $E$. Here $X'$ has only
index one terminal singularities. Let $E' \subset X'$ be the preimage of $E$. The
natural map $E'\to E$ is an isomorphism. 
The Stein factorization induces the following diagram
\begin{equation*}
\xymatrix{
E\subset X\supset C\ar[d]^{f}&E'\subset X'\supset C'\ar[l]\ar[d]^{f'}
\\
D\subset Z\ni o& D'\subset Z'\ni o'\ar[l]
} 
\end{equation*}
where $D:= f(E)$ and $D':= f(E')$. Here $Z'\to Z$ is a double
cover branched over $D$. 
By construction, 
$f'$ is crepant with respect to $K_{X'}$ and the fibers of $f'$ have dimension $\le 1$.
Therefore, $Z'$ has 
\type{cDV} points only (if $f$ is divisorial, then $Z'$ has
a double curve).
\end{sconstruction}

\begin{proof}[Sketch of the proof of \xref{thm:index2} and \xref{flip-index=2}.]
The above construction defines a $\mumu_2$-action on $X'/Z'$ and the 
quotient is $X/Z$. The fixed point set of the action on $Z'$ is precisely 
$D'$. Since $Z'\ni o'$ is a \type{cDV} point, it is a hypersurface in $\CC^4$, thus it 
can be written down explicitly. This will enable us to get equations for $X$ and 
$Z$. We have an $\mumu_2$-equivariant embedding $(Z',o')\subset (\CC^4_{y_1,\dots, y_4},0)$,
and we may assume that the coordinates are eigenvectors and $y_1,\dots,y_j$
are those of weights $1$. Thus $D'=\{y_1 =\cdots=y_j = 0\}\cap Z'$. Hence $j = 1$ or $2$.
\begin{sclaim}
\begin{enumerate}
\item
If $D'$ is Cartier, then $f$ is divisorial and $D$ is singular
along $f(E)$, where $E$ is the $f$-exceptional divisor.
\item
If $D'$ is not Cartier, then $f$ is flipping, $D$ is smooth
and $C$ is irreducible.
\end{enumerate}
\end{sclaim}

\begin{proof}
If $j = 1$ then $D'$ is Cartier. In this case, $f$ must be divisorial. Indeed,
otherwise since $f'$ is an isomorphism
outside the origin and $E'$ is $f'$-ample, $D'$ cannot be Cartier. 
Hence $f$ contracts an exceptional divisor $E \subset X$. Then for a general fiber $l$ of 
$E$ we have $K_X\cdot l=-1$. Hence $E \cdot F = 2$. Therefore, $D$ has a double 
curve along the image of $E$ and is smooth elsewhere. If $E$ is chosen 
generically, then $D$ has an ordinary double curve along the image of $E$.

Assume that $j = 2$. Then $\{y_1 = y_2 = 0\}$ must be contained in $Z'$. 
Furthermore, $D'$ is irreducible and this implies
that $C$ is irreducible. 
Then $E \to D$ is an isomorphism outside the origin, in 
fact, it  turns out to be an isomorphism. In particular, $D$ is smooth. 
This implies that $f$ is flipping.
\end{proof}

First consider the flipping case.
Since $\{y_1 =y_2 = 0\} \subset Z'$, the equation of $Z'$ can be written in the form
$y_1\phi_1 + y_2\phi_2 = 0$. If $\wt(\phi_1) = \wt(\phi_2) = 1$,
then $y_1\phi_1 + y_2\phi_2 \in (y_1,y_2)^2$, which
implies that $Z'$ is singular along $\{y_1 =y_2 = 0\}$. This is impossible. Thus
$\wt(\phi_1) = \wt(\phi_2) = 1$. Since $Z'$ is a double point, either $\phi_1$
or $\phi_2$ must contain a
linear term. Assume that $\phi_1$ contains $y_j$. By wt reasons $j = 3$ or $4$. 
Now we can rewrite the equation in the following form
\begin{equation*}
y_1y_3+y_2p(y_2^2,y_4)=0.
\end{equation*}
With this explicit equation we can easily compute everything. The variety $X$ is obtained by blowing up $\{y_2 = y_3 = 0\}$ and taking quotient by the group action. This gives us one singular point with the required equation. The flipped variety $X^+$ is obtained by blowing up $\{y_1 = y_2 = 0\}$ and taking quotient by the group action. To get equations for $(Z,o)$, we note that the invariants of the $\mumu_2$-action on $\CC\{y_1,\dots, y_4\}$ are
\begin{equation*}
z_1=y_2^2,\quad
z_2=y_1y_2,\quad
z_3=y_3,\quad
z_4=y_4,\quad
z_5=y_1^2.
\end{equation*}
We get exactly the equations given by the minors of the matrix in the assertion of the theorem. 
A hyperplane section given by $z_4 =cz_1$. 

Now consider the divisorial case.
Then $D' \subset Z'$ is Cartier and
the $\mumu_2$-action is given by $\wt(y)= (0,0,0,1)$.
Let $D'$ be given by $y_4 = \psi(y_1,y_2,y_3) = 0$.
Thus we can
write the equation of $Z'$ in the form
\begin{equation*}
y_4^2\phi(y_1,\dots,y_4)+\psi(y_1,y_2,y_3)=0.
\end{equation*}
Since $f'$ is crepant, $Z'$ cannot be smooth, in particular, $\mult_0(\psi) \ge 2$.
The equation of $Z$ is now given by
\begin{equation}
\label{index2equationZ}
t\phi(y_1,y_2,y_3,t)+\psi(y_1,y_2,y_3)=0,\quad (t=y_4^2).
\end{equation}
In particular, this shows that $(Z,o)$ is an (isolated) \type{cDV} point
(cf. \xref{theorem-main-Q-Cartier-i}).
Now the proof proceeds by a careful analysis of the equations.
See \cite[\S~4]{KM92} for details.
\end{proof}

\subsection{}
Now we consider $\QQ$-conic bundles. 
The case of singular base surface is easy:
\begin{sproposition}[{\cite[\S~3]{P97}}, {\cite{MP:cb2}}]
A $\QQ$-conic bundle of index two over a
singular base is either of type \xref{item-main-th-impr-barm=1} or toroidal \xref{ex-toric}.
\end{sproposition}
Index two $\QQ$-conic bundles over a smooth base were classified in \cite[\S~3]{P97} 
and \cite{MP:cb1}. Similar to birational case these are quotients of some elliptic fibrations by an
involution. 
On the other hand, one can note that there exists an embedding to a 
relative weighted projective space:

\begin{theorem}
\label{th-index=2}
Let $f: (X,C)\to (Z,o)$ be a $\QQ$-conic bundle germ of index
two. Assume that $(Z,o)$ is smooth. Fix an isomorphism $(Z,o)\simeq
(\CC^2,0)$. Then there is an embedding
\begin{equation}
\label{eq-diag-last-2}
\vcenter{
\xymatrix{X \ar@{^{(}->}[r] \ar[rd]_{f}& \PP(1,1,1,2)\times \CC^2
\ar[d]^{p}
\\
&\CC^2}}
\end{equation}
such that $X$ is given by two equations
\begin{equation}
\label{eq-eq-index2}
\begin{array}{l}
q_1(y_1,y_2,y_3)-\psi_1(y_1,\dots,y_4;u,v)=0,
\\
q_2(y_1,y_2,y_3)-\psi_2(y_1,\dots,y_4;u,v)=0,
\end{array}
\end{equation}
where $\psi_i$ and $q_i$ are weighted quadratic in $y_1,\dots,y_4$
with respect to $\wt(y_1,\dots,y_4)=(1,1,1,2)$ and
$\psi_i(y_1,\dots,y_4;0,0)=0$. The only non-Gorenstein point of $X$
is $(0,0,0,1; 0,0)$. Up to projective transformations, the
following are the possibilities for $q_1$ and $q_2$:
\par\medskip\noindent
\begin{tabularx}{1\textwidth}{l|X|X|X}
\hline
{\rm no.}& $q_1$ & $q_2$ & $f^{-1}(o)$
\\[5pt]\hline
\nom\label{cla-index-2-4} &
$y_1^2-y_2^2$ & $y_1y_2-y_3^2$ & 
$C_1+C_2+C_3+C_4$
\\
\nom\label{cla-index-2-22}&
$y_1y_2$ & $(y_1 +y_2)y_3$ & $2C_1+C_2+C_3$
\\
\nom\label{cla-index-2-1-3} & 
$y_1y_2 - y_3^2$ & $y_1y_3$ & $3C_1+C_2$
\\
\nom\label{cla-index-2-2-2} & 
$y_1^2-y_2^2$ & $y_3^2$ & $2C_1+2C_2$
\\
\nom\label{cla-index-4a} & 
$y_1y_2 - y_3^2$ & $y_1^2$ & $4C_1$
\\
\nom\label{cla-index-4b} & 
$y_1^2$ & $y_2^2$ & $4C_1$
\end{tabularx}
\par\medskip\noindent
Conversely, if $X\subset\PP(1,1,1,2)\times \CC^2$ is given by
equations of the form \eqref{eq-eq-index2} and singularities of $X$
are terminal, then the projection $f: (X,f^{-1}(0)_{\red}) \to
(\CC^2,0)$ is a $\QQ$-conic bundle of index two.
\end{theorem}

\begin{sremark}
A general member $H\in |\OOO_X|$ is normal in the case \xref{cla-index-4a}
and non-normal in the case \xref{cla-index-4b}.
\end{sremark}

\begin{proof}[Sketch of the proof]
First we prove the last statement. By our assumption $X$ has only terminal
singularities. Then $X$ does not contain the surface
$\{y_1=y_2=y_3=0\} = \Sing(\PP\times \CC^2)$ (otherwise both $\psi_1$
and $\psi_2$ do not depend on $y_4$). By the adjunction formula,
$K_X=-L|_X$, where $L$ is a Weil divisor on $\PP\times \CC^2$ such
that the restriction $L|_{\PP}$ is $\OOO_{\PP}(1)$. Therefore, $X\to
\CC^2$ is a $\QQ$-conic bundle. It is easy to see that the only
non-Gorenstein point of $X$ is $(0,0,0,1;0,0)$ and it is of index
two.

Now let $f: (X,C)\to (Z,o)\simeq (\CC^2,0)$ be a $\QQ$-conic
bundle germ of index two. Let $P\in X$ be a point of index two. 
Let $\pi: (X^\sharp,P^\sharp)\to (X,P)$ be the 
index-one cover.
We need the following lemma.

\begin{slemma}[{\cite[12.1.9]{MP:cb1}}]
\label{lemma-12-fiber-new}
Let $F^\sharp=\pi^{-1} (F)_{\red}$ 
be the pull-back of $F$.
Let $\Gamma:=f^{-1}(o)$ be the scheme fiber and let 
$\Gamma^\sharp=\pi^{-1} (\Gamma)$. Then we have
\begin{equation*}
\OOO_{F^\sharp \cap \Gamma^\sharp} \simeq \CC[x,y]/(xy,\, x^2+y^2).
\end{equation*}
Furthermore, the $\mumu_2$-action is given by $\wt(x,y) \equiv (1,1) \mod 2$.
\end{slemma}

Using this lemma one can apply 
arguments of \cite[pp. 631--633]{Mori:ci}
to get the desired embedding $X\subset \PP(1,1,1,2)\times Z$
considering the graded anti-canonical $\OOO_Z$-algebra
\begin{equation*}
\RRR:=\bigoplus_{i\ge 0}\RRR_i,\quad\mbox{where}
\quad \RRR_i:= H^0(\OOO_X(-iK_X)).
\end{equation*}
We sketch the main idea.

Let $w$ be a local generator of $\OOO_{X^\sharp} (-K_{X})$ at
$P^\sharp$, let $u$, $v$ 
be coordinates on $Z=\CC^2$,
and let $z=0$ be the local equation of $F^\sharp$ in 
$(X^\sharp, P^\sharp)$.
Using the vanishing of $H^1(\OOO_X(-K_X))$ for $i>0$
and the exact sequence
\begin{equation*}
0 \to \OOO_X(-(i-1)K_X) \to \OOO_X(-iK_X)\to
\OOO_F(-iK_X)\to 0
\end{equation*}
one can see 
\begin{equation*}
\RRR_i/(zw)\RRR_{i-1}\simeq H^0(\OOO_F(-iK_X)), \quad i>0.
\end{equation*}
Therefore,
\begin{equation*}
\RRR_i/(zw)\RRR_{i-1}+(u,\, v)\RRR_i
=\bigl(\OOO_{F^\sharp\cap \Gamma^\sharp}(-iK_X)\bigr)^{\mumu_2}.
\end{equation*}
By Lemma \xref{lemma-12-fiber-new} we have an embedding
\begin{equation*}
\RRR/(zw,\, u,\, v)\RRR
\hookrightarrow
\left(\CC[x,\, y,\, w]/(xy,\, x^2+y^2)\right)^{\mumu_2}.\end{equation*}
Using $\RRR_0/(u,v)\RRR_0=\CC$, one can easily see that
\begin{equation*}
\RRR/(zw,\, u,\, v)\RRR=
\CC[y_1,y_2,y_4]/(y_1y_2,\, y_1^2+y_2^2),
\end{equation*}
where $y_1=xw$, $y_2=yw$, $y_4=w^2$.
Put $y_3:=zw$. Then similar to \cite[pp. 631--633]{Mori:ci}
we obtain
\begin{equation*}
\RRR \simeq 
\OOO_Z[y_1,y_2,y_3,y_4]/\III,
\end{equation*}
where $\III$ is generated by the following regular
sequence
\begin{equation*}
\begin{array}{ll}
y_1y_2+y_3\ell_1(y_1,\dots,y_3)&+\psi_1(y_1,\dots,y_4;u,v),
\\[5pt]
y_1^2+y_2^2+y_3\ell_2(y_1,\dots,y_3)&+\psi_2(y_1,\dots,y_4;u,v)
\end{array}
\end{equation*}
with $\psi_i(y_1,\dots,y_4;0,0)=0$.
\end{proof}
Note also that the construction \xref{index2:constr1}
in the $\QQ$-conic bundle case produces an elliptic fibration.
It can be used for classification (see \cite[\S~3]{P97}).

\begin{sexample}
Let $X\subset\PP(1,1,1,2)\times\CC^2_{u,v}$
is given by the equations
\begin{eqnarray*}
y_1y_2&=&(au+bu^2+cuv)y_4,
\\
(y_1+y_2+y_3)y_3&=&vy_4,
\end{eqnarray*}
where $a, b, c \in \CC$ are constants.
It is easy to check that the projection $X\to \CC^2$ is a $\QQ$-conic bundle 
as in \xref{cla-index-2-4}. The only singular point 
is of type \type{cA/2}.
If $a\ne 0$, then this point is a cyclic quotient of type
$\frac{1}{2}(1,1,1)$.
\end{sexample}

\begin{sexample}
Let $X\subset\PP(1,1,1,2)\times\CC^2_{u,v}$
is given by the equations
\begin{eqnarray*}
y_1^2&=&uy_3^2+vy_4
\\
y_2^2&=&uy_4+vy_3^2
\end{eqnarray*}
Then the projection $X\to \CC^2$ is a $\QQ$-conic bundle of type \xref{cla-index-4b} 
containing
one singular point of type $\frac{1}{2}(1,1,1)$ and two ordinary double 
points.
\end{sexample}

More examples are given in \cite[\S~7 and Remark~6.7.1]{MP:IA},  and \cite[\S~3]{P97}. 
It can be shown \cite[\S 7]{MP:IA} that every type of terminal index two singularity
can occur on some index two $\QQ$-conic bundle 
as in~\xref{cla-index-4a} or \xref{cla-index-4b}.

\section{Locally imprimitive germs}
\label{sect:imprimitive}
In this section we collect the results concerning extremal curve germs 
with a locally imprimitive point.
Note that in this case  the imprimitive point is unique 
and the splitting cover  is locally primitive along 
arbitrary irreducible component of the central curve (see
Corollary \ref{prop-cyclic-quo}). Moreover, one can show that the imprimitive point 
is the only  
non-Gorenstein point, see
\cite[Th.~6.7, 9.4]{Mori:flip} and \cite[\S~7]{MP:cb1}.

The following theorem summarizes the results contained in \cite{Mori:flip},
{\cite{KM92}}, {\cite{MP:cb1}}, {\cite{MP:IA}}.

\begin{theorem}
\label{th:imprimitive}
Let $f: (X, C\simeq \PP^1)\to (Z,o)$ be an extremal curve germ such that
$(X,C)$ is locally imprimitive. 
Let $P\in X$ be the imprimitive
point and let $m$, $s$ and $\bar m$ be its index, splitting degree
and subindex, respectively. In this case, $P$ is the only
non-Gorenstein point and $X$ has at most one type~\typec{III}
point. Then one of the following holds.

\begin{emptytheorem}[{\cite[1.2.3]{MP:cb1}}]
\label{item-main-th-impr-barm=2-s=4}
$f$ is a $\QQ$-conic bundle, $(X,C)$ is of type~\typec{IE^\vee} at $P$, $(Z,o)$
is Du Val of type~\type{A_3}, $X$ has a cyclic quotient singularity $P$
of type $\frac18(5,1,3)$ and has no other singular points.
Furthermore, $(X,C)$ is the quotient of the index-two $\QQ$-conic
bundle germ given by the following two equations in
$\PP(1,1,1,2)_{y_1,\dots,y_4}\times \CC^2_{u,v}$
\begin{equation*}
\label{eq-imp-exc-8-eq-a}
\begin{array}{lll}
y_1^2-y_2^2&=&u \psi_1(y_1,\dots,y_4;u,v)+v\psi_2(y_1,\dots,y_4;u,v),
\\[5pt]
y_1y_2-y_3^2&=&u \psi_3(y_1,\dots,y_4;u,v)+v\psi_4(y_1,\dots,y_4;u,v)
\end{array}
\end{equation*}
by $\mumu_{4}$-action:
\begin{equation*}
y_1\mapsto -\ii y_1,\quad y_2\mapsto \ii y_2,\quad y_3\mapsto - y_3,\quad 
y_4\mapsto
\ii y_4,\quad u\mapsto \ii u,\quad v\mapsto -\ii v
\end{equation*}
\textup(as an example one can take $\psi_1=\psi_4=y_4$, $\psi_2=\psi_3=0$\textup).
\end{emptytheorem}

\begin{emptytheorem}[{\cite[1.2.4]{MP:cb1}}]
\label{item-main-th-impr-barm=1}
$f$ is a $\QQ$-conic bundle, $(X,C)$ is of type~\typec{ID^\vee} at $P$, $(Z,o)$
is Du Val of type~\type{A_1}, $(X,C)$ is a quotient of a Gorenstein conic
bundle given by the following equation in $\PP^2_{y_1,y_2,y_3}\times
\CC^2_{u,v}$
\begin{equation*}
y_1^2+y_2^2+\psi(u,v)y_3^2=0, \qquad \psi(u,v)\in\CC\{u^2,\, v^2,\, uv\}
\end{equation*}
by $\mumu_{2}$-action:
\begin{equation*}
u\mapsto -u,\quad v\mapsto -v,\quad y_1\mapsto -y_1,\quad y_2\mapsto y_2,\quad 
y_3\mapsto y_3.
\end{equation*}
Here $\psi(u,v)$ has no multiple factors. In this case, $(X,P)$ is
the only singular point and it is of type~\type{cA/2} or \type{cAx/2}.
\end{emptytheorem}

\begin{emptytheorem}[{\cite[1.2.5]{MP:cb1}}]
\label{item-main-th-impr-barm=2-s=2-cycl}
$f$ is a $\QQ$-conic bundle, $(X,C)$ is of type~\typec{IA^\vee} at $P$ with $\bar 
m=2$, $s=2$, $(Z,o)$
is Du Val of type~\type{A_1}, $(X,P)$ is a cyclic quotient singularity
of type $\frac{1}{4}(1,1,3)$,
and $(X,C)$ is the quotient of the index-two $\QQ$-conic bundle germ
given by the following two equations in
$\PP(1,1,1,2)_{y_1,\dots,y_4}\times \CC^2_{u,v}$
\begin{equation*}
\begin{array}{lll}
y_1^2-y_2^2&=&u \psi_1(y_1,\dots,y_4;u,v)+v\psi_2(y_1,\dots,y_4;u,v),
\\[5pt]
y_3^2&=&u \psi_3(y_1,\dots,y_4;u,v)+v\psi_4(y_1,\dots,y_4;u,v)
\end{array}
\end{equation*}
by $\mumu_{2}$-action:
\begin{equation*}
y_1\mapsto y_1,\quad y_2\mapsto - y_2,\quad y_3\mapsto y_3,\quad y_4\mapsto -
y_4,\quad u\mapsto - u,\quad v\mapsto - v.
\end{equation*}
As an example one can take $\psi_1=\psi_4=y_4$, $\psi_2=0$, $\psi_3=uy_2^2+\lambda y_1y_2$,
where $\lambda$ is a constant.
If $\lambda\neq 0$, then $P$ is the only singular point. If 
$\lambda=0$, then $X$ 
has also a type~\typec{III} point.
\end{emptytheorem}

\begin{emptytheorem}[{\cite[1.2.6]{MP:cb1}}]
\label{item-main-th-impr-barm=2-s=2-cAx/4}
$f$ is a $\QQ$-conic bundle, $(X,C)$ is of type~\typec{II^\vee} at $P$, $(Z,o)$
is Du Val of type~\type{A_1}, 
and $(X,C)$ is the quotient of the same form as in
\xref{item-main-th-impr-barm=2-s=2-cycl}.
As an example one can take $\psi_1=u^2y_4$, $\psi_2=\psi_4=y_4$, 
$\psi_3=uy_2^2+\lambda y_1y_2$, 
where $\lambda$ is a constant.
\end{emptytheorem}

\begin{emptytheorem}[{\cite[Theorem~4.11.2]{KM92}}]
\label{imprimitiveII}
$f$ is divisorial, $(X,C)$ is of type~\typec{II^\vee} at $P$, a general member
$H\in |\OOO_X|$ is normal. 
The graph $\Delta(H,C)$ is of the form 
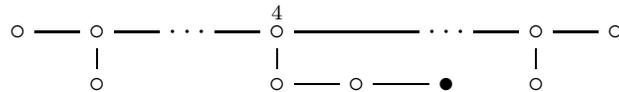
\begin{equation*}
\xymatrix@R=3pt@C=17pt{
\circ\ar@{-}[r]&\circ\ar@{-}[r]&\cdots\ar@{-}[r]&\overset{4}\circ\ar@{-}[rr]
&&\cdots
\ar@{-}[r]&\circ\ar@{-}[r]&\circ
\\
&\circ\ar@{-}[u]&&\circ\ar@{-}[u]\ar@{-}[r]&\circ\ar@{-}[r]&\bullet&\circ\ar@{-}
[u]
} 
\end{equation*}
In this case $(X,C)$ is a quotient of an index two divisorial curve germ $(\bar X,\bar C)$ 
by $\mumu_2$ that acts freely outside $P$ and switches two components of $\bar C$.
The point $(Z,o)$ is terminal of index two given by 
\begin{equation*}
\{t\phi(y_1,t)+y_3^2- y_2^2=0\}/\mumu_2(1,1,0,1)
\end{equation*}
\textup(cf. \eqref{index2equationZ}\textup) where the image of the exceptional divisor 
is the curve $\{y_2=y_3=t=0\}/\mumu_2$. 
\end{emptytheorem}

\begin{emptytheorem}[{\cite[Theorem~1.9]{MP:IA}}]
\label{imprimitiveIA}
$f$ is birational, a general member
$H\in |\OOO_X|$ is normal and has only log terminal singularities of 
class~\type{T} \textup(see \xref{typeT} below\textup).
The graph $\Delta(H,C)$ is of the form 
\begin{equation}
\label{imprimitiveIA-graph}
\vcenter{
\xymatrix@R=5pt{\overset{c_1}\circ\ar@{-}[r]&\overset{c_2}\circ\ar@{-}[r]
&\cdots\ar@{}[r]&\cdots\ar@{-}[r]&
\overset{c_r}\circ\ar@{-}[r]&\cdots\ar@{-}[r]&
\overset{c_n}\circ
\\
&\circ\ar@{-}[r]&\circ\ar@{-}[r]&\cdots\ar@{-}[r]&\bullet\ar@{-}[u]
}}
\end{equation}
Here $r\neq 1,\, n$ and the chain $[c_1,\dots,c_n]$ corresponds to the non-Du 
Val singularity 
$(H,P)$ of class~\type{T}. The chain of $(-2)$-vertices in the bottom line 
corresponds to a Du Val
point $(H,Q)$. It is possible that this chain is empty \textup(i.e., $(H,Q)$ is 
smooth\textup).
The germ $(X,C)$ is of type~\typec{IA^\vee} and $C^\sharp$ explained in 
\xref{splitting-degree} is reducible. 
The contraction $f$ is divisorial or flipping according as $(H'_{Z'},o')$ 
explained in \xref{divisorial-or-flipping} is Gorenstein or not.
\end{emptytheorem}
\end{theorem}

\begin{scase} 
\label{typeT}
Recall that a surface log terminal singularity $(H,P)$ is called a singularity of class \type{T}, if it 
admits a one-parameter smoothing $\{H_t\}$, $H_0=H$ whose total space $X=\cup H_t$ is $\QQ$-Gorenstein
\cite{LW86}, \cite{KSh88}. 
By the inversion of adjunction 
this total space must be terminal.
Any singularity of class \type{T} is either Du Val or a cyclic quotient 
\begin{equation*}
\frac{1}{m^2d}(1,\, mdt-1),\qquad \gcd(m,t)=1. 
\end{equation*} 
There is an explicit characterization of such singularities in terms of minimal resolutions, 
see \cite[\S~3]{KSh88} for details. 
\end{scase}

\subsection{}
The rough idea of the proof of Theorem \xref{th:imprimitive} is to apply the construction \eqref{eq:base-cover}.
Then $(X,C)$ can be viewed as a quotient of an extremal curve germ $(X',C')$
with reducible central fiber by $\mumu_s$.
In the case \typec{ID^\vee} we have $\bar m=1$. Hence, $(X',C')$ is a Gorenstein conic bundle
germ \xref{prop:Gor}\xref{prop:Gor-cb}. Then it is easy to write
down the action explicitly \cite[\S~2]{P97}. Similarly, in the cases \typec{IE^\vee}
and \typec{II^\vee} we have $\bar m=2$. Then $(X',C')$ is an extremal 
curve germ of index two and we can apply the results of Sect. \xref{sect:index2}.
The case \typec{IC^\vee} does not occur \cite[Th.~6.1(i)]{Mori:flip}, \cite[7.3]{MP:cb1}.

\subsection{}
Consider the case \typec{IA^\vee}.
We need
the following helpful observation which allows to study a general hyperplane section
$H\in |\OOO_X|$. It will also be used below in the case \typec{IA}.

\begin{slemma}\label{lemma:lc}
Let $(Z,o)$ be a normal threefold singularity and let 
$D_Z\in |-K_X|$ be a general member. Assume that $(D_Z,o)$
is a Du Val singularity of type \type{A}. Then for 
a general hyperplane section $H_Z$, the pair $(X, H_Z+D_Z)$
is lc. In particular, $(H_Z, o)$ is a cyclic quotient singularity.
\end{slemma}
\begin{proof}
Clearly,
$H_Z\cap D_Z$ is general hyperplane section of $(D_Z,o)$ and so 
$H_Z\cap D_Z=\Gamma_1+\Gamma_2$ for some irreducible curves $\Gamma_i$
such that the pair $(D_Z, \Gamma_1+\Gamma_2)$ is lc.
By the inversion of adjunction so is the pair $(Z,D_Z+H_Z)$ 
\cite[\S 3]{Sh:flips}, 
\cite{Kawakita2007}.
Hence $(H_Z,\Gamma_1+\Gamma_2)$ is lc and 
$(H_Z,o)$ is a cyclic quotient singularity 
(see, e.g., \cite[Ch. 3]{Utah}).
\end{proof}

\begin{sproposition}
\label{prop:lc}
Let $f: (X,C)\to (Z,o)$ be an extremal curve germ 
\textup($C$ is not necessarily irreducible\textup). 
Let $D\in |{-}K_X|$ and $H\in |\OOO_X|$ be general members. 
Let $\Lambda$ be the non-normal locus of $H$ and let $\nu: H^\n\to H$ be the normalization 
\textup(if $H$ is normal we put $\Lambda=\emptyset$ and 
$\nu=\id$\textup).

Assume that $D\cap C$ is a point $P$ such that
$(D,P)$ is a Du Val singularity of type~\type{A}.
Then the pairs $(X, \, D+H)$ and $(H^\n, \, \nu^{-1}(D)+\nu^{-1}(\Lambda))$ are log canonical. 
In particular, $H$ has only normal crossings in codimension one.
If $f$ is birational, then the pair $(Z,D_Z+H_Z)$ is also log canonical, where $D_Z=f(D)\in |{-}K_Z|$ and $H_Z:=f(H)\in |\OOO_Z|$.
In this case, $(H_Z,o)$ is a cyclic quotient singularity. 
\end{sproposition}

\begin{proof}
First we consider the case where $f$ is birational.
Then $(D_Z,o)\simeq (D,P)$ is a Du Val singularity of type~\type{A}.
By Lemma \xref{lemma:lc} the pair $(X, H_Z+D_Z)$
is lc. 
Take $H:=f^*H_Z$. Then $K_X+D+H=f^*(K_Z+D_Z+H_Z)$, i.e., the contraction 
$f$ is $K_X+D+H$-crepant. Hence the pair $(X,D+H)$ is lc
and so is the pair $(H^\n, \, \nu^{-1}(D)+\nu^{-1}(\Lambda))$ again
by the inversion of adjunction. 

Now consider the case where $Z$ is a surface.
First we claim that $(X, \, D+H)$ is lc near $D$.
Consider the restriction 
$\varphi=f_D: (D,P) \to (Z,o)$. 
Let $\Xi\subset Z\simeq \CC^2$ be the branch divisor of $\varphi$.
By the Hurwitz formula we can write 
$K_{D}=\varphi^*\bigl(K_Z+\frac12 \Xi\bigr)$.
Hence,
\begin{equation*}
K_{D}+H|_D=\varphi^*\Bigl(K_Z+\frac12 \Xi+H_Z\Bigr).
\end{equation*} 
Using this and the inversion of adjunction we get the following equivalences:
$(X, \, D+H)$ is lc near $D$ $\Longleftrightarrow$
$(D, H|_D=\varphi^*H_Z)$ is lc $\Longleftrightarrow$
$(Z=\CC^2, \frac 12 \Xi+H_Z)$ is lc.
Thus it is sufficient to show that $(Z, \frac 12 \Xi+H_Z)$ is lc.
Let $\xi(u,v)=0$ be the equation of $\Xi\subset \CC^2$. 
Then $(D,P)$ is given by the equation
$w^2=\xi(u,v)$ in $\CC^3_{u,v,w}$. 
By the classification of Du Val singularities we can choose 
coordinates $u$, $v$ so that 
$\xi=u^2+v^{n+1}$. 
Take $H_Z:=\{v-u=0\}$.
Then $\ord_0 \xi(u,v)|_{H_Z}=2$.
By the inversion of adjunction the pair $(Z,H_Z+\frac12 \Xi)$ is lc. 
Thus we have shown that $(X,D+H)$ is lc near $D$.
Assume that $(X,D+H)$ is not lc at some point $Q\in C$.
By the above, $Q\notin D$. 
Note that $H$ is smooth outside $C$ by Bertini's theorem.
If $H$ is normal, then we have an immediate contradiction by 
a connectedness result \cite[Th.~6.9]{Sh:flips} 
applied to $(H,D|_H)$.
If $H$ is not normal, we can apply the same result on the normalization. 
\end{proof}

\begin{scase}
\label{pf:IAdual:normality}
We claim that $H$ is normal.
Assume the converse, i.e. $H$ is singular along $C$. 
The lemma above implies that in our situation $C$ is the minimal 
log canonical center of $(X,H)$ \cite{Kaw97}. 
Now let $\tau: (X',C')\to (X,C)$ be the torsion free cover \xref{torsion-free-cover}
and let $H':=\tau^*H$. Then the pair $(X',C')$ is log canonical 
and $C'$ is its minimal log canonical center \cite[20.4]{Utah}. 
Since the minimal log canonical center is normal \cite{Kaw97},
we conclude that $C'$ is irreducible. This contradicts imprimitivity of $(X,C)$ at $P$.
\end{scase}

\begin{scase}
\label{pf:IAdual:normal}
Thus $H$ is normal and then $P$ is the only log canonical center of the pair 
$(X,H+D)$. This implies that the pair $(X,H)$ is plt.
Since $H$ is a Cartier divisor, the singularities of $H$ are of class \type{T}
\textup(see \xref{typeT}\textup). 
This gives very strong restriction to the dual graph of the minimal resolution.
If $(X,C)$ is a $\QQ$-conic bundle germ, then 
using completely combinatorial techniques one can show that for 
$\Delta(H,C)$ there is only one possibility (cf. \cite{P04:s}): 
\begin{equation*}
\xymatrix{\overset4\circ\ar@{-}[r]&\bullet\ar@{-}[r]&\circ\ar@{-}[r]&\circ\ar@{-}[r]&\circ} 
\end{equation*} 
But in this case the pair $(H,C)$ is plt, hence so is $(H',C')$, 
whence $C'$ cannot split and this case does not occur.
In the birational case we obtain \eqref{imprimitiveIA-graph}. 
Since $(H,C)$ cannot be plt as above, we have $r\neq 1,\, n$.
This concludes the explanation of the proof of Theorem \xref{th:imprimitive}.
\end{scase}

\subsection{}\label{splitting-degree}
To decide whether an extremal curve germ $(X,C)$ is locally imprimitive at $P$, 
one needs to compute the inverse image $C^\sharp$ of $C$ in the index-one cover 
$(X^\sharp,P^\sharp)$ which can be computed within $H^\sharp$ the pull back of 
$H$ as in \xref{ex:IAdual}, once the diagram like \eqref{imprimitiveIA-graph} is 
exhibited. Indeed, $P$ is imprimitive if and only if the splitting degree $s>1$, which is 
equal to the number of irreducible components of $C^\sharp$, see \xref{splitting}.

\subsection{}
\label{divisorial-or-flipping}
To distinguish divisorial and flipping contractions in the case \xref{imprimitiveIA}
one can use the following arguments. 
Let $f':(X',C')\to (Z',o')$ be the torsion free cover \eqref{eq:base-cover}.
By \xref{theorem-main-Q-Cartier-i}
the germ $(X,C\simeq \PP^1)$ is divisorial if and only if the point $(Z,o)$ is terminal
and if and only if the point $(Z',o')$ is terminal of index one (i.e. isolated \type{cDV}).
Note that in our case $(H_Z,o)$ is a cyclic quotient singularity
and so is its pull-back $(H_Z',o')$.
Hence the divisoriality of $(X,C)$ is equivalent to that $(H_Z',o')$ is 
Gorenstein, that is, Du Val singularity in our case.
Once $(H,C\simeq \PP^1)$ is given, one can find its splitting cover $(H',C')$
and so the surface germ $(H_Z',o')$ can be computed.

\begin{example}
\label{ex:IAdual}
Consider the quotient surface singularity 
\begin{equation*}
(H,P)=(\CC^2_{u,v},0)/\mumu_{m^2}(1,m-1),\qquad m\ge 3.
\end{equation*} 
It is of class \type{T} and for its index-one cover we have 
\begin{equation*}
(H^\sharp,P^\sharp)=(\CC^2,0)/\mumu_{m}(1,m-1).
\end{equation*} 
Hence it is Du Val of type \type{A_{m-1}}. Consider the $\mumu_{m}$-equivariant curve
\begin{equation*}
C^\sharp=\{u^{m-2}-v^{m+2}=0\}/\mumu_{m}\subset H^\sharp
\end{equation*} 
and $C=C^\sharp/\mumu_{m}$.
Then $C^\sharp$ is irreducible (resp. has two irreducible components)
if $m$ is odd (resp. $m$ is even)
and it is easy to see that $C$ is smooth.
Now consider the weighted 
$\frac1{m^2}(1, m-1)$-blowup of $(H,P)$.
In the chart $v\neq 0$ the origin is a Du Val point $\CC^2/\mumu_{m-1}(-1,m^2)$ 
of type \type{A_{m-2}}, the exceptional divisor $\Lambda$ is $v'=0$, and the proper transform $\hat C$ of $C$ 
is given by $v'=u'^{m-2}$. Hence, on the minimal resolution of the 
\type{A_{m-2}}-point, both $\Lambda$ and $\hat C$ meet the same end of the chain.
Therefore, the dual graph $\Delta(H,C)$ is of the form
\begin{equation*}
\xymatrix@R=0pt{
\overset{m+2}\circ\ar@{-}[r]&\circ\ar@{-}[r]&{\overbrace{\circ\lin\cdots\lin\circ}^{m-3}}
\\
&\bullet\ar@{-}[u]
} 
\end{equation*} 
Now suppose that $C$ is a compact curve, $C\simeq \PP^1$ and consider a surface germ $(H,C)$ 
whose minimal resolution has the above form. 
It is easy to see that $K_H\cdot C=-2/m$ and $C$ can be contracted to a cyclic quotient singularity $(H_Z,o)$
of type $\frac14(1,1)$.
There is a Gorenstein threefold germ $X^\sharp$ with $\mumu_{m}$-action containing $H^\sharp$ 
as a $\mumu_{m}$-stable hypersurface.
According to \xref{def:existence} (and \cite[\S~3]{KSh88}) the germ $(H,C)$ has a smoothing in a 
$\QQ$-Gorenstein family. Thus there exists a $\QQ$-Gorenstein threefold $X$ containing $H$
as a Cartier divisor.
By the inversion of adjunction (see \cite[\S~3]{Sh:flips}, \cite[Ch.~17]{Utah}) 
$X$ has only terminal singularities. By arguments similar to \xref{theorem-main-def}
we see that there exists a birational contraction $f:X\to Z$ extending $H\to H_Z$.
We note that $C^\sharp$ can be identified with the pull back of $C$ by the splitting cover of $(X,C)$ at $P$.
Now we distinguish two cases according to the parity of $m$.

a) $m$ is even. Then $(X,C)$ is imprimitive of splitting degree $2$ at $P$.
Since $(H_Z,o)$ is a type~\type{T} singularity of index $2$, 
its pull-back $(H_Z',o')$ in the torsion free (degree $2$) cover \eqref{eq:base-cover}
is Du Val and so $(Z',o')$ is a \type{cDV} point.
Hence, both contractions $f'$ and $f$ are divisorial. 

b) $m$ is odd. Then $(X,C)$ is primitive
and the contraction is flipping by \eqref{eq:KC}.
Note that in this case the singularity $(Z,o)$ is not 
$\QQ$-Gorenstein.
On the other hand, since $(H_Z,o)$ is a singularity of class \type{T}, it has a $\QQ$-Gorenstein
smoothing. This smoothing belongs to a component of the versal deformation space 
which is different from that corresponding to $(Z,o)$ \cite[3.9]{KSh88}.
\end{example}

\section{Cases \typec{IC} and \typec{IIB}}
In this section we consider curve germs of types \typec{IC} and \typec{IIB}.

\subsection{Case \typec{IIB}}
\label{IIB-local-description}
Let $(X,P)$ be the germ of a three-dimensional terminal singularity and let $C\subset (X,P)$ be 
a smooth curve. Recall that the triple $(X,C,P)$ is said to be of type~\typec{IIB} if 
$(X,P)$ is a terminal singularity of type~\type{cAx/4} and
there are analytic isomorphisms
\begin{eqnarray*}
(X,P) &\simeq& \{y_1^2-y_2^3+\alpha =0\}/\mumu_4\subset \CC^4_{y_1,\dots,y_4} /\mumu_4(3,2,1,1),
\\
C&=& \{y_1^2-y_2^3=y_3=y_4=0\}/\mumu_4,
\end{eqnarray*}
where $\alpha=\alpha(y_1,\dots,y_4)\in (y_3,\, y_4)$ 
is a semi-invariant with $\wt(\alpha)\equiv 2\mod 4$ and the quadratic part
$\alpha_2$ of $\alpha(0,0,y_3,y_4)$ is not zero
(see \cite[A.3]{Mori:flip}).
We say that $(X,P)$ is a \emph{simple} (resp. \emph{double}) \type{cAx/4}-point if 
$\rk \alpha_2=2$
(resp. $\rk \alpha_2=1$). 

\begin{theorem}[{\cite{MP:ICIIB}}]
\label{main-IIB}
Let $f:(X,C\simeq\PP^1)\to (Z,o)$ be an extremal curve germ.
Suppose that $X$ contains a point $P$ of type~\typec{IIB}.
Then $(X,C)$ is not flipping \cite[Th.~4.5]{KM92} and $P\in X$ is the unique singular point of $X$ on $C$.
Furthermore, a general member $H\in |\OOO_X|$ is normal, smooth outside $P$,
and has only rational singularities. The following are the only possibilities
for the dual graphs of $(H,C)$ and $H_Z:=f(H)$:
{\rm
\begin{longtable}{c|c|p{0.25\textheight}|c|c}
{\rm No.} & \type{cAx/4}-point & \multicolumn{1}{c|}{$\Delta(H,C)$} &$\Delta(H_Z,o)$ 
\\\hline
\endhead
\nom
\label{IIB:thm-A2case-simple}
&simple&
$
\xymatrix@R=0pt@C=12pt{
\overset{3}\circ\ar@{-}[r]&\overset{4}\circ\ar@{-}[r]&\circ\ar@{-}[r]&\circ\ar@{-}[r]&\circ
\\
&\underset{3}\circ\ar@{-}[u]&\circ\ar@{-}[r]\ar@{-}[u]&\bullet
}
$&\type{A_2}& \type{d}
\\\hline
\nom
\label{IIB:thm-smooth-case-simple}
&simple&
$
\xymatrix@R=0pt@C=11pt{
\overset{3}\circ\ar@{-}[r]&\circ\ar@{-}[r]&\circ\ar@{-}[r]
&\circ\ar@{-}[r]&\circ\ar@{-}[r]&\circ
\\
&&\underset3\circ\ar@{-}[r]&\underset4\circ\ar@{-}[u]&&\bullet \ar@{-}[u]
}
$&\type{A_0}& \type{d}
\\\hline
\nom
\label{IIB:thm-D4case-double}
&double&
$
\xymatrix@R=3pt@C=12pt{
\circ\ar@{-}[r]&\circ\ar@{-}[r]&\circ\ar@{-}[r]&\overset4\circ\ar@{-}[r]&\overset3\circ\ar@{-}[r]&\circ
\\
&\bullet\ar@{-}[r]&\circ\ar@{-}[u]&&\circ\ar@{-}[u]&\circ\ar@{-}[ul]
}
$& \type{D_4}&\type{d}
\\\hline
\nom
\label{IIB-theorem-conic-bundle-case-double}
&double&
\multicolumn{2}{l|}{$
\xymatrix@R=0pt@C=8pt{
\circ\ar@{-}[r]&\overset{3}\circ\ar@{-}[r]&\circ\ar@{-}[r]&\circ\ar@{-}[r]&\circ\ar@{-}[r]&\circ\ar@{-}[r]&\circ\ar@{-}[r]&\circ\ar@{-}[r]&\bullet
\\
&\circ\ar@{-}[u]&&&\underset{4}\circ\ar@{-}[u]
}$}
& \type{c}
\end{longtable}}
The last column indicates if the germ is divisorial \typec{d} or $\QQ$-conic bundle \typec{c};
and the column $\Delta(H_Z,o)$ is not used in the latter case \typec{c}.
\end{theorem}
An example of divisorial contraction of type \xref{IIB:thm-A2case-simple}
is given in \cite[4.12]{KM92}.
The case \xref{IIB:thm-smooth-case-simple} was studied also by T.~Ducat
\cite[Thm. 4.1(2b)]{Ducat:16} in terms of symbolic blowups of smooth 
threefolds. 

\begin{proof}[Sketch of the proof]
In our case a general member $D\in |-K_X|$ contains $C$, has only Du Val singularities,
and the graph $\Delta(D,C)$ has the form \xref{ge:IIB}.
Under the identifications
of \xref{IIB-local-description}, 
a general member $D\in |-K_X|$ near $P$ is given by $\lambda y_3+\mu y_4=0$
for some $\lambda,\, \mu \in \OOO_X$
such that $\lambda(0)$, $\mu(0)$ are general in $\CC^*$
\cite[2.11]{KM92}, \cite[\S 4]{MP:cb3}.
Let $\Gamma:=H\cap D$. 

By \cite[Th.~4.5]{KM92} the contraction $f$ is not flipping. 
If $f$ is divisorial, we put $D_Z:=f(D)$ and $\Gamma_Z:=f(\Gamma)$. 
Then $D_Z\in |-K_Z|$, $H_Z$ is a general hyperplane section of 
$(Z,o)$, and $\Gamma_Z$ is a general hyperplane section of $D_Z$. 
If $f$ is a $\QQ$-conic bundle, we put $D_Z:=\Spec_{Z}f_*\OOO_D$
(the Stein factorization) and let $\Gamma_Z\subset D_Z$ be the image of $\Gamma$.
In both cases $D_Z$ is a Du Val singularity of type \type{E_6} by \xref{ge:IIB}.

We claim that $\Gamma=C+\Gamma_1$ \textup(as a scheme\textup), where $\Gamma_1$ is a reduced irreducible curve,
and $H$ is normal, smooth outside $P$, and has only rational singularities.
Consider two cases:

\begin{scase}\label{case-IIB-H-divisorial}
{\bf Case: $f$ is divisorial.}
Since the point $(Z,o)$ is terminal of index $1$, the germ $(H_Z,o)$ is a Du Val singularity.
Since $\Gamma_Z$ is a general hyperplane section of $D_Z$
we see that the graph $\Delta(D,\Gamma)$ has the form
\begin{equation}
\label{equation-IIB-graph-E6}
\vcenter{
\xymatrix@R=0pt{
&&\overset{2}\circ\ar@{-}[d]&\vartriangle\ar@{-}[l]
\\
\underset{1}\circ\ar@{-}[r]&\underset{2}\circ\ar@{-}[r]&\underset{3}\circ\ar@{-}[r]&\underset{2}\circ\ar@{-}[r]&\underset{1}{\diamond}
}}
\end{equation}
where $\vartriangle$ corresponds to the proper transform of $\Gamma_Z$
and numbers attached to vertices are coefficients of corresponding exceptional curves in the pull-back of $\Gamma_Z$.
By Bertini's theorem $H$ is smooth outside $C$.
Since the coefficient of $C$ equals $1$, $D\cap H=C+\Gamma_1$ (as a scheme), so $H$ is smooth 
outside $P$. In particular, $H$ is normal.
Since $f_H: H\to H_Z$ is a birational contraction and $(H_Z,o)$
is a Du Val singularity, the singularities of $H$ are rational.
\end{scase}

\begin{scase}\label{case-IIB-H-conic-bundle}
{\bf Case: $f$ is a $\QQ$-conic bundle.}
We may assume that, in a suitable coordinate system, 
the germ $(D_Z, o_Z)$ is given by
$x^2+y^3+z^4=0$ and the double cover $(D_Z, o_Z) \longrightarrow (Z,o)$
is just the projection to the $(y,z)$-plane.
Then $\Gamma_Z$ is given by $z=0$.
As in the case above we see that the graph $\Delta(D,\Gamma)$ has the form
\eqref{equation-IIB-graph-E6}. Therefore, $H$ is smooth outside $P$.
The restriction $f_H: H\to H_Z$ is a rational curve fibration.
Hence $H$ has only rational singularities. This proves our claim.
\end{scase}
\begin{scase}
Further, $\gr_C^1\OOO\simeq \OOO_{\PP^1}(d_1)\oplus \OOO_{\PP^1}(d_2)$ for some 
$d_1\ge d_2$. Since $H^1(\gr_C^1\OOO)=0$ by Corollary \ref{cor-C-pa=0}\ref{cor-C-pa=0a},
we have $d_2\ge -1$. 
Since $H$ is normal, we have $d_1\ge 0$ (see Lemma 
\ref{lemma-grC}).
On the other hand, 
$\deg \gr_C^1\OOO=1-i_P(1)$ by \eqref{eq-grO-iP1}.
One can compute $i_P(1)=2$ from \cite[(2.12)]{Mori:flip} for $P$ 
of type (IIB) described in \ref{IIB-local-description}.
Therefore, 
\[
\gr_C^1\OOO\simeq \OOO_{\PP^1}\oplus \OOO_{\PP^1}(-1)
\]
and $\OOO_C(-H)=\OOO\subset \gr_C^1\OOO$, i.e. the local equation of $H$
must be a generator of $\OOO\subset \gr_C^1\OOO$.
In the notation of \xref {IIB-local-description} the surface 
$H\subset X$ is locally near $P$ given by 
the equation $y_3 v_3+ y_4 v_4=0$, where $v_3,\, v_4\in \OOO_{P^\sharp, X^\sharp}$
are semi-invariants with $\wt(v_i)\equiv 3$
and at least one of $v_3$ or $v_4$ contains a linear term in $y_1$.
Therefore,
the surface germ $(H,P)$ can be given in $\CC^4/\mumu_4(3,2,1,1)$ by 
two equations:
\begin{eqnarray*}
\label{equation-IIB-H}
y_1^2-y_2^3+\eta(y_3,y_4)+ \phi(y_1,y_2,y_3,y_4)&=&0,
\\
y_1l(y_3,y_4)+y_2q(y_3,y_4)+\xi(y_3,y_4)+\psi(y_1,y_2,y_3,y_4) &=&0,
\end{eqnarray*}
where $\eta$, $l$, $q$ and $\xi$ are homogeneous polynomials 
of degree $2$, $1$, $2$ and $4$, respectively, $\eta\neq 0$, $l\neq 0$,
$\phi, \, \psi \in (y_3,\, y_4)$, $\sigma\mbox{-}\ord \phi\ge 3/2$, $\sigma\mbox{-}\ord \psi\ge 2$.
Moreover, $\rk \eta=2$
(resp. $\rk \eta=1$) if $(X,P)$ is a simple (resp. double) \type{cAx/4}-point. 
Then considering the weighted $\frac14(3,2,1,1)$-blowup and using 
rationality of $(H,P)$, as well as, generality of $H$ in $|\OOO_X|$
one can obtain the possibilities in Theorem \xref{main-IIB}. See \cite[\S~3]{MP:ICIIB}
for details.\qedhere
\end{scase}
\end{proof}

\subsection{Case \typec{IC}}
Let $(X,C\simeq \PP^1)$ be an extremal curve germ. Assume that $(X,C)$ has
a type~\typec{IC} point $P$ of index $m$. Then $P$ is the only singular point of $X$,
$m$ is odd $\ge 5$,
and $w_p (0) = (m - 1)/m$. 
Moreover, $i_P(1) = a_1=2$ and
\begin{equation*}
(X,C,P)\simeq \bigl(\CC_{y_1,y_2,y_4}^3,\{y_1^{m-2}-y_2^2=y_4=0\}, 0\bigr)/\mumu_m(2, m - 2, 1).
\end{equation*}
(see Lemma \xref{lemma:local-eq}\xref{lemma:local-eq1} and \cite[5.5, 6.5, A.3]{Mori:flip}).
Below is a complete classification of extremal curve germs of type \typec{IC}:

\begin{theorem}[{\cite[\S 8]{KM92}}, {\cite{MP:ICIIB}}]
\label{IC-main}
Let $f:(X,C\simeq \PP^1)\to (Z,o)$ be an extremal curve germ of type~\typec{IC}.
Let $P\in X$ be \textup(a unique\textup) singular point
and let $m$ be its index.
Then a general member $H\in |\OOO_X|$ is normal, smooth outside $P$,
has only rational singularities.
Moreover, $(X,C)$ is not divisorial and
we have one of the following:
\begin{emptytheorem}
$(X,C)$ is flipping and the following are the only possibilities
for the dual graphs of $(H,C)$ and $H_Z=f(H)$\textup:
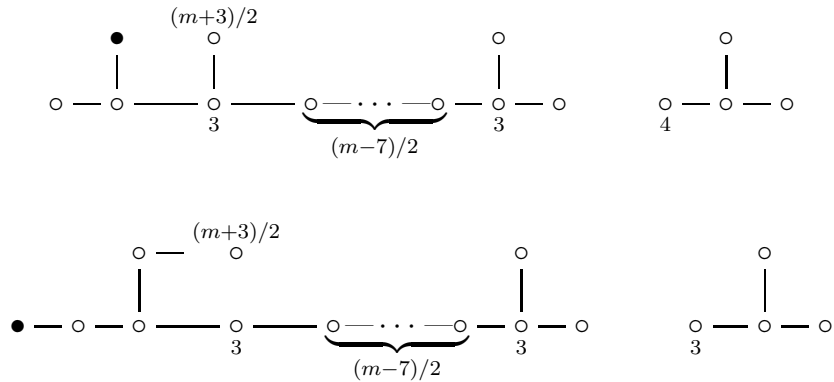
\begin{equation*}
\xymatrix@R=-2pt@C=10pt{
&\bullet\ar@{-}[d]&\overset{(m+3)/2}\circ\ar@{-}[d]&
&\circ\ar@{-}[d]&
&&&\circ\ar@{-}[d]&
\\
\circ\ar@{-}[r]&
\circ\ar@{-}[r]&
\underset{3}\circ\ar@{-}[r]&
\mbox{$\underset{(m-7)/2}{\underbrace{\circ\lin \cdots\lin\circ}}$}\ar@{-}[r]
&
\underset{3}\circ\ar@{-}[r]&
\circ
&&
\underset{4}\circ\ar@{-}[r]&\circ\ar@{-}[r]&\circ
}
\end{equation*}
\begin{equation*}
\xymatrix@R=1pt@C=10pt{
&&\circ\ar@{-}[d]&\overset{(m+3)/2}\circ\ar@{-}[l]&
&\circ\ar@{-}[d]&
&&\quad&\circ\ar@{-}[d]&
\\
\bullet\ar@{-}[r]&\circ\ar@{-}[r]&
\circ\ar@{-}[r]&
\underset{3}\circ\ar@{-}[r]&
\mbox{$\underset{(m-7)/2}{\underbrace{\circ\lin \cdots\lin\circ}}$}\ar@{-}[r]
&
\underset{3}\circ\ar@{-}[r]&
\circ
&&
\underset{3}\circ\ar@{-}[r]&\circ\ar@{-}[r]&\circ
}
\end{equation*}
where 
\mbox{$\underset{3}\circ\lin\underset{(m-7)/2}{\underbrace{\circ\lin \cdots\lin\circ}}\lin\underset{3}\circ$}
must be replaced with $\underset {4}\circ$ in the case $m=5$.

\end{emptytheorem}

\begin{emptytheorem}
\label{IC-(8.3.2)}
$(X,C)$ is a $\QQ$-conic bundle, $m=5$, 
and $\Delta(H,C)$ has the form\textup:
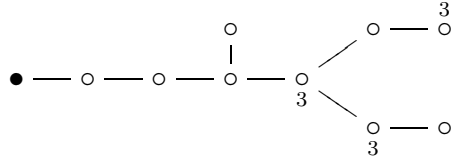
\begin{equation*}
\xymatrix@R=-2pt@C=14pt{
&&&\circ\ar@{-}[d]&&\circ\ar@{-}[r]&\overset3\circ
\\
\bullet\ar@{-}[r]&\circ\ar@{-}[r]&\circ\ar@{-}[r]&\circ\ar@{-}[r]&
\underset3\circ\ar@{-}[rd]\ar@{-}[ru]
\\
&&&&&\underset3\circ\ar@{-}[r]&\circ
}
\end{equation*} 
\end{emptytheorem}
\end{theorem}
In the flipping case, the general member $H^+\in |\OOO_{X^+}|$
of the flipped variety is also computed. Here $X^+$ is either
of index two or Gorenstein \cite[A.3]{KM92}.

\section{Case \typec{IA}}
\label{sect:IA}
\subsection{}
\label{IA:setup}
An extremal curve germ $(X, C\simeq \PP^1)$ is said to be of type~\typec{IA} if it 
contains exactly one non-Gorenstein point $P$ which is of type~\typec{IA}. 
For readers' convenience, we note the following characterization 
for an extremal curve germ $(X, C\simeq \PP^1)$ 
to be of type~\typec{IA} with $D\in |-K_X |$ a general member (see \xref{prop:local-primitive}
and \xref{thm:ge1}): 
\begin{quote}
$(X, C)$ is of type~\typec{IA} if and only if (i) $P$ is locally primitive, (ii) 
$D \cap C$ is a single point, and (iii) $(X,P)$ is not of 
type~\type{cAx/4}.
\end{quote}

\begin{scase}\label{IA:setup:c}
From now on we assume that the germ $(X, C\simeq \PP^1)$
satisfies the assumptions of \xref{IA:setup}.
The following are the only possibilities for the singularity
$(X,P)$:
\begin{enumerate}
\item 
$(X,P)$ is of type \type{cA/m}, in this case $(X,C)$ is said to be of type \typec{k1A}
according to \cite{KM92};
\item 
$(X,P)$ is of type \type{cD/3};
\item 
$(X,P)$ is of type \type{cAx/2}, \type{cD/2} or \type{cE/2}.
\end{enumerate}
Thus in our case of type \typec{IA}, $(X,C)$ is semistable if and only if it is of type \typec{k1A}.
Extremal curve germs of index two are classified in 
Sect.~\xref{sect:index2}. Thus we discuss here cases \typec{k1A} and \type{cD/3}. 
We start with $\QQ$-conic bundles: 
\end{scase}

\begin{theorem} [{\cite[1.6]{MP:IA}}]
\label{thm:IA:cb}
Let $(X,C\simeq\PP^1)$ be a $\QQ$-conic bundle germ of index $m>2$ and
of type~\typec{IA}. 
Let $P\in X$ be the non-Gorenstein point.
Then $(X,P)$ is a point of type~\type{cA/m}
and a general member $H\in |\OOO_X|$ is not normal.
Furthermore, the dual graph of $(H^{\n},C^{\n})$, the normalization $H^{\n}$
and the inverse image $C^{\n}$ of $C$, has the form:
\begin{equation}
\label{IA-n-normal-graph:cb}
\underbrace{\overset{a_r}\circ\lin\cdots\lin\overset{a_1}\circ}_{\varDelta_1}
\lin\bullet\lin 
\underbrace{\overset{b_1}\circ\lin\cdots\lin\overset{b_s}\circ}_{\varDelta_2}
\end{equation}
\par\noindent
\textup(in particular, $C^{\n}$ is irreducible\textup).
Here 
the chain $\varDelta_1$ \textup(resp., $\varDelta_2$\textup) 
corresponds to the singularity of type $\frac1m(1,a)$
\textup(resp., $\frac1m(1,-a)$\textup) 
for some integer $a$ \ \textup($1\le a<m$\textup) relatively prime to $m$.
The germ $(H,C)$ is analytically isomorphic to the germ 
along the line $y=z=0$
of the hypersurface given by the following weighted polynomial of degree $2m$
in variables $x$, $y$, $z$, $u$:
\begin{equation*}
\phi:= x^{2m-2a}y^2+x^{2a}z^2+yzu
\end{equation*}
in $\PP(1,a,m-a,m)$. 
Furthermore $(X,C)$ is given as an analytic germ of a subvariety of
$\PP(1,a,m-a,m) \times \CC_t$ along $C \times 0$ given by
\begin{equation*}
\phi+ \alpha_1x^{2m-a}y+\alpha_2x^{m-a}uy+\alpha_3x^{2m}+\alpha_4x^{m}u+\alpha_5u^{2}=0
\end{equation*}
for some $\alpha_1,\ldots \alpha_5\in t\OOO_{0,{\CC_t}}$ and there is a $\QQ$-conic
bundle structure $X \to \CC^2$ through which the second projection $X \to \CC_t$ factors.
\end{theorem}

\begin{theorem}[{\cite[1.9]{MP:IA}}, see also \cite{Tzi:05}]
\label{theorem-main-birational}
Let $(X,C)$ be a birational extremal curve germ of type~\typec{k1A}.
Let $P\in X$ be the point of index $m\ge 2$.

\begin{emptytheorem}
\label{thm:IA-normal}
If a general element $H$ is normal, then the graph $\Delta(H,C)$ has
the same form as in \eqref{imprimitiveIA-graph}, however the cases $r=1$ and $r=n$ are not excluded.
\end{emptytheorem}

\begin{emptytheorem}
\label{thm:IA-not-normal}
If every member of $|\OOO_X|$ is non-normal, then the dual graph of the normalization $(H^{\n},C^{\n})$ 
is of the form
\begin{equation}
\label{IA-n-normal-graph}
\underbrace{\overset{a_r}\circ\lin\cdots\lin\overset{a_1}\circ}_{\varDelta_1}
\lin\bullet\lin
\underbrace{\overset{c_1}\circ\lin\cdots\lin\overset{c_l}\circ}_{\varDelta_3}
\lin\diamond\lin
\underbrace{\overset{b_1}\circ\lin\cdots\lin\overset{b_s}\circ}_{\varDelta_2}
\end{equation}
\textup(in particular, $C^{\n}$ is reducible\textup).
The chain $\varDelta_1$ \textup(resp., $\varDelta_2$\textup) 
corresponds to the singularity of type $\frac1m(1,a)$
\textup(resp., $\frac1m(1,-a)$\textup) for some $a$ with $\gcd(m,a)=1$ 
and the chain $\varDelta_3$ 
corresponds to the point $(H^{\n},Q^{\n})$, where $Q^{\n}=C_1^{\n}\cap C_2^{\n}$.
Moreover, 
\begin{equation*}
\sum (c_i-2)\le 2 \quad\text{and}\quad \widetilde C_1^2 +\widetilde C_2^2 +5-\sum(c_i-2)\ge 0,
\end{equation*}
where $\widetilde C=\widetilde C_1 +\widetilde C_2$ is the proper transform of $C$ on the minimal resolution
$\widetilde H$.
Both components of $\widetilde C$ are contracted on the minimal model 
of $\widetilde H$.
In this case, 
\begin{equation*}
(X,C,P)\simeq \bigl(\{\alpha(x_1,\dots,x_4)=0\}, \text{$x_1$-axis}, 0\bigr)/\mumu_m(1,a,-a,0),
\end{equation*} 
where $\gcd(m,a)=1$ and $\alpha=0$ is the equation of 
a terminal \type{cA/m}-point in $\CC^4/\mumu_m$. 
\textup(In particular, $(X,C)$ is of type~\typec{IA}\textup).
\end{emptytheorem}

Conversely, for any germ $(H,C\simeq \PP^1)$ of the form 
\xref{thm:IA-normal} or \xref{thm:IA-not-normal}
admitting 
a birational contraction $(H,C)\to (H_Z,o)$ there exists 
a threefold birational contraction 
$f: (X,C)\to (Z,o)$ as in \xref{IA:setup} of type~\typec{IA}
such that $H\in |\OOO_X|$.
\end{theorem}

\begin{scase}
To study a general member $H\in |\OOO_X|$
we can use Lemma \xref{prop:lc}.
However we cannot assert as in \xref{pf:IAdual:normality}
that $H$ is normal. In fact, arguments similar to 
\xref{pf:IAdual:normal} show that the case of normal $H$ does not occur 
if $(X,C)$ is a $\QQ$-conic bundle. 
\end{scase}

\subsection{}
Let us outline the proofs of Theorems \ref{thm:IA:cb} and \ref{theorem-main-birational}.
The case where $H$ is normal is teated in the same way as \xref{imprimitiveIA} (see \xref{pf:IAdual:normal})
and $X$ can be recovered as a one-parameter deformation space by \xref{def:existence}.
Examples are given in \xref{ex:IA-n} below.

Suppose that $H$ is not normal.
Let $\nu: H^\n\to H$ be the normalization and let $C^\n\subset H^\n$ be the inverse image of $C$.
By the inversion of adjunction the pair $(H,C)$ is slc, the pair $(H^\n,C^\n)$ is lc,
and the point $P\in (H,C)$ is slt \cite[16.9]{Utah}.
In particular, $H$ is a generically normal crossing divisor.
At certain (finite number) of dissident points $H$ may have singularities worse than 
just normal crossing points. 

\begin{scase}\label{IA:H:points}
Since $H$ has $\QQ$-Gorenstein smoothing, 
by \cite[Theorem 4.24, 5.2]{KSh88} the only possibilities are:
\begin{itemize}
\item 
Pinch point: $\{x^2-y^2z = 0\} \subset \CC^3$.

\item
Degenerate cusp of embedding dimension at most $4$, where a degenerate cusp 
is a non-normal Gorenstein singularity having a semi-resolution whose 
exceptional divisor is a cycle of smooth rational curves or a rational
nodal curve (see \cite{SB:degen}). 

\item
Slt singularity of the form
\begin{equation*}
\{xy = 0\}/\mumu_m(a,-a, 1),\quad \gcd(a, n) = 1\}
\end{equation*} 
(this point corresponds to $P\in H$).
\end{itemize}
\end{scase}

\begin{scase}
The restriction $\nu_C: C^\n\to C$ of the normalization to the 
inverse image of $C$ is a double cover. 
We distinguish two possibilities:
\begin{enumerate}
\item
\label{case:n-normal:irre}
$C^\n$ is smooth irreducible and $\nu_C$ is branched 
at two points,
\item
\label{case:n-normal:red}
$C^\n$ has two irreducible components meeting at one point and 
the restriction of $\nu_C$ to each of them is an isomorphism.
\end{enumerate}

A detailed analysis (see \cite{MP:IA} and also \cite{Tzi:05}) shows that 
\xref{case:n-normal:irre} leads to the $\QQ$-conic bundle case \eqref{IA-n-normal-graph:cb}
while \xref{case:n-normal:red} leads to the birational case \eqref{IA-n-normal-graph}.
In both cases the subgraphs $\varDelta_1$ and $\varDelta_2$ correspond to points $P^n_1,\, P^\n_2\in H^\n$
lying over $P\in H$. 
\qed
\end{scase}

\begin{scase}
To recover $X$ as a one-parameter deformation space
we also can apply arguments as in \xref{def:existence}.
However, in the case of non-normal surface $H$, it needs some restriction to singularities and 
additional technical tools \cite{Tziolas2009}. 
Fortunately, the results of \cite{Tziolas2009} are applicable if $H$ has singularities described above.
Moreover, the miniversal deformation family of $(H,C)$ in the $\QQ$-conic bundle case 
is computed explicitly \cite[6.8.3]{MP:IA}. \qed
\end{scase}

\begin{scase}
\label{IA:check-divisoriality}
To check divisoriality one can use the criterion
\xref{theorem-main-Q-Cartier-i}. Indeed, if $f$ is divisorial, then 
$(Z,o)$ is a terminal point and its index equals $1$ because 
$(X,C)$ is primitive (see \xref{base-change}). Therefore, 
its general hyperplane section $(H_Z,o)$ must be a Du Val singularity.
If on the contrary $f$ is flipping, then $(Z,o)$ is not $\QQ$-Gorenstein and 
$(H_Z,o)$ cannot be Du Val. Given a graph $\Delta(H, C)$ of type \eqref{imprimitiveIA-graph}
one can easily draw the graph $\Delta(H_Z)$ contracting 
black vertices successfully. Thus the Du Val condition
of $(H_Z,o)$ can be checked in purely combinatorial terms. 
\end{scase}

\begin{sremark}
Assume that in the assumptions of \xref{thm:IA-normal} and \eqref{imprimitiveIA-graph} 
we have $r=1$ or $r=n$.
Then the graph $\Delta(H,C)$ is a chain. In this case there exists 
an element $D\in |-K_X|$ \emph{containing} $C$ and having Du Val singularities only.
This is a particular case of the situation considered in \cite{Mori:ss}
where a powerful algorithm to construct $(X,C)$ was obtained.
\end{sremark}

\begin{scase}
\label{HTU}
One special case of Theorem \xref{theorem-main-birational} was studied in details in \cite{HTU}.
There the authors assumed that the nearby fiber $H_t$ of the one-parameter deformation 
$\cup H_t=X$ has $b_2(H_t)=1$. This strong assumption is equivalent to that
$H$ is normal and has so-called \emph{Wahl singularity} at $P$: $(H,P)\simeq \CC^2/\mumu_{m^2}(1, ma-1)$.
Under this assumption, it is shown that birational germs of this type 
belong to the same deformation family as those of \typec{k2A} studied in \cite{Mori:ss},
constructed the universal family, and the algorithm \cite{Mori:ss} of computing 
flips was extended. 
\end{scase}

\begin{sexamples}
\label{ex:IA-n}
Consider several examples of extremal germs of type \xref{thm:IA-normal}:
\begin{enumerate}
\item 
The index two germs \xref{index2flipping}-\xref{KM:4.7.3.1.4}
are of type \typec{IA}. By using arguments of \xref{IA:check-divisoriality}
one can conclude that the germ as in \xref{index2flipping} is flipping and those in
\xref{KM:4.7.3.1.1}-\xref{KM:4.7.3.1.4} are divisorial.
\item
Let $\Delta(H,C)$ be of the form
\begin{equation*}
\xymatrix{
\bullet\ar@{-}[r]&\overset{c_1}\circ\ar@{-}[r]&\cdots\ar@{-}[r]&\overset{c_n}\circ
} 
\end{equation*} 
where the white vertices form a dual graph of a non-Du Val \type{T}-singularity \textup(see \xref{typeT} \textup).
It is easy to see that $C$ can be contracted to 
a cyclic quotient non-Du Val point. 
Therefore, the one-parameter deformation 
produces a flipping contraction. Since $(H,C)$ is plt, the contraction is primitive.

\item
Let $\Delta(H,C)$ be of the form
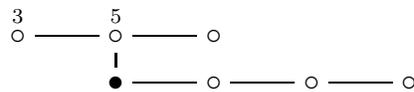
\begin{equation*}
\xymatrix@R=1pt{
\overset{3}\circ\ar@{-}[r]&\overset{5}\circ\ar@{-}[r]&\circ
\\
&\bullet\ar@{-}[u]\ar@{-}[r]&\circ\ar@{-}[r]&\circ\ar@{-}[r]&\circ
} 
\end{equation*} 
This is an example of a divisorial contraction to a smooth point. 

\item
A series of examples were given in \xref{ex:IAdual} b).
\end{enumerate}
\end{sexamples}

For completeness, we provide an example of birational curve germ
with non-normal $H$.

\begin{sexample}[{\cite[Ex. 2]{Tzi:05}}, {\cite[Ex. 6.10.3]{MP:IA}}]
Consider a surface $\tilde H$ containing a configuration with the following graph
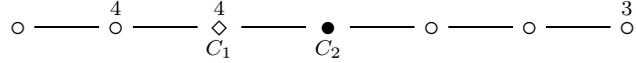
\[
\xymatrix{
\overset{}\circ\ar@{-}[r]&
\overset{4}{\circ} \ar@{-}[r]&
\underset{C_1} {\overset{4}\diamond}\ar@{-}[r]&
\underset{C_2}\bullet\ar@{-}[r]&
\overset{}\circ\ar@{-}[r]&
\overset{}\circ\ar@{-}[r]&
\overset{3}\circ
}
\]
Contracting all curves except those marked by $C_1$ and $C_2$, we obtain a 
normal surface $H^n$ having two cyclic quotient singularities
$P_1$ and $P_2$ of types $\frac17(1,2)$ and $\frac 17(1,-2)$.
Identifying the curves $C_1$ and $C_2$ we obtain a non-normal surface 
$H$ so that the map $\nu: H^n\to H$ is the normalization. 
The dissident singularities of $H$ are a degenerate cusp of multiplicity $2$ and embedding dimension $3$
at $\nu(C_1\cap C_2)$, and one point of type
$\{xy = 0\}/\mumu_7(2,-2, 1)$. The results of \cite{Tziolas2009} are applicable here 
and so there exists a one-parameter smoothing $X\supset H\supset C$
which is a divisorial curve germ and $H$ is general in $|\OOO_X|$,
see \cite[Prop. 6.3 and Th. 6.10]{MP:IA}.
\end{sexample}

\subsection{Points of type~\type{cD/3}.}
Let $(X,C,P)$ be a triple of type~\typec{IA}, where $(X,P)$ is 
a singularity of type~\type{cD/3} \cite{Mori:sing}, \cite{Reid:YPG}.
These triples are described as follows
(see \cite[6.5]{KM92}). 
Put $\sigma:=(1,1,2,3)$.
Up to coordinate 
change the point $(X,C,P)$ is given in $\CC^4_{y_1,\dots,y_4}$ as follows
\begin{equation}
\label{eq:cD/3}
\begin{array}{l}
(X,C,P)=\bigl(\{\alpha=0\}, \ \{\text{$y_1$-axis}\},\ 0\bigr)/\mumu_3(1,1,2,0),
\\[1pt]
\alpha=y_4^2+y_3^3+\delta_3(y_1,y_2)
+(\text{terms of degree $\ge 4$}),
\end{array}
\end{equation}
where 
$\delta_3\neq 0$ is homogeneous of degree $3$ and
$\alpha$ is invariant. Moreover, 
\begin{equation*}
\alpha\equiv y_1^\ell y_i \mod (y_2, y_3,y_4)^2,
\end{equation*} 
where
$\ell=\ell(P)$ and $i =2$ (resp. $3$, $4$) if $\ell\equiv 2$ (resp. $1$, $0$) $\mod 3$ 
\cite[(2.16)]{Mori:flip}.
If $\delta_3(y_1,y_2)$ is square free (resp. has a double factor, is a cube of a linear form),
then $(X,P)$ is said to be a \emph{simple} (resp. \emph{double}, \emph{triple}) \type{cD/3} point.

Extremal curve germs containing a terminal 
singular point of type~\type{cD/3} 
are described by the following theorem. 

\begin{stheorem}[{\cite[Th.~6.2-6.3]{KM92}}, {\cite[Th.~4.5, 4.8]{MP:IA}}]
\label{cD/3:thm)}
Let \mbox{$f:(X,C\simeq \PP^1)\to (Z,o)$} be an extremal curve germ
having a point $P$ of type~\type{cD/3}.
Then $f$ is a birational contraction, not a $\QQ$-conic bundle.
General members $H\in |\OOO_X|$ and $H_Z=f(H)\in |\OOO_Z|$ are normal and 
have only rational singularities.
We have the following possibilities for graphs $\Delta(H,C)$ and $\Delta(H_Z,o)$
and local invariants.
\par\medskip\noindent
{\rm
\begin{longtable}{c|c|c|p{0.2\textheight}|l|c}
{\rm No.} & $\ell(P)$ & $i_P(1)$ & $\Delta(H,C)$ &$\Delta(H_Z,o)$ 
\\\hline
\endhead
\multicolumn{5}{c}{Cases of simple \type{cD/3} point $P$}
\\\hline
\nom
\label{cD/3:flip:.3.1)}
& $2$& $1$ &
$\xymatrix@R=0pt@C=15pt{
\bullet\ar@{-}[r]&\overset3\circ\ar@{-}[r]&\circ\ar@{-}[r]&\overset3\circ
\\
&&\underset3\circ\ar@{-}[u]
}$
&
$\xymatrix@R=0pt@C=10pt{
&\overset3\circ\ar@{-}[d]
\\
\circ\ar@{-}[r]&\circ\ar@{-}[r]&\underset3\circ
}$ & \type{f}
\\\hline
\nom
\label{cD/3:thm:A2}
& $2$ & $1$ &
$
\xymatrix@R=3pt@C=15pt{
\bullet\ar@{-}[r]&\overset{3}\circ\ar@{-}[r]&\circ\ar@{-}[r]&\overset{3}\circ
\\
\circ\ar@{-}[u]&&\underset{3}\circ\ar@{-}[u]&
} 
$
&\type{A_2} & \type{d}
\\\hline
\multicolumn{5}{c}{Cases of double \type{cD/3} point $P$}
\\\hline
\nom
\label{cD/3:flip:.3.2)}
& $2$ & $1$ &
$\xymatrix@R=7pt@C=15pt{
\bullet\ar@{-}[d]&&\circ\ar@{-}[d]
\\
\underset3\circ\ar@{-}[r]&\circ\ar@{-}[r]&\underset3\circ\ar@{-}[r]&\circ
\\
&&\circ\ar@{-}[u]
}$
&
$\xymatrix@R=7pt@C=10pt{
\circ\ar@{-}[d]&\circ\ar@{-}[d]
\\
\circ\ar@{-}[r]&\underset3\circ\ar@{-}[r]&\circ
\\
&\circ\ar@{-}[u]
}$ & \type{f}
\\\hline
\nom
\label{cD/3:thm:D4}
& $2$ & $1$ &
$
\xymatrix@R=9pt@C=15pt{
\bullet\ar@{-}[d]&\circ\ar@{-}[l]&\circ\ar@{-}[d]&
\\
\underset{3}\circ
\ar@{-}[r]&\circ\ar@{-}[r]&\underset{3}\circ\ar@{-}[r]&\circ
\\
&&\circ\ar@{-}[u]&
}
$
& \type{D_4} & \type{d}
\\\hline
\nom
\label{cD/3:flip:iP=34)}
& $3$, $4$ & $2$ &
$
\xymatrix@R=9pt@C=15pt{
&\bullet\ar@{-}[d]&\circ\ar@{-}[d]&\overset3\circ 
\\
&\circ\ar@{-}[r]&\underset3\circ\ar@{-}[r]&\circ\ar@{-}[u]
\\
&&\circ\ar@{-}[u]
}
$
&
$
\xymatrix@R=7pt@C=15pt{
\circ\ar@{-}[d]
\\
\circ\ar@{-}[r]&\circ\ar@{-}[r]&\underset3\circ
\\
\circ\ar@{-}[u]
}
$ & \type{f}
\\\hline
\multicolumn{5}{c}{Case of triple \type{cD/3} point $P$}
\\\hline
\nom
\label{cD/3:thm:E6}
&$3$, $4$& $2$ &
$
\xymatrix@R=9pt@C=15pt{
\bullet\ar@{-}[d]&&\circ\ar@{-}[d]\ar@{-}[r]&\circ
\\
\circ\ar@{-}[r]&\circ\ar@{-}[r]&\overset{3}\circ\ar@{-}[r]&\circ
\\
&&\circ\ar@{-}[u]&\circ\ar@{-}[l]
}
$& \type{E_6} & \type{d}
\end{longtable}}
In the cases 
\xref{cD/3:flip:.3.1)}, \xref{cD/3:flip:.3.2)}, \xref{cD/3:flip:iP=34)}, and \xref{cD/3:thm:E6} 
the variety $X$ is smooth outside $P$ and in the cases \xref{cD/3:thm:A2} and \xref{cD/3:thm:D4}\ $X$
may has at most one type \typec{III} point.
The last column indicates if the germ is flipping \typec{f} or divisorial \typec{d}.
\end{stheorem}

Note that \cite[\S~6]{KM92} and \cite[\S~4]{MP:IA} provide much more information about 
these contractions:
infinitesimal structure, criterion for an arbitrary germ to be of the corresponding type,
and computations of flipped varieties \cite[A.1]{KM92}. Flipping contractions 
can be constructed explicitly by patching certain open subsets:

\begin{sexample}[{\cite[6.11]{KM92}}]
\label{ex:cD/3:flip}
Let $V\supset C$ be a germ of a smooth threefold along $C\simeq \PP^1$
such that $\NNN_{C/V}\simeq \OOO_C\oplus\OOO_C$. 
Pick a point $P\in C$ and let $(v_1, v_2, v_3)$ be coordinates at
$(V, P)$ such that $(C, P) = \{\text{$v_1$-axis}\}$. Let $(X,C, P)$ be a \type{cD/3} point as in
\eqref{eq:cD/3} with $\ell=2$.
For suitable $\varepsilon_1$ and $\varepsilon_2$ such that $0<\varepsilon_1<\varepsilon_1\ll 1$, 
$(y_1^3, y_4, y_1y_3)$ form 
coordinates for $U = (X, P)\cap \{\varepsilon_1 < | y_1^3| < \varepsilon_2\}$ by the implicit function theorem. 
Thus
$v_1 = y_1^3$, $v_2 = y_4$, and $v_3 =y_1y_3$
patch $(X, P)$ and $V\setminus (V, P) \cap \{|v_1| < \varepsilon_1\}$ along $U$. 
By \cite[6.2.4]{KM92} the germ $(X,C)$ is a flipping
curve germ of type \type{cD/3} as in \xref{cD/3:flip:.3.1)}
or \xref{cD/3:flip:.3.2)} (depending on the choice of $\delta_3$ in \eqref{eq:cD/3}).
\end{sexample}

More examples of flipping contractions
are given in \cite[6.17 and 6.21]{KM92}.
To show that all the possibilities in Theorem~\xref{cD/3:thm)} occur one can also use 
the deformation arguments \xref{def:existence}:

\begin{sexample}
\label{ex:cD/3:E6}
Consider the surface contraction $f_H: H\to H_Z$ with dual graph 
\xref{cD/3:thm:E6} and consider the following triple of germs:
\begin{equation*}
(X, H, P) = \bigl(\{y_2^3+y_3^3+y_3y_1^4+y_4^2 \}, \{y_4=y_1y_3\}, 0\bigr) /\mumu_3(1,1,2,0),
\end{equation*} 
where $H$ is cut out by $y_4=y_1y_3$. 
Here $(X,P)$ is a triple \type{cD/3}-singularity (see 
\xref{eq:cD/3}).
By \cite[4.12]{MP:IA} the dual graph of the 
minimal resolution of $(H,P)$ is the same as that in \xref{cD/3:thm:E6}.
By \xref{def:existence} one obtains a
birational contraction $f: X\to Z$ extending $f_H: H\to H_Z$,
which is as in \xref{cD/3:thm:E6}. Examples similar to \xref{cD/3:thm:A2} and
\xref{cD/3:thm:D4} are given in \cite[4.14]{MP:IA}.
\end{sexample}

Divisorial contractions of type \xref{cD/3:thm:E6} were studied also in 
\cite[5.1(2)]{Tzi:10} by a different method.

\section{Case \typec{IIA}} 
\begin{setup}
\label{Set-up}
Let $(X,C)$ be an extremal
curve germ and let \mbox{$f: (X, C)\to (Z,o)$} be the corresponding
contraction.
Assume that $(X,C)$ has a point $P$ of type~\typec{IIA}.
Then by \cite[6.7, 9.4]{Mori:flip} and
\cite[8.6, 9.1, 10.7]{MP:cb1} $P$ is the only non-Gorenstein point of $X$
and $(X,C)$ has at most one Gorenstein singular point $R$ \cite[6.2]{Mori:flip},
\cite[9.3]{MP:cb1}. Since $P\in (X,C)$ is locally primitive,
the topological index of $(X,C)$ equals $1$.
Hence the base $(Z,o)$ is smooth
in the $\QQ$-conic bundle case,
and 
is a \type{cDV} point (or smooth) in the divisorial case
(cf. \xref{thm:div:Q-Cartier}).
\end{setup}

\subsection{}
\label{(7.5)}
According to \cite[A.3]{Mori:flip} we can express the \typec{IIA} point as
\begin{equation}
\label{equation-XC}
\begin{array}{rcl}
(X, P)&=&
\{\alpha=0\}/\mumu_4(1, 1, 3, 2)\subset\CC^4_{y_1,\dots, y_4}/\mumu_4(1, 1, 3, 2),
\\[1pt]
C&=&\{y_1\text{-axis}\}/\mumu_4,
\end{array}
\end{equation}
where $\alpha=\alpha(y_1,\dots, y_4)$ is a semi-invariant such that
\begin{equation}
\label{equation-alpha}
\wt\alpha\equiv 2\mod 4,\qquad \alpha\equiv y_1^{\ell(P)}y_j\mod (y_2, y_3, y_4)^2,
\end{equation}
where $j= 2$ (resp. $3$, $4$)\ if $\ell(P)\equiv 1$ (resp. $3$, $0$) $\mod 4$
(see \eqref{equation-iP-lP}) and $(\I_C^\sharp)^{(2)}=(y_j)+(\I_C^\sharp)^{2}$.
Moreover, $y_2^2,\, y_3^2\in \alpha$ (because $(X,P)$ is a terminal point of type
\type{cAx/4}).
Note that $\ell(P)\not\equiv 2\mod 4$ because
of the lack of a variable with $\wt\equiv 0\mod 4$.

\begin{theorem}[{\cite[7.2-7.4]{KM92}}, {\cite{MP:IIA-1}}, {\cite{MP:IIA-2}}]
\label{IIA:thm} 
Let $f:(X,C\simeq \PP^1)\to (Z,o)$ be an extremal curve germ
having a point $P$ of type~\typec{IIA}.
We have the following possibilities for graphs $\Delta(H,C)$ and $\Delta(H_Z,o)$
and local invariants.
\par\medskip\noindent
{\rm
\begin{longtable}{c|c|c|p{0.4\textwidth}|l|c}
{\rm No.} & $i_P(1)$ & $\ell(P)$ & \multicolumn{1}{c}{$\Delta(H,C)$} &\multicolumn{1}{|c|}{$\Delta(H_Z,o)$} &
\\\hline\endhead
\multicolumn{6}{c}{\rm Cases: $H$ is normal}
\\\hline
\nom
\label{IIA:flip:iP=1a}
&$1$& $1$ &
$\xymatrix@R=0pt@C=13pt{
&&\overset{4}\circ\ar@{-}[d]
\\
\bullet\ar@{-}[r]&\underset{4}\circ\ar@{-}[r]&\circ\ar@{-}[r]&\circ
}$
&
$\xymatrix@R=0pt@C=13pt{
&\overset{4}\circ\ar@{-}[d]
\\
\underset{3}\circ\ar@{-}[r]&\circ\ar@{-}[r]&\circ
}$
&
\type{f}

\\\hline
\nom
\label{IIA:flip:iP=1b}
&$1$& $1$ &
$\xymatrix@R=0pt@C=13pt{
\circ\ar@{-}[d] &&\overset{4}\circ\ar@{-}[d]
\\
\bullet\ar@{-}[r]&\underset{4}\circ\ar@{-}[r]&\circ\ar@{-}[r]&\circ
}$
&
$\xymatrix@R=0pt@C=13pt{
&\overset{4}\circ\ar@{-}[d]
\\
\circ\ar@{-}[r]&\circ\ar@{-}[r]&\circ
}$
&
\type{f}

\\\hline
\nom
\label{IIA:flip:iP=2}
&$2$& $3$, $4$ &
$\xymatrix@R=10pt@C=13pt{
\circ\ar@{-}[r]&\circ\ar@{-}[rd]&\circ\ar@{-}[d]
\\
&&\overset{4}{\circ}\ar@{-}[r]&\circ
\\
\bullet\ar@{-}[r]&\circ\ar@{-}[ru]&\circ\ar@{-}[u]
}$
&
$\xymatrix@C=13pt@R=10pt{
\circ\ar@{-}[d] &\circ\ar@{-}[d]
\\
\circ\ar@{-}[r]&\overset3\circ\ar@{-}[r]&\circ
\\
&\circ\ar@{-}[u]
}$
&
\type{f}

\\\hline
\nom\label{IIA:normalA1:a}
&$1$& $1$ &
$
\xymatrix@R=3pt@C=10pt{
\circ\ar@{-}[r]\ar@{-}[d]&\circ&\circ\ar@{-}[d]
\\
\bullet\ar@{-}[r]&\underset4\circ\ar@{-}[r]
&\circ\ar@{-}[r]&\underset4\circ
}
$&\type{A_1}
&
\type{d}

\\\hline
\nom\label{IIA:normalA1:b}
&$1$& $1$ &
$
\xymatrix@R=3pt@C=10pt{
\circ\ar@{-}[d]&&\overset3\circ\ar@{-}[d]&\overset4\circ\ar@{-}[d]
\\
\bullet\ar@{-}[r]&\underset3\circ\ar@{-}[r]
&\circ\ar@{-}[r]&\circ
}
$& \type{A_1} &\type{d}

\\\hline
\nom\label{IIA:normalD5}
&$2$ & $3$, $5$ &
$
\xymatrix@R=7pt@C=10pt{
\bullet\ar@{-}[d]&\circ\ar@{-}[d]&\circ\ar@{-}[d]&\circ\ar@{-}@/^1pt/[dl]
\\
\circ\ar@{-}[r]&\circ\ar@{-}[r]
&\underset{4}\circ\ar@{-}[r]&\circ\ar@{-}[r] &\circ 
}
$&\type{D_{5}}&\type{d}

\\\hline
\nom\label{IIA:normal:cb}
& $2$ & $4$, $5$&
\multicolumn{1}{l|}{
$
\xymatrix@R=3pt@C=13pt{
\bullet\ar@{-}[r]&\circ\ar@{-}[d]&&\circ\ar@{-}[d]&\circ\ar@{-}[d]
\\
\circ\ar@{-}[r]&\circ\ar@{-}[r]
&\underset3\circ\ar@{-}[r]&\circ\ar@{-}[r]&\underset3\circ\ar@{-}[r]&\circ
}
$}&&\type{c}

\\\hline
\multicolumn{5}{c}{Cases: $H$ is not normal}
\\\hline
\nom\label{IIA-n-normal-div}
\footnote{This case was erroneously
omitted in \cite[Th. 3.6 and Cor. 3.8]{Tzi:05D}.}
&& &
$
\xymatrix@R=7pt@C=13pt{
&\circ\ar@{-}[d]
\\
\underset {} \bullet \ar@{-}[r]
&\underset 3\circ\ar@{-}[r]&\circ\ar@{-}[r]&\circ
\\
&\circ\ar@{-}[u]
}
$&\type{D_{5}}&\type{d}

\\\hline
\nom\label{IIA:n-normal-cb}
&& &
\multicolumn{1}{l|}{$
\xymatrix@R=7pt@C=11pt{
&\circ\ar@{-}[r]&\overset {3}\circ\ar@{-}[d]\ar@{-}[r]&\circ
\\
\bullet \ar@{-}[r] &\underset {}\circ\ar@{-}[r]&\circ\ar@{-}[r]&\circ
}
$}&&\type{c}
\end{longtable}}
The variety $X$ can have \textup(at most one\textup) extra 
type \typec{III} singular point in all cases except for 
\xref{IIA:flip:iP=1a},
\xref{IIA:flip:iP=2},
\xref{IIA:normalD5}, and
\xref{IIA:normal:cb} where the singular point is unique. 
\end{theorem}

Examples of flipping contractions can be constructed similar to~\xref{ex:cD/3:flip}.

\begin{example}[{\cite[7.6.4]{KM92}}]
\label{ex:IIA:flip}
Let $V\supset C$ be a germ of a smooth threefold along $C\simeq \PP^1$
such that $\NNN_{C/V}\simeq \OOO_C\oplus\OOO_C$. 
Pick a point $P\in C$ and let $(v_1, v_2, v_3)$ be coordinates at
$(V, P)$ such that $(C, P) = \{\text{$v_1$-axis}\}$. 
Let $(X,C, P)$ be a \typec{IIA}-point as in
\eqref{equation-XC}-\eqref{equation-alpha} with $\alpha\equiv y_1y_2 \mod (y_2, y_3, y_4)^2$.
For suitable $\varepsilon_1$ and $\varepsilon_2$ such that $0<\varepsilon_1<\varepsilon_1\ll 1$, 
$(y_1^4, y_1^2y_4, y_1y_3)$ form 
coordinates for $U = (X, P)\cap \{\varepsilon_1 < | y_1^4| < \varepsilon_2\}$ by the implicit function theorem. 
Thus
$v_1 = y_1^4$, $v_2 =y_1^2 y_4$, and $v_3 =y_1y_3$
patch $(X, P)$ and $V\setminus (V, P) \cap \{|v_1| < \varepsilon_1\}$ along $U$. 
By \cite[7.2.4]{KM92} the germ $(X,C)$ is a flipping
curve germ of type \typec{IIA} as in \xref{IIA:flip:iP=1a}.
See \cite[7.9.4, 7.12.5]{KM92} for more examples of flipping contractions.
\end{example}

The existence in the above theorem in the case where $H$ is normal can be established 
by using arguments of \xref{def:existence}.
In the case \xref{IIA:normalD5} we have also explicit example:

\begin{example}[{\cite[6.6]{MP:IIA-1}}]
\label{ex:IIA-normal}
Let $Z \subset\CC^5_{z_1,\ldots,z_5}$ be defined by two equations:
\begin{eqnarray*}
0&=& z_2^2+z_3+z_4z_5^k+z_1^3,\qquad k\ge 1,\\
0&=& z_1^2z_2^2 + z_4^2-z_3z_5+z_1^3z_2+cz_1^2z_4.
\end{eqnarray*}
By eliminating $z_3$ using the first equation,
one sees easily that $(Z,0)$ is a
threefold singularity of type~\type{cD_{5}}.
Let $B \subset Z$ be the $z_5$-axis, and
let $f: X \to Z$ be the weighted blowup of $B$ with
weight $(1,1,4,2,0)$. 
By an easy computation one sees that
$C:=f^{-1}(0)_{\red} \simeq {\PP}^1$
and $X$ is covered by two charts: $z_1$-chart
and $z_3$-chart. The origin
of the $z_3$-chart
is a type~\typec{IIA} point $P$ with $\ell(P)=3$:
\begin{equation*}
\{y_1^3y_3+y_2^2+y_3^2+y_4(y_1^2y_2^2+y_4^2+y_1^3y_2+c y_1^2 y_4)^k=0\}/\mumu_{4}(1,1,3,2),
\end{equation*}
where $(C,P)$ is the $y_1$-axis.
Moreover, $X$ is smooth outside $P$.
Thus $X\to Z$ is a divisorial contraction of type~\xref{IIA:normalD5}. 
See also {\cite[8.3.3]{MP:IIA-1}} for an example with $\ell(P)=5$.
\end{example}

The case \xref{IIA:normalD5} was also studied by N.~Tziolas \cite[Th. 3.6]{Tzi:05D}.

The existence of \xref{IIA-n-normal-div} can be shown similar to Example~\xref{ex:IIA-normal}: 

\begin{example}[{\cite[3.6]{MP:IIA-2}}]
\label{ex:IIA-n-normal}
Let $Z \subset {\CC}^5_{z_1,\ldots,z_5}$ be
defined by
\begin{eqnarray*}
0&=& z_2^2+z_3+z_4z_5^k-z_1^3,\qquad k\ge 1,\\
0&=& z_1^2z_2^2+z_4^2-z_3z_5.
\end{eqnarray*}
Then $(Z,0)$ is a
threefold singularity of type~\type{cD_{5}}.
Let $B \subset Z$ be the $z_5$-axis and
let $f: X \to Z$ be the weighted $(1,1,4,2,0)$-blowup.
The origin
of the $z_3$-chart
is a type~\typec{IIA} point $P$ with $\ell(P)=3$:
\begin{equation*}
\{-y_1^3y_3+y_2^2+y_3^2+y_4(y_1^2y_2^2+y_4^2)^k=0\}/\mumu_{4}(1,1,3,2),
\end{equation*}
where $(C,P)$ is the $y_1$-axis.
In the $z_1$-chart we have a type~\typec{III} point.
See also \cite[3.7]{MP:IIA-2} for an example of 
a divisorial germ as in \xref{IIA-n-normal-div} whose singular locus 
consists of a single \typec{IIA} point $P$ with $\ell(P)=7$.
\end{example}

\begin{example}[{\cite[4.8]{MP:IIA-2}}]
\label{example-conic-bundle-lP=4+III}
Let $X$ be the the hypersurface of weighted degree $10$ in 
the weighted projective space $\PP(1,1,3,2,4)_{x_1,\dots, x_4, w}$ 
given by the equation 
\begin{equation*}w\phi_6 -x_1^6\phi_4=0,\quad \text{where}\quad
\begin{array}{lll}
\phi_6&:=&x_1^4x_4+x_3^2+x_2^2w+\delta x_4^3,
\\
\phi_4&:=&x_4^2+\nu x_2x_3+\eta x_1^2x_4+\mu x_1^3x_2
\end{array}
\end{equation*}
(for simplicity we assume that the coefficients $\delta$, $\nu$, $\eta$ are 
general).
Regard $X$ as a small analytic neighborhood of $C$.
In the affine chart $U_w:=\{w\neq 0\}\simeq \CC^4/\mumu_{4}(1,1,3,2)$
the variety $X$ is given by 
\begin{equation*}
\phi_6(y_1,y_2,y_3,y_4, 1) - y_1^6\phi_4(y_1,y_2,y_3,y_4, 1)=0
\end{equation*}
and $C$ is the $y_1$-axis.
Clearly, it has the form \eqref{equation-alpha}.
So, the origin $P\in (X,C)$ is a type~\typec{IIA} point with $\ell(P)=4$. 

In the affine chart $U_1:=\{x_1\neq 0\}\simeq \CC^4$
the variety $X$ is defined by 
\begin{equation*}
w\phi_6(1,z_2,z_3,z_4, w) - \phi_4(1,z_2,z_3,z_4, w)=0.
\end{equation*}
If $\mu\neq 0$, then $X$ is smooth outside $P$, i.e. $(X,C)$ is as 
in the case \xref{IIA:normal:cb}. If $\mu=0$, then 
$(X,C)$ has a type~\typec{III} point at $(0,0,0,\eta)$.

We claim that $(X,C)$ admits a structure of a $\QQ$-conic bundle germ 
as in \xref{IIA:n-normal-cb} (resp. \xref{IIA:normal:cb}) if 
$\mu=0$ (resp. $\mu\neq 0$).
\end{example}
\begin{proof}
Consider the surface $H=\{\phi_6=\phi_4=0\}\subset X$.
Let $\psi: H^{\n}\to H$ be the normalization (we put $H^{\n}=H$ if $H$ is normal)
and let $C^{\n}:=\psi^{-1}(C)$. 
One can explicitly check that $H$ is normal and smooth outside $P$ if $\mu \neq 0$
and $H$ is singular along $C$, the curve
$C^{\n}$ is irreducible and rational, and $\psi_C:= C^{\n}\to C$ is a double cover
if $\mu =0$.
Moreover, the singularities of $H^{\n}$ are rational.
Note that $H$ is a fiber of the fibration $\pi: X\to D$ over a small disk around the origin 
given by the rational function $ \phi_4/w =\phi_6/x_1^6$
which is regular in a neighborhood of $C$.
Analyzing the minimal resolution
one can show that 
there exists a rational curve fibration $f_H: H\to B$, where $B\subset \CC$
is a small disk around the origin, such that $C=f_H^{-1}(0)_{\red}$.
Now the existence of a contraction is a consequence of the following.
\end{proof}

\begin{sclaim} 
\label{IIA:claim:exist}
\begin{enumerate}
\item \label{IIA:claim:exist1}
$H^1(\hat X,\OOO_{\hat X})=0$,
where $\hat X$ denotes the completion of $X$ along $C$.
\item \label{IIA:claim:exist2}
The contraction \mbox{$f_H: H\to B$} extends to a contraction \mbox{$\hat f: \hat X\to \hat Z$}.
\item \label{IIA:claim:exist3}
There exists a contraction \mbox{$f:X\to Z$} that approximates \mbox{$\hat f: \hat X\to \hat Z$}.
\end{enumerate}
\end{sclaim}
\begin{proof}
For \xref{IIA:claim:exist1} we refer to \cite[4.8.4]{MP:IIA-2}.

\xref{IIA:claim:exist2}
Since $H^1(\OOO_{\hat X})=0$, from the exact sequence
\begin{equation*}
0 \xrightarrow{\hspace*{20pt}} \OOO_X \xrightarrow{\hspace*{20pt}} \OOO_X (H) \xrightarrow{\hspace*{20pt}} \OOO_H (H)\xrightarrow{\hspace*{20pt}} 0 
\end{equation*}
we see that the map $H^0(\OOO_{\hat X} (\hat H))\to H^0(\OOO_{\hat H} (\hat H))$
is surjective.
Hence there exists a divisor $\hat H_1\in |\OOO_{\hat X}|$ such that $\hat H_1|_{\hat H}=\hat \CCC$. Then the divisors 
$\hat H$ and $\hat H_1$ define a contraction $\hat f: \hat X\to \hat Z$.

\xref{IIA:claim:exist3}
Let $F$ be the scheme fiber of $f_H: H\to B$ over the origin.
The above arguments shows that the deformations of $F$ are unobstructed.
Therefore the corresponding component of the Douady space is smooth and two-dimensional.
This allow us to produce a contraction on $X$.
\end{proof}

\begin{example}[{\cite[4.9]{MP:IIA-2}}]
\label{example-conic-bundle-lP=8}
Similarly to Example \xref{example-conic-bundle-lP=4+III}, let
$X\subset \PP(1,1,3,2,4)$ be a small analytic neighborhood of $C=
\{\text{$(x_1,w)$-line}\}$ given by the equation
$x_1^6\phi_4-w \phi_6 =0$, where 
\begin{eqnarray*}
\phi_6&:=&x_3^2+x_2^2w+\delta x_4^3+cx_1^2x_4^2,
\\
\phi_4&:=&x_4^2+\nu x_2x_3+\eta x_1^2x_4.
\end{eqnarray*}
It is easy to check that $P:=(0:0:0:0:1)$ is the only 
singular point of $X$ on $C$ and it is a type~\typec{IIA} point with $\ell(P)=8$.
The rational function $\phi_4/w=\phi_6/x_1^6$ near $C$ defines a fibration
whose central fiber $H$ is given by $\phi_4=\phi_6 =0$ such that $\Delta(H,C)$ is
of type \xref{IIA:n-normal-cb}.
The existence of a contraction
$f: X\to Z$ can be shown similar to Claim \xref{IIA:claim:exist}.
\end{example}

\begin{example}[{\cite[4.9.1]{MP:IIA-2}}]
\label{example-conic-bundle-normal-H}
In a similar way we can construct an example of a $\QQ$-conic bundle with $\ell(P)=5$
and normal $H$ as in \xref{IIA:normal:cb}.
Consider $X\subset \PP(1,1,3,2,4)$ given by $w\phi_6-x_1^6\phi_4=0$, where
\begin{eqnarray*}
\phi_6&:=&x_1^5 x_2+x_2^2w+x_3^2+\delta x_4^3+cx_1 ^2 x_4^2
\end{eqnarray*}
and $\phi_4$ is as in \xref{example-conic-bundle-lP=4+III}.
In the affine chart $U_w\simeq \CC^4/\mumu_{4}(1,1,3,2)$
the origin $P\in (X,C)$ is a type~\typec{IIA} point with $\ell(P)=5$. 
It is easy to see that $X$ is smooth outside $P$.
The rational function $\phi_4/w=\phi_6/x_1^6$ defines 
a fibration on $X$ near $C$ with central fiber $H=\{\phi_4=\phi_6=0\}$.
\end{example}

\appendix
\renewcommand{\theequation}{\Alph{section}.\arabic{subsection}.\arabic{equation}}
\renewcommand{\thesubsection}{\Alph{section}.\arabic{subsection}}
\section{A remark on divisorial contractions}

\begin{proposition}\label{prop:mult}
Let $(X,C\simeq \PP^1)$ be a divisorial curve germ 
with one non-Gorenstein point which is not of type~\type{cA/m} with $m>2$
and let
$f : (X, C) \to (Z,o)$ be corresponding contraction.
Let $E\subset X$ be the exceptional divisor and let $B:=f(E)$ be the 
blowup curve. Then the multiplicity $\mult_{o}(B)$ is given by the following table.
{\rm
\begin{center}
\begin{tabularx}{1\textwidth}{lXcll}
$(X,C)$ & $\Delta(H,C)$ &\mbox{$\mult_{o}(B)$}&$H_Z$&$D_Z$
\\
\hline
\\[-9pt]
\typec{IIA}&\xref{IIA:normalA1:a}, \xref{IIA:normalA1:b} & 3
&\type{A_1}&\type{D_{2n+1}}
\\
\typec{IIA}&\xref{IIA:normalD5}, \xref{IIA-n-normal-div} &1
&\type{D_5}&\type{D_{2n+1}}
\\
\type{(IIB)}& \xref{IIB:thm-A2case-simple}, \xref{IIB:thm-D4case-double} & 2
&\type{A_2}, \type{D_4}&\type{E_6}
\\
\type{(IIB)}& \xref{IIB:thm-smooth-case-simple} & 5
&\type{A_0}&\type{E_6}
\\
\type{cD/3}& \xref{cD/3:thm:A2}, \xref{cD/3:thm:D4} & 2
&\type{A_2}, \type{D_4}&\type{E_6}
\\
\type{cD/3}& \xref{cD/3:thm:E6} & 1
&\type{E_6}&\type{E_6}
\\
\type{cA/2}&\xref{KM:4.7.3.1.1}& $n$&\type{A_1}&\type{A}
\\
\type{cA/2}&\xref{KM:4.7.3.1.2}& 3&\type{A_0}&\type{A}
\\
\type{cA/2}& \xref{KM:4.7.3.1.3} & 1&\type{A_2}&\type{A}
\\
\type{cA/2}& \xref{KM:4.7.3.1.4}& 4&\type{A_0}&\type{A}
\\
\type{cAx/2}, \type{cD/2}&\xref{KM:4.7.4}-\xref{KM:4.7.5} &1&\type{D}&\type{D_{}}
\\
\type{cE/2}& \xref{KM:4.7.4}-\xref{KM:4.7.6} &1&\type{D}, \type{E_6}&\type{E_7}
\end{tabularx}
\end{center}}
\noindent
where $H_Z$ is a general hyperplane section of $(Z,o)$
and $D_Z$ is a general hyperplane section of $(Z,o)$ passing through $B$.
In the \type{cA/2}-case the meaning of $n$ is the same as in \xref{KM:4.7.3.1.1}.
\end{proposition}

The cases with $\mult_{o}(B)=1$, i.e. those with smooth $B$, were studied 
in details by N.~Tziolas 
\cite{Tzi:03}, \cite{Tzi:05}, \cite{Tzi:05D}, \cite{Tzi:10}.

\begin{proof}
Recall that $Z$ is $\QQ$-Gorenstein and $E$ is $\QQ$-Cartier divisor
(Theorem~\xref{thm:div:Q-Cartier}).
By classification in all our cases $Z$ is in fact Gorenstein
(that is, $H_Z$ has at worst Du Val singularity).
Hence, $E\in |K_X|$.
Let $H:=f^*(H_Z)$. Let $D\in |-K_X|$ be a general member.
We have 
$-K_X\cdot C=1/m$, 
where
$m$ is the index of the non-Gorenstein point (see \xref{lemma:KC}).
For simplicity assume that $H$ is normal.
The case \xref{IIA-n-normal-div} can be treated in a similar way.
\begin{lemma}[{\cite[Lemma~5.1]{Tzi:05}}]
If in the above notation $H$ is normal, then
\begin{equation}
\mult_o(B)=-\frac{(K_C\cdot C)^2}{(C^2)_H}=-\frac{1}{m^2(C^2)_H}.
\end{equation} 
\end{lemma} 

Now let $\psi: \hat H\to H$ be the minimal resolution.
Write $\psi^* C= \hat C+ \Theta$,
where $\Supp(\Theta)\subset \Exc(\psi)$ and $\Theta=\sum \theta_i\Theta_i$.
Since $\hat C^2=-1$, we have
\begin{equation}
\label{equation-proposition-C2}
C^2= -1+\hat C\cdot \Theta=-1+\sum\nolimits' \theta_i,
\end{equation}
where $\sum'$ runs through all the components $\Theta_i$ meeting $\hat C$.
The coefficients $\theta_i$ are computed from the standard system
of linear equations:
\begin{equation*}
0=-\Theta_j\cdot\psi^* C=\Theta_j\cdot \hat C+\sum_i \theta_i\Theta_j\cdot \Theta_i.
\end{equation*}
Now $\mult_{o}(B)$ can be computed by using \eqref{equation-proposition-C2}.
Consider for example the cases 
\xref{IIA:normalA1:a}, \xref{IIA:normalA1:b} and \xref{IIA:normalD5} of Theorem \xref{IIA:thm}
(other cases are similar).
In the graphs below
we attach the coefficients $\theta_i$ of $\Theta=\psi^*C-\hat C$
to the corresponding vertices and indicate the value of $C^2$. This immediately
gives us the values of $\mult_o(B)$ as desired.\qedhere

\begin{align*}\text{\xref{IIA:normalA1:a}}&
\xymatrix@R=0pt@C=17pt{
&&&& &\overset{2/16}\circ\ar@{-}[d]
\\
&\underset{1/3}\circ\ar@{-}[r]&\underset{2/3}\circ\ar@{-}[r]&\underset{1}\bullet\ar@{-}[r]&
\underset{5/16}\circ\ar@{-}[r]
&\underset{4/16}\circ\ar@{-}[r]&\underset{1/16}\circ&
}
&
\scriptstyle C^2=-1/48
\\
\text{{\xref{IIA:normalA1:b}}}&
\xymatrix@R=0pt@C=17pt{
&&&& &\overset{7/48}{\circ}\ar@{-}[d]
\\
&&\underset{1/2}\circ\ar@{-}[r]&\underset{1}\bullet\ar@{-}[r]&\underset{23/48}\circ\ar@{-}[r]
&\underset{7/16}\circ\ar@{-}[r]&\underset{1/4}\circ\ar@{-}[r]&\underset{1/16}\circ
}
&
\scriptstyle C^2=-1/48
\\
\text{\xref{IIA:normalD5}}&
\vcenter{
\xymatrix@R=0pt@C=17pt{
&&&\overset{7/16}\circ\ar@{-}[d]&\overset{3/16}\circ\ar@{-}[d]&
\overset{3/16}\circ\ar@{-}@/^7pt/[dl]
\\
&\underset 1 \bullet\ar@{-}[r]&\underset{15/16}\circ\ar@{-}[r]&\underset{7/8}\circ\ar@{-}[r]
&\underset{3/8}\circ\ar@{-}[r]&\underset{1/4}\circ\ar@{-}[r]&\underset{1/8}\circ
}}
&\scriptstyle C^2=-1/16
\end{align*}     
\end{proof}


\def\cprime{$'$}

\end{document}